\documentclass[preprint,12pt]{elsarticle}
\pdfoutput=1
\journal{Journal of Symbolic Computation}
\usepackage[utf8]{inputenc}
\usepackage[T1]{fontenc}
\usepackage{amsmath}
\usepackage{amsthm,amssymb}
\usepackage{xcolor}
\usepackage{rotating}
\usepackage{colortbl,tabularx}
\usepackage{multirow}
\usepackage{bbold} %
\usepackage{booktabs}
\usepackage{thm-restate,apptools}
\makeatletter
\def\mathcolor#1#{\@mathcolor{#1}}
\def\@mathcolor#1#2#3{%
\protect\leavevmode
\begingroup
\color#1{#2}#3%
\endgroup
}
\makeatother
\usepackage[backgroundcolor=yellow,colorinlistoftodos]{todonotes}

\newcommand{\lfrac}[2]{\ensuremath{{#1}/{#2}}}
\newcommand{\mathsc}[1]{\ensuremath{\text{\normalfont\scshape{#1}}}}
\newcommand{\Field}{\ensuremath{\mathbb{K}}}
\DeclareMathOperator{\tensorproduct}{\ensuremath{\otimes}}
\newcommand{\Inverse}[1]{\ensuremath{{#1}^{-1}}}
\newcommand{\RR}{\ensuremath{\mathbb{R}}}

\newcommand{\BDual}[1]{\ensuremath{{{\bigl({#1}\bigr)}^{\star}}}}

\newcommand{\mat}[1]{\ensuremath{\mathsf{#1}}}%
\newcommand{\matrixsize}[2]{\ensuremath{{{#1}{\times}{#2}}}}
\newcommand{\MatrixProduct}[2]{\ensuremath{{{\mat{#1}}\cdot{\mat{#2}}}}}
\newcommand{\Matrices}[2]{\ensuremath{{{#1}^{#2}}}}
\newcommand{\MatrixRank}[1]{\ensuremath{{\textup{Rank}\,{#1}}}}
\DeclareMathOperator{\Trace}{\ensuremath{\textup{Trace}}}
\newcommand{\Transpose}[1]{\ensuremath{{{#1}}^{\intercal}}}

\newcommand{\vectorization}[1]{\ensuremath{\textup{Vect}{(#1)}}}%
\newcommand{\vectorif}[2]{\ensuremath{{\left(#1\right)}_{#2}}}
\newcommand{\HadamardProduct}[2]{\ensuremath{{#1}\odot{#2}}}
\newcommand{\InvTranspose}[1]{{{\mat{#1}}^{-\intercal}}}
\newenvironment{smatrix}{\left[\begin{smallmatrix}}{\end{smallmatrix}\right]}
\def\triadone{brown}
\def\triadtwo{red}
\def\triadthree{blue}
\def\triadfour{green}
\def\triadfive{magenta}
\def\triadsix{gray}
\def\triadseven{violet}
\newcommand{\FMMA}[4]{\ensuremath{{\langle{{#1}{\times}{#2}{\times}{#3}{:}{#4}}\rangle}}}
\newcommand{\firstdim}{\ensuremath{m}}
\newcommand{\seconddim}{\ensuremath{k}}
\newcommand{\thirddim}{\ensuremath{n}}
\newcommand{\tensor}[1]{\ensuremath{\mathcal{#1}}}
\newcommand{\tpow}[2]{\ensuremath{#1^{\otimes^{#2}}}} %
\newcommand{\Contraction}[3]{\ensuremath{{{#1}\!\mid^{#2}_{#3}}}}

\newcommand{\LeftTensor}[1]{\ensuremath{{\mat{M}}_{#1}}}
\newcommand{\RightTensor}[1]{\ensuremath{{\mat{N}}_{#1}}}
\newcommand{\ProductTensor}[1]{\ensuremath{{\mat{O}}_{#1}}}
\newcommand{\HMRepresentation}[3]{{\ensuremath{\left[{{#1};{#2};{#3}}\right]}}}
\DeclareMathOperator{\IsotropyComposition}{\circ}
\newcommand{\Isotropy}[1]{\ensuremath{\mathsf{#1}}}
\newcommand{\IsotropyAction}[2]{\ensuremath{{#1}\diamond{#2}}}
\newcommand{\growthfactor}{\ensuremath{{\gamma_{2}}}}
\newcommand{\GrowthFactor}[1]{\ensuremath{{\growthfactor{\left(#1\right)}}}}

\newcommand{\GF}[2]{\ensuremath{\gamma_{#1,#2}}} %
\newcommand{\AF}[2]{\ensuremath{\Gamma_{#1,#2}}} %
\newcommand{\EF}[3]{\ensuremath{{f_\mathsc{alg}}_{#1,#2}}^{(#3)}} %

\newcommand{\bigO}[1]{\ensuremath{O\mathopen{(}{#1}\mathclose{)}}}
\newcommand{\bbigO}[1]{\ensuremath{O\bigl({#1}\bigl)}}

\newcommand{\row}[2]{#1_{#2,*}}
\newcommand{\col}[2]{#1_{*,#2}}
\newcommand{\abs}[1]{\left|#1\right|}
\newcommand{\norm}[1]{\left\|#1\right\|}
\newcommand{\pnorm}[1]{\left\|#1\right\|_p}
\newcommand{\psnorm}[1]{\left\|#1\right\|_{p^\star}}
\newcommand{\qnorm}[1]{\left\|#1\right\|_q}
\newcommand{\qsnorm}[1]{\left\|#1\right\|_{q^\star}}

\newcommand{\dualnorm}[1]{\left\|#1\right\|_\star}
\newcommand{\xnorm}[2]{\left\|#1\right\|_{#2}}
\newcommand{\xnormexp}[3]{\left\|#1\right\|_{#2}^{#3}}
\newcommand{\maxnorm}[1]{{\left\|{#1}\right\|}_\infty}

\newcommand{\textbigmaxnorm}[1]{{{\bigl\|{#1}\bigr\|}_{\infty}}}
\newcommand{\onenorm}[1]{\left\|#1\right\|_1}
\newcommand{\twonorm}[1]{\left\|#1\right\|_2}
\newcommand{\Fnorm}[1]{\left\|#1\right\|_F}

\newcommand{\zeronorm}[1]{\left\|#1\right\|_0}
\newcommand{\ulp}{\varepsilon}
\newcommand{\comp}[1]{{\widehat{#1}}}
\newcommand{\matr}[2]{\text{Mat}_{#2}(#1)}
\newcommand{\plinopt}{\href{https://github.com/jgdumas/plinopt}{\textsc{PLinOpt}}}

\usepackage{algorithm}
\usepackage[noend]{algpseudocode}

\algnewcommand{\IfThen}[2]{%
\State{}\algorithmicif\ #1\ \algorithmicthen\ #2}
\algnewcommand{\IfThenEnd}[2]{%
\State{}\algorithmicif\ #1\ \algorithmicthen\ #2\ \algorithmicend\ \algorithmicif}
\algnewcommand{\IfThenElse}[3]{%
\State{}\algorithmicif\ #1\ \algorithmicthen\ #2\ \algorithmicelse\ #3}

\usepackage{hyperref}
\makeatletter
\hypersetup{%
breaklinks=true,
colorlinks=true,
linkcolor=purple,
citecolor=blue,
urlcolor=teal,
hypertexnames=false,
naturalnames=true,
}
\makeatother
\usepackage[nameinlink,capitalize,noabbrev]{cleveref}
\Crefname{proposition}{Proposition}{Propositions}
\crefname{proposition}{Proposition}{Propositions}
\crefname{equation}{Eq.}{Eqs.} %
\Crefname{equation}{Equation}{Equations}
\crefformat{footnote}{#2\footnotemark[#1]#3}
\newcounter{generalthm}

\newtheorem{theorem}[generalthm]{Theorem}

\newtheorem{definition}[generalthm]{Definition}

\newtheorem{lemma}[generalthm]{Lemma}
\newtheorem{notation}[generalthm]{Notation}

\newtheorem{proposition}[generalthm]{Proposition}
\newtheorem{corollary}[generalthm]{Corollary}
\newtheorem{remark}[generalthm]{Remark}

\begin{document}
\begin{frontmatter}
\title{Towards automated generation of fast and accurate algorithms for recursive matrix multiplication}
\author[UGA]{Jean-Guillaume Dumas}
\affiliation[UGA]{%
        organization={Univ.\ Grenoble Alpes, UMR CNRS 5224 LJK},
        addressline={150, place du Torrent, \mbox{IMAG-CS 40700}},
        city={Grenoble CEDEX 9},
        postcode={38058},
        country={France}
}
\author[UGA]{Cl\'ement Pernet}

\author[ULille]{Alexandre Sedoglavic}
\affiliation[ULille]{%
        organization={Université de Lille, UMR CNRS 9189 CRISTAL},
        addressline={Cité scientifique},
        city={Villeneuve d'Ascq},
        postcode={59650},
        country={France}
}

\begin{abstract}
We propose a strategy for the generation of fast and accurate versions
of non-commutative recursive matrix multiplication algorithms.
To generate these algorithms, we consider matrix and tensor norm
bounds governing the stability and accuracy of numerical matrix
multiplication. We start by a unification on known max-norm bounds
on matrix multiplication stability and then extend them to further
norms and more generally to recursive bilinear algorithms and the alternative basis matrix
multiplication algorithms.
Then our strategy has three phases.
First, we reduce those bounds by minimizing a growth factor along the
orbits of the associated matrix multiplication tensor decomposition.
Second, we develop heuristics that minimize the number of operations
required to realize a bilinear formula, while further improving its
accuracy.
Third, we perform an alternative basis sparsification that improves on
the time complexity constant and mostly preserves the overall
accuracy.
For instance this strategy allows us to propose a non-commutative
algorithm for multiplying~\({\matrixsize{2}{2}}\)-matrices using~\(7\)
coefficient products.
This algorithm reaches simultaneously a better accuracy in practice
compared to previously known such fast~\(\FMMA{2}{2}{2}{7}\) Strassen-like
algorithms and a time complexity bound with the best currently known
leading term (obtained via alternative basis sparsification).
We also present detailed results of our technique on other recursive matrix multiplication algorithms, such as Smirnov's~\(\FMMA{3}{3}{6}{40}\) family of algorithms.
\end{abstract}
\end{frontmatter}
\newpage
\tableofcontents
\newpage
\section{Introduction}\label{sec:introduction}
All practicable subcubic matrix multiplication algorithms are built on a recursive application of elementary small matrix product schemes.
Even if these algorithms are numerically stable, they bear more numerical errors than the classical one.
\par
In this paper, we show how to construct new elementary base schemes maintaining the best know complexity exponent for the chosen dimension while improving on the numerical accuracy and the leading complexity constant.
\par
To do so, our approach originates from a unification of the previous forward error analysis and
their generalization, where a \emph{growth factor} is identified as a key parameter governing stability.
We then produce \emph{variants} of the considered elementary base matrix products algorithms improving this growth factor.
This is done using the geometry of the tensor representation of these products.
\par
Once found a suitable tensor representation variant, we show how to derive efficient straight line programs using either the \emph{alternative basis} approach~\cite{Karstadt:2017aa}, an automated common sub-expression elimination process~\cite{jgd:2024:plinopt} or a combination of both.
\par
We apply this approach on the \(\FMMA{2}{2}{2}{7}\)-Strassen and~\(\FMMA{3}{3}{6}{40}\)-Smirnov base schemes hence producing algorithms that demonstrate experimentally the best accuracy among all considered variants of these schemes, while enjoying the best known asymptotic complexity exponents and leading constants.
\par
The first error analysis of a sub-cubic matrix multiplication algorithm was proposed by
Brent~\cite{brent:1970a} for Strassen's algorithm, showing that the max-norm of the error is bounded
by the max-norm of the input matrices, multiplied by the machine precision and an error function~\({f_\text{ALG}(n)=\bbigO{n^{\log_{2}\gamma}}}\).
We call \(\gamma\) the \emph{growth factor}, it depends on the algorithm and its value is~\({\gamma=12}\) for Strassen's algorithm.
Bini and Lotti~\cite{bini:1980} then generalized the analysis to any fast matrix multiplication, and explored for the first time variants of Strassen's algorithm using isotropies with coefficients restricted to powers of~\(2\).
They show that in this context the exponent of Strassen's algorithm is minimal.
Demmel et al.~\cite{demmel:2007a} generalize Bini and Lotti's results to a broader range of algorithms, including the group theoretic ones, showing that all fast matrix multiplication algorithms are norm-wise stable. Then Ballard et al.~\cite{BBDLS16}  generalize these bounds, introducing the contribution of a base case algorithm and proves practical bounds for algorithms based on a series of known elementary base matrix products.
Lastly, Schwartz et al.~\cite{Schwartz:2024aa} extend these analyses to algorithms based on the
alternative basis approach~\cite{Karstadt:2017aa,Beniamini:2019aa,BCHKS:2020}.
Note that this series of contributions, all involve an additional logarithmic factor, compared to Brent's and its generalization by  Higham~\cite{Higham:2002}.
Yet, the growth factor~\(\gamma\) always stands as the key parameter governing the stability of these algorithms.
All these norm-wise bounds use the max-norm, in which the value~\({\gamma=12}\) for Strassen algorithm
could not be improved. Following Dai and Lim~\cite{Dai:2023aa} who proposed a similar analysis
of one recursive level of Strassen's algorithm in Frobenius norm (or 2-norm), we propose here a unified
accuracy analysis of recursive algorithms with base case, using any independent choice of norms for
the left and right hand sides of the bound.  It generalizes some and improves on other
state-of-the-art
bounds~\cite{brent:1970a,bini:1980,demmel:2007a,ballard:2012a,BBDLS16,Dai:2023aa,Higham:2002,Schwartz:2024aa}.
This relaxation on the norms, makes  Strassen algorithm's growth factor non longer minimal, which
allows us to design an optimization program to produce a variant with improved accuracy, and prove
that it lies within \(2.6\%\) of the optimal. Experiments demonstrate that these variants are also
more accurate in the max-norm.

\par
Our optimization process is based on the remark that fast matrix multiplication tensor representation present some \emph{isotropies}, each of them allowing to define a new variant of the algorithm.
The complexity \emph{exponent} is invariant along orbits defined by these symmetries while the considered growth factor is not.
This made possible to elaborate a heuristic mixing symbolic and numerical computations that construct  variants of the considered base algorithms with reduced growth factor and thus with theoretically better accuracy.
Those algorithms sometimes require slightly more operations, thus worsening the \emph{constant} factor of the leading complexity term.
\par
The most efficient variants are then obtained from these bilinear
formulas by minimizing the number of operations required to realize
them. We present another heuristics for this; it makes use of common
sub-expression eliminations with rational coefficients, of potential
factorization via the kernel of the matrices defining the considered
bilinear maps, as well as via Tellegen's transposition principle.
\par
We therefore finally propose further variants obtained by an
alternative basis sparsification, similar to those introduced
in~\cite{Karstadt:2017aa,Beniamini:2019aa}. In fine, again thanks to a
minimization of the number of operations required to realize them, we
obtain variants having a time complexity bound with the best currently
known leading term, that simultaneously improve on the accuracy (i.e.\
mostly preserving in practice the numerical accuracy with or without
alternative basis sparsification).
\par
Our \textsc{c++} tools for the minimization of the number of
operations are gathered in the
\plinopt~library~\cite{jgd:2024:plinopt}.
We also forked the Matlab framework of~\cite{Dai:2023aa}
in~\cite{jgd:2024:mFMM} to benchmark our implementations of the
resulting fast and accurate matrix multiplication algorithms.
We show in the following that our results are also
comforted by numerical experiments.

\paragraph{Additions to the ISSAC 2024 proceedings papers}
This paper extends the results
of~\cite{jgd:2024:accurate} as follows:
\begin{itemize}
\item We provide a unified approach to the study of the stability of
  bilinear algorithms;
\item We provide accuracy bounds on bilinear algorithms and matrix
  multiplications, including alternative basis versions, supporting any
  combination of  $p$-norms and illustrated on all combinations of \(2\) and \(\infty\) norms;
\item We improved the search for cancellations in the generation of
  straight-line programs and compare our approach to state of the art
  common sub-expression elimination;
\item We apply our complete optimization technique to any combination of Smirnov's
  \(\FMMA{3}{3}{6}{40}\), \(\FMMA{3}{6}{3}{40}\) and  \(\FMMA{6}{3}{3}{40}\) algorithms.
\end{itemize}

\paragraph{Paper organization}
\Cref{sec:GeneralFramework} presents the three complementary viewpoints on matrix product algorithm representations that will be used throughout the paper.
In~\Cref{sec:accuracyTheoriticalBound} we propose the unified error bounds on bilinear maps and matrix multiplication algorithms, highlighting how the growth factor parameter governs accuracy and extend this analysis to matrix product algorithms based on alternative basis.
On a relaxed growth factor in Frobenius norm, we present a heuristic, in~\Cref{sec:numericalStabilityMeasure}, using a descent algorithm to reach some local minima.
In the \(\FMMA{2}{2}{2}{7}\) case, we show that it lies within at most~\(2.6\%\) of the optimal.
Finally, \Cref{sec:implem} presents several automated common
sub-expression eliminations techniques  and  minimization heuristics
to produce optimized straight line programs from a tensor
representation of a matrix multiplication algorithm.
We conclude by the presentation of accuracy benchmarks, on the
\(\FMMA{2}{2}{2}{7}\) case and several cases for the
\(\FMMA{3}{6}{3}{40}\) family,
demonstrating the efficiency and effectiveness of our approach.

\section{Some matrix product representations}\label{sec:GeneralFramework}
Throughout this paper, we will use several complementary viewpoints on matrix multiplication algorithms: straight line programs (\textsc{slp}), tensor decompositions and a matrix representation of the corresponding bilinear map.
For a more detailed discussion of this framework, we refer the reader to~\cite{Landsberg:2016ab}.
We expose each of them following the illustrative example of Strassen's example.
\paragraph{Illustrating example}
To start, we briefly review the tensorial representation of bilinear maps, using the example of Strassen's fast~\({\matrixsize{2}{2}}\)-matrix product introduced in~\cite{strassen:1969}.
\par
The product~\({\mat{C}=\MatrixProduct{A}{B}}\) of~\({\matrixsize{2}{2}}\) matrices can be computed by Strassen's algorithm using the following computations:
\begin{equation}\label{eq:StrassenMultiplicationAlgorithm}
\begin{array}{ll}
\mathcolor{\triadone}{\rho_{1}}\leftarrow{\mathcolor{\triadone}{a_{11}}(\mathcolor{\triadone}{b_{12}-b_{22}})},
&
\mathcolor{\triadfour}{\rho_{4}}\leftarrow{(\mathcolor{\triadfour}{a_{12}-a_{22}})(\mathcolor{\triadfour}{b_{21}+b_{22}})},
\\
\mathcolor{\triadtwo}{\rho_{2}}\leftarrow{(\mathcolor{\triadtwo}{a_{11}+a_{12}})\mathcolor{\triadtwo}{b_{22}}},
&
\mathcolor{\triadfive}{\rho_{5}}\leftarrow{(\mathcolor{\triadfive}{a_{11}+a_{22}})(\mathcolor{\triadfive}{b_{11}+b_{22}})},
\\
\mathcolor{\triadthree}{\rho_{3}}\leftarrow{(\mathcolor{\triadthree}{a_{21}+a_{22}}) \mathcolor{\triadthree}{b_{11}}},
&
\mathcolor{\triadseven}{\rho_{7}}\leftarrow{(\mathcolor{\triadseven}{a_{21}-a_{11}})(\mathcolor{\triadseven}{b_{11}+b_{12}})},
\\
\mathcolor{\triadsix}{\rho_{6}}\leftarrow{\mathcolor{\triadsix}{a_{22}}(\mathcolor{\triadsix}{b_{21}-b_{11}})},
&
 \begin{smatrix} c_{11} &c_{12} \\ c_{21} &c_{22} \end{smatrix}
 \!=\!
 \begin{smatrix}
 \mathcolor{\triadfive}{\rho_{5}} + \mathcolor{\triadfour}{\rho_{4}} - \mathcolor{\triadtwo}{\rho_{2}} + \mathcolor{\triadsix}{\rho_{6}} &
 \mathcolor{\triadsix}{\rho_{6}} + \mathcolor{\triadthree}{\rho_{3}} \\
 \mathcolor{\triadtwo}{\rho_{2}} + \mathcolor{\triadone}{\rho_{1}}&
 \mathcolor{\triadfive}{\rho_{5}} + \mathcolor{\triadseven}{\rho_{7}} + \mathcolor{\triadone}{\rho_{1}}- \mathcolor{\triadthree}{\rho_{3}}
 \end{smatrix}.
\end{array}\hspace{-15pt}
\end{equation}
This straight-line program (a.k.a.\ \textsc{slp}) encodes the following bilinear map over a field~\(\RR\) with~\({\firstdim,\seconddim,\thirddim}\) --- that are all equal to~\(2\) in Strassen's algorithm.
\begin{equation}\label{eq:mxnTimesnxp}
\beta_{\textsc{mm}}(A,B):
\begin{array}[t]{ccl}
\Matrices{\RR}{\matrixsize{\firstdim}{\seconddim}}\times\Matrices{\RR}{\matrixsize{\seconddim}{\thirddim}}&\rightarrow&\Matrices{\RR}{\matrixsize{\firstdim}{\thirddim}},\\
(\mat{A},\mat{B})&\mapsto&\MatrixProduct{A}{B}.
\end{array}
\end{equation}
Indices~\({\firstdim,\seconddim}\) and~\(\thirddim\) are used in this section for clarity, helping
to distinguish easily the different spaces involved in the following sections such
as~\(\RR^{\matrixsize{m}{k}}\) the space of~\(\matrixsize{m}{k}\) matrices with coefficients in~\(\RR\).
\subsection{Matrix multiplication encoded by tensors decomposition}\label{sec:tensordecomp}
Following~\cite{Landsberg:2016ab}, we begin by expressing the bilinear map defining matrix product into the formalism of tensor decomposition, that provides a clear way to represent algorithms associated to such product and its associated symmetries.
\par
Strassen's algorithm~(\ref{eq:StrassenMultiplicationAlgorithm}) is encoded as the decomposition of the matrix multiplication tensor, expressed as a sum of seven rank-one tensors~\(\tensor{S}\) defined by the following relations:
\begin{equation}\label{eq:StrassenTensor}
\begin{array}{r}
\tensor{S}=\sum_{i=1}^{7}{\LeftTensor{i}}\!\tensorproduct\!{\RightTensor{i}}\!\tensorproduct\!{\ProductTensor{i}}=
\mathcolor{\triadfive}{{\begin{bmatrix}1&0\\0&1\end{bmatrix}}\!\tensorproduct\!{\begin{bmatrix}1&0\\0&1\end{bmatrix}}\!\tensorproduct\!\begin{bmatrix}1&0\\0&1\\\end{bmatrix}}
\\[\bigskipamount]
+
\mathcolor{\triadfour}{\begin{bmatrix}0&1\\0&-1\\\end{bmatrix}\!\tensorproduct\!\begin{bmatrix}0&0\\1&1\\\end{bmatrix}\!\tensorproduct\!\begin{bmatrix}1&0\\0&0\\\end{bmatrix}}
\!+\!
\mathcolor{\triadseven}{\begin{bmatrix}-1&0\\1&0\\\end{bmatrix}\!\tensorproduct\!\begin{bmatrix}1&1\\0&0\\\end{bmatrix}\!\tensorproduct\!\begin{bmatrix}0&0\\0&1\\\end{bmatrix}}
\\[\bigskipamount]
+
\mathcolor{\triadtwo}{\begin{bmatrix}1&1\\0&0\\\end{bmatrix}\!\tensorproduct\!\begin{bmatrix}0&0\\0&1\\\end{bmatrix}\!\tensorproduct\!\begin{bmatrix}-1&0\\1&0\\\end{bmatrix}}
\!+\!
\mathcolor{\triadone}{\begin{bmatrix}1&0\\0&0\\\end{bmatrix}\!\tensorproduct\!\begin{bmatrix}0&1\\0&-1\\\end{bmatrix}\!\tensorproduct\!\begin{bmatrix}0&0\\1&1\\\end{bmatrix}}
\\[\bigskipamount]
+
\mathcolor{\triadsix}{\begin{bmatrix}0&0\\0&1\\\end{bmatrix}\!\tensorproduct\!\begin{bmatrix}-1&0\\1&0\\\end{bmatrix}\!\tensorproduct\!\begin{bmatrix}1&1\\0&0\\\end{bmatrix}}
\!+\!
\mathcolor{\triadthree}{\begin{bmatrix}0&0\\1&1\\\end{bmatrix}\!\tensorproduct\!\begin{bmatrix}1&0\\0&0\\\end{bmatrix}\!\tensorproduct\!\begin{bmatrix}0&1\\0&-1\end{bmatrix}}
\end{array}
\end{equation}
in~\({\BDual{{\RR}^{\matrixsize{\firstdim}{\seconddim}}}\tensorproduct{\BDual{{\RR}^{\matrixsize{\seconddim}{\thirddim}}}}\tensorproduct{{\RR}^{\matrixsize{\firstdim}{\thirddim}}}}\) (with~\({{\firstdim}={\seconddim}={\thirddim}={2}}\) in~\Cref{eq:StrassenTensor}).
In the above tensor decomposition, each summand is a \emph{rank-one tensor} and its \emph{tensor rank} is the number~\(r\) of such element~(\(7\) there).
\par
Let us recall the following classical standpoint whose associated notations are used in the sequel.
\begin{definition}\label{def:FrobeniusInnerProduct}
The spaces~\({\Matrices{\RR}{\matrixsize{\firstdim}{\thirddim}}}\) can be endowed with the classical \emph{Frobenius inner product}~\({{\langle\mat{M},\mat{N}\rangle}={\Trace({\MatrixProduct{\Transpose{\mat{M}}}{\mat{N}}})}}\) that establishes an isomorphism between the space~\(\Matrices{\RR}{\matrixsize{\firstdim}{\thirddim}}\) and its dual space~\({\BDual{\Matrices{\RR}{\matrixsize{\firstdim}{\thirddim}}}}\).
\end{definition}
Given a tensor decomposition as in~\Cref{eq:StrassenTensor}, the corresponding bilinear map~\Cref{eq:mxnTimesnxp} is obtained using the third 2-contraction of the tensor~\({\tensor{\tensor{S}}\tensorproduct{\mat{A}}\tensorproduct{\mat{B}}}\) as defined in the following map:
\begin{equation}
\label{eq:SecondContraction}
\begin{array}{c}
{\left({\BDual{\Matrices{\RR}{\matrixsize{\firstdim}{\seconddim}}}}\!\tensorproduct{\BDual{\Matrices{\RR}{\matrixsize{\seconddim}{\thirddim}}}}\!\tensorproduct{\Matrices{\RR}{\matrixsize{\firstdim}{\thirddim}}}\right)}
\!\tensorproduct\!
{\left({\Matrices{\RR}{\matrixsize{\firstdim}{\seconddim}}}\!\tensorproduct\!{\Matrices{\RR}{\matrixsize{\seconddim}{\thirddim}}}\right)}
\!\rightarrow\!
    {\Matrices{\RR}{\matrixsize{\firstdim}{\thirddim}}}, \\
    {\left(\sum_{i=1}^{r}{\LeftTensor{i}}\!\tensorproduct\!{\RightTensor{i}}\!\tensorproduct\!{\ProductTensor{i}}\right)}
\tensorproduct
({\mat{A}}\tensorproduct{\mat{B}})\mapsto
\sum_{i=1}^{r}\langle{\LeftTensor{i}},\mat{A}\rangle
\langle{\RightTensor{i}},\mat{B}\rangle
{\ProductTensor{i}}.
\end{array}
\end{equation}
Furthermore, we are going to use in~\Cref{sec:Isotropies} the following relationship between the classical trace of linear operator and tensor representations introduced so far.
The Frobenius inner product of a matrix~\(\mat{O}\) with a matrix product~\(\MatrixProduct{M}{N}\) defines a trilinear form~\({\Trace(\MatrixProduct{\Transpose{O}}{\MatrixProduct{M}{N}})}\) as follows:
\begin{equation}\label{eq:TrilinearForm}
\Contraction{\tensor{S}}{}{3}:
\begin{array}[t]{ccc}
{\RR}^{\matrixsize{\firstdim}{\seconddim}}\times{{\RR}^{\matrixsize{\seconddim}{\thirddim}}}\times{({\RR}^{\matrixsize{\firstdim}{\thirddim}})}^{*}&\rightarrow&{\RR},\\
(\mat{M},\mat{N},\Transpose{\mat{O}})&\mapsto&\langle\mat{O},\MatrixProduct{M}{N}\rangle.
\end{array}
\end{equation}
This viewpoint will be used to harness the symmetries of the tensor decomposition, as presented
in~\Cref{sec:Isotropies,sec:numericalStabilityMeasure} to develop new fast matrix multiplication algorithms with improved accuracy.
\par

\subsection{Hopcroft-Musinski representation}
A compact and elegant representation introduced in~\cite{hopcroft:1973} by Hopcroft and Musinski encodes the sum of rank-one tensors using three matrices, similar to the approach used in the Strassen tensor decomposition~(\ref{eq:StrassenTensor}), represented by the matrices~\({\mat{L}_{\tensor{S}},\mat{R}_{\tensor{S}}}\) and~\(\mat{P}_{\tensor{S}}\):
\begin{equation}\label{eq:StrassenHMRepresentation}
\begin{bmatrix}%
\mathcolor{\triadfive}{1}&\mathcolor{\triadfive}{0}&\mathcolor{\triadfive}{0}& \mathcolor{\triadfive}{1}\\
\mathcolor{\triadfour}{0}&\mathcolor{\triadfour}{1}&\mathcolor{\triadfour}{0}&\mathcolor{\triadfour}{{-1}}\\
\mathcolor{\triadseven}{-1}&\mathcolor{\triadseven}{0}&\mathcolor{\triadseven}{1}& \mathcolor{\triadseven}{0}\\
 \mathcolor{\triadtwo}{1}&\mathcolor{\triadtwo}{1}&\mathcolor{\triadtwo}{0}& \mathcolor{\triadtwo}{0}\\
 \mathcolor{\triadone}{1}&\mathcolor{\triadone}{0}&\mathcolor{\triadone}{0}& \mathcolor{\triadone}{0}\\
 \mathcolor{\triadsix}{0}&\mathcolor{\triadsix}{0}&\mathcolor{\triadsix}{0}& \mathcolor{\triadsix}{1}\\
 \mathcolor{\triadthree}{0}&\mathcolor{\triadthree}{0}&\mathcolor{\triadthree}{1}& \mathcolor{\triadthree}{1}
\end{bmatrix},\quad%
\begin{bmatrix}%
 \mathcolor{\triadfive}{1}&\mathcolor{\triadfive}{0}&\mathcolor{\triadfive}{0}& \mathcolor{\triadfive}{1}\\
 \mathcolor{\triadfour}{0}&\mathcolor{\triadfour}{0}&\mathcolor{\triadfour}{1}& \mathcolor{\triadfour}{1}\\
 \mathcolor{\triadseven}{1}&\mathcolor{\triadseven}{1}&\mathcolor{\triadseven}{0}& \mathcolor{\triadseven}{0}\\
 \mathcolor{\triadtwo}{0}&\mathcolor{\triadtwo}{0}&\mathcolor{\triadtwo}{0}& \mathcolor{\triadtwo}{1}\\
 \mathcolor{\triadone}{0}&\mathcolor{\triadone}{1}&\mathcolor{\triadone}{0}&\mathcolor{\triadone}{-1}\\
\mathcolor{\triadsix}{-1}&\mathcolor{\triadsix}{0}&\mathcolor{\triadsix}{1}& \mathcolor{\triadsix}{0}\\
 \mathcolor{\triadthree}{1}&\mathcolor{\triadthree}{0}&\mathcolor{\triadthree}{0}& \mathcolor{\triadthree}{0}
\end{bmatrix},\quad%
\Transpose{%
\begin{bmatrix}%
 \mathcolor{\triadfive}{1}&\mathcolor{\triadfive}{0}&\mathcolor{\triadfive}{0}& \mathcolor{\triadfive}{1}\\
 \mathcolor{\triadfour}{1}&\mathcolor{\triadfour}{0}&\mathcolor{\triadfour}{0}& \mathcolor{\triadfour}{0}\\
 \mathcolor{\triadseven}{0}&\mathcolor{\triadseven}{0}&\mathcolor{\triadseven}{0}& \mathcolor{\triadseven}{1}\\
\mathcolor{\triadtwo}{-1}&\mathcolor{\triadtwo}{1}&\mathcolor{\triadtwo}{0}& \mathcolor{\triadtwo}{0}\\
 \mathcolor{\triadone}{0}&\mathcolor{\triadone}{1}&\mathcolor{\triadone}{0}& \mathcolor{\triadone}{1}\\
 \mathcolor{\triadsix}{1}&\mathcolor{\triadsix}{0}&\mathcolor{\triadsix}{1}& \mathcolor{\triadsix}{0}\\
 \mathcolor{\triadthree}{0}&\mathcolor{\triadthree}{0}&\mathcolor{\triadthree}{1}&\mathcolor{\triadthree}{-1}
\end{bmatrix}
}.%
\end{equation}
\begin{notation}
Given an~\(\matrixsize{\firstdim}{\seconddim}\)-matrix~\(\mat{A}\), we denote by~\(\row{\mat{A}}{i}\) the~\(i\)th row and by~\(\vectorization{\mat{A}}\) the row-major vectorization of this matrix, i.e.\ the vector~\(v\) in~\(\RR^{\firstdim\seconddim}\) of the matrix coefficients such that~\({v_{i\seconddim+j} = a_{i,j}}\).
We also denote by~\(\matr{v}{\firstdim,\seconddim}\) the reciprocal operation, building an~\(\matrixsize{\firstdim}{\seconddim}\) matrix from an~\(\firstdim\seconddim\)-dimensional vector.
\end{notation}

\begin{definition}\label{def:HMRepresentation}
The \emph{\textsc{hm}} representation of a tensor~\({T=\sum_{i=1}^{r}{\LeftTensor{i}}\tensorproduct{\RightTensor{i}}\tensorproduct{\ProductTensor{i}}}\) in the space~\({{\BDual{\Matrices{\RR}{\matrixsize{\firstdim}{\seconddim}}}}\tensorproduct{\BDual{\Matrices{\RR}{\matrixsize{\seconddim}{\thirddim}}}}\tensorproduct{\Matrices{\RR}{\matrixsize{\firstdim}{\thirddim}}}}\), is formed by the three matrices~\(\mat{L}\) in~\(\Matrices{\RR}{\matrixsize{r}{\firstdim\seconddim}}\),~\(\mat{R}\) in~\(\Matrices{\RR}{\matrixsize{r}{\seconddim\thirddim}}\), and~\(\mat{P}\) in~\(\Matrices{\RR}{\matrixsize{\firstdim\thirddim}{r}}\), such that~\(\row{\mat{L}}{i}\) is~\(\vectorization{\LeftTensor{i}}\),~\(\row{\mat{R}}{i}\) is~\(\vectorization{\RightTensor{i}}\) and~\(\col{\mat{P}}{i}\) is~\(\vectorization{\Transpose{\ProductTensor{i}}}\).
Such a~\textsc{hm} representation will be denoted by~\(\HMRepresentation{\mat{L}}{\mat{R}}{\mat{P}}\).
\end{definition}
\textsc{hm} representations will be used throughout~\Cref{sec:accuracyTheoriticalBound,sec:implem}.
The bilinear map of matrix multiplication algorithm is expressed from an \textsc{hm}
representation as follows:
\begin{lemma}
From an \textsc{hm} representation~\(\HMRepresentation{\mat{L}}{\mat{R}}{\mat{P}}\), the matrix
multiplication bilinear map is expressed as
\begin{equation}\label{eq:HMRepresentation2MatrixMultiplicationFormula}
\begin{array}{cccc}
\beta_{\textsc{MM}}: & {{\Matrices{\Field}{\matrixsize{\firstdim}{\seconddim}}}\times{\Matrices{\Field}{\matrixsize{\seconddim}{\thirddim}}}}
&\rightarrow&
\Matrices{\Field}{\matrixsize{\firstdim\thirddim}{1}},\\
&(\mat{A},\mat{B}) & \mapsto & \MatrixProduct%
{\mat{P}}
{\left( \HadamardProduct
         {\left(\MatrixProduct{L}{\vectorization{\mat{A}}}\right)}%
         {\left(\MatrixProduct{R}{\vectorization{\mat{B}}}\right)}%
\right)},
\end{array}
\end{equation}
where~\(\HadamardProduct{}{}\) stands for the Hadamard product.
\end{lemma}
More generally, we will say that any bilinear map~\({\beta : {\RR^{e}} \times {\RR^{f}} \rightarrow {\RR^{g}}}\)
is represented by an \textsc{hm} representation~\({\HMRepresentation{\mat{L}}{\mat{R}}{\mat{P}}}\)
in~\({{\RR^{{r}\times{e}}}\times{\RR^{{r}\times{f}}}\times{\RR^{{r}\times{g}}}}\) if it satisfies
\begin{equation}
\beta(u,v)=\sum_{i=1}^{r} \col{\mat{P}}{i} \HadamardProduct{(\row{\mat{L}}{i}\cdot u)}{(\row{\mat{R}}{i}\cdot{v})}.
\end{equation}

In order handle recursive applications of a bilinear map in~\Cref{sec:accuracyTheoriticalBound}, we
will use the following notations:
given an operator~\(\beta\) and a base case operator \({\beta^{(0)}:\RR^{e_0}\times\RR^{f_0}\rightarrow\RR^{g_0}}\), a recursive bilinear operator~\(\tpow{\beta}{\ell}\) is defined using as follows:
\begin{equation}
\tpow{\beta}{\ell}:
\begin{array}[t]{cl}
\RR^{e_0e^\ell}\times\RR^{f_0f^\ell}&\rightarrow\RR^{g_0g^\ell},\\
(u,v) &\mapsto \sum_{i=1}^r \col{\mat{P}}{i} \tpow{\beta}{\ell-1}(\row{\mat{L}}{i} \cdot u, \row{\mat{R}}{i}\cdot{v}),
\end{array}
\end{equation}
where \(\tpow{\beta}{0}=\beta^{(0)}\).

Remark that the expression~\({\row{\mat{L}}{i}\cdot{u}}\) is an abuse of notation for the operation
where each coefficient~\(l_{i,j}\) of~\(\mat{L}\) multiplies a block of~\(e_0e^{\ell-1}\) contiguous
coefficients of~\(u\), namely:~\(\row{\mat{L}}{i}\cdot{u} =
  \row{\mat{L}}{i} \cdot \matr{u}{e,e_0e^{\ell-1}}
\).
\begin{notation}\
\begin{enumerate}
\item For convenience, we also define three dimensions~\({E=e_{0}e^{\ell}}\),~\({F=f_{0}f^{\ell}}\) and~\({G=g_{0}g^{\ell}}\).
\item When this operator encodes an~\(\matrixsize{\firstdim}{\seconddim}\)
    by~\(\matrixsize{\seconddim}{\thirddim}\) matrix multiplication formula, the map is denoted
    by~\(\beta_{\textsc{mm}}\) and we
    have~\(e=\firstdim\seconddim,f=\seconddim\thirddim,g=\firstdim\thirddim\) and
    also~\(M=m_0m^\ell,K=k_0k^\ell\) and~\(N=n_0n^\ell\).
  \end{enumerate}
\end{notation}
Now we turn to symmetries of matrix product tensor decomposition.
\subsection{Symmetries of matrix product tensor decomposition}\label{sec:Isotropies}
Remark that the matrix product is associated
to~\({\Trace(\MatrixProduct{\mat{A}}{\MatrixProduct{\mat{B}}{\mat{C}}})}\) by~\Cref{eq:TrilinearForm} and that, given three invertible matrices~\({\mat{U},\mat{V},\mat{W}}\) of suitable sizes and the classical trace properties, this trace is equal to:
\begin{equation}\label{eq:isotropy}
\begin{array}{l}
\Trace\bigl(\Transpose{(\MatrixProduct{\MatrixProduct{\mat{A}}{\mat{B}}}{\mat{C}})}\bigr)
=\Trace(\MatrixProduct{\mat{C}}{\MatrixProduct{\mat{A}}{\mat{B}}})
=\Trace(\MatrixProduct{\mat{B}}{\MatrixProduct{\mat{C}}{\mat{A}}})\\
\textrm{and to}\ \Trace\bigl(\MatrixProduct{\Inverse{\mat{U}}}{\MatrixProduct{\mat{A}}{\mat{V}}}
\cdot\Inverse{\mat{V}}\cdot{\mat{B}}\cdot{\mat{W}}\cdot\Inverse{\mat{W}}\cdot{\mat{C}}\cdot{\mat{U}}\bigr).
\end{array}
\end{equation}
These relations illustrate the next theorem and induce the isotropy action on matrix product tensor decomposition presented below:
\begin{theorem}[{\cite[\S~2.8]{groot:1978a}}]\label{thm:groot}
The \emph{isotropy group} of the~\(\matrixsize{\firstdim}{\seconddim}\) by~\(\matrixsize{\seconddim}{\thirddim}\) matrix multiplication tensor is the semidirect product
	\begin{equation}
	{\left(
	{{\emph{\textsc{psl}}^{\pm}({{\RR}^{\firstdim}})}}
	\times
	{{\emph{\textsc{psl}}^{\pm}({{\RR}^{\seconddim}})}}
	\times
	{{\emph{\textsc{psl}}^{\pm}({{\RR}^{\thirddim}})}}
	\right)}\!\rtimes{\mathfrak{S}_{3}},
	\end{equation}
where~\(\emph{\textsc{psl}}\) stands for the group of matrices of determinant~\({\pm{1}}\) and~\(\mathfrak{S}_{3}\) for the symmetric group on~\(3\) elements.
\end{theorem}
As we do not harness the~\(\mathfrak{S}_{3}\) part of this isotropy group, we do not give further details on its action. Let us now describe how this group acts on \textsc{hm} representation.
\begin{definition}\label{lem:sandwiching}
Let~\(\Isotropy{g}\) denotes the isotropy associated to matrices~\({(\mat{U}\times\mat{V}\times\mat{W})}\) in~\({{{\emph{\textsc{psl}}^{\pm}({{\RR}^{\firstdim}})}}\times{{\emph{\textsc{psl}}^{\pm}({{\RR}^{\seconddim}})}}\times{{\emph{\textsc{psl}}^{\pm}({{\RR}^{\thirddim}})}}}\) and~\(\tensor{T}\) a rank-one tensor~\({\mat{A}\tensorproduct\mat{B}\tensorproduct\mat{C}}\);
the action~\({\IsotropyAction{\Isotropy{g}}{\tensor{T}}}\) of~\(\Isotropy{g}\) on~\(\tensor{T}\) is the rank-one tensor:
\begin{equation}
{\left(\MatrixProduct{\InvTranspose{\mat{U}}}{\MatrixProduct{\mat{A}}{\Transpose{\mat{V}}}}\right)}\!
\tensorproduct\!
{\left(\MatrixProduct{\InvTranspose{\mat{V}}}{\MatrixProduct{\mat{B}}{\Transpose{\mat{W}}}}\right)}\!
\tensorproduct\!
{\left(\MatrixProduct{\InvTranspose{\mat{W}}}{\MatrixProduct{\mat{C}}{\Transpose{\mat{U}}}}\right)}.
\end{equation}
This action is extended by additivity to higher tensor rank tensors.
\par
Given two isotropies~\(\Isotropy{g}_{1}\) defined by matrices~\({({\mat{U}_{1}}\times{\mat{V}_{1}}\times{\mat{W}_{1}})}\) and~\({\Isotropy{g}_{2}}\) defined by matrices~\({({\mat{U}_{2}}\times{\mat{V}_{2}}\times{\mat{W}_{2}})}\) both in~\({{{\emph{\textsc{psl}}^{\pm}({{\RR}^{\firstdim}})}}\times{{\emph{\textsc{psl}}^{\pm}({{\RR}^{\seconddim}})}}\times{{\emph{\textsc{psl}}^{\pm}({{\RR}^{\thirddim}})}}}\), the composition~\({\Isotropy{g}_{1}\IsotropyComposition \Isotropy{g}_{2}}\) is given by~\({(\MatrixProduct{\mat{U}_{1}}{\mat{U}_{2}}\times{\MatrixProduct{\mat{V}_{1}}{\mat{V}_{2}}}\times{\MatrixProduct{\mat{W}_{1}}{\mat{W}_{2}}})}\).
\end{definition}
The isotropies action on an \textsc{hm} representation is a direct consequence of the above results and presented in the following lemma.
\begin{lemma}\label{lem:actionOnHMRepresentation}
Let~\(\Isotropy{g}\) be~\({{({\mat{U}}\times{\mat{V}}\times{\mat{W}})}}\) in~\({{{\emph{\textsc{psl}}^{\pm}({{\RR}^{\firstdim}})}}\times{{\emph{\textsc{psl}}^{\pm}({{\RR}^{\seconddim}})}}\times{{\emph{\textsc{psl}}^{\pm}({{\RR}^{\thirddim}})}}}\) and an \emph{\textsc{hm}} representation~\(\HMRepresentation{\mat{L}}{\mat{R}}{\mat{P}}\) of a matrix product tensor decomposition, the action~\({\IsotropyAction{\Isotropy{g}}{\HMRepresentation{\mat{L}}{\mat{R}}{\mat{P}}}}\) of~\(\Isotropy{g}\) on~\(\HMRepresentation{\mat{L}}{\mat{R}}{\mat{P}}\) is another \emph{\textsc{hm}} representation of a matrix product tensor decomposition defined by:
\begin{equation}\label{eq:isotropyActionOnHMRepresentation}
\HMRepresentation%
{\MatrixProduct{L}{\bigl({\Transpose{\mat{V}}\tensorproduct{\Inverse{\mat{U}}}}\bigr)}}%
{\MatrixProduct{R}{\bigl({\Transpose{\mat{W}}\tensorproduct{\Inverse{\mat{V}}}}\bigr)}}%
{\MatrixProduct{\bigl({{\mat{U}}\tensorproduct{\InvTranspose{\mat{W}}}}\bigr)}{P}}.
\end{equation}
\end{lemma}
Dealing with a tensor decomposition or with the associated \textsc{hm} representation is not strictly equivalent; In~\Cref{lem:sandwiching} there is no need to care about the determinants of the matrices~\({(\mat{U},\mat{V},\mat{W})}\) while this fact is no more true for \Cref{eq:isotropyActionOnHMRepresentation} as (say)~\(\mat{U}\) acts on two Different components.
\par
The following theorem recalls that all~\(\matrixsize{2}{2}\)-matrix product algorithms with~\(7\) coefficient multiplications are obtained by this single orbit of the action of isotropies on Strassen tensor decomposition:
\begin{theorem}[{\cite[\S~0.1]{groot:1978}}]\label{thm:IsotropiesActTransitivelyOnOptimalAlgorithm}
The group~\({{\emph{\textsc{psl}}^{\pm}({{\RR}^{\matrixsize{2}{2}}})}^{\times{3}}}\) acts transitively on the variety of fast algorithms multiplying~\(\matrixsize{2}{2}\)-matrices.
\end{theorem}
More generally, isotropy action on any decomposition may define other matrix product algorithm of same tensor rank but with potentially more interesting characteristics as shown in \Cref{sec:numericalStabilityMeasure}.
We make explicit these properties in the following section.
\section{Bounds on the accuracy of bilinear maps}\label{sec:accuracyTheoriticalBound}
This section deals with the error analysis of recursive bilinear maps.
If a finite-dimensional real vector space~\(\mathbb{U}\) is equipped with a norm~\(\norm{\cdot}\) we
will denote by~\(\dualnorm{\cdot}\) the related dual norm; for any~\({\phi:\mathbb{U}\rightarrow\RR}\),
its norm~\({\dualnorm{\phi}}\) is defined by~\({\text{sup}(|\phi(v)|,\|v\|\leq{1})}\).
For instance, the max-norm~\(\maxnorm{\cdot}\) and the one-norm~\(\onenorm{\cdot}\) are dual one
with the other, while the two-norm~\(\twonorm{\cdot}\) is self-dual.
\par
We will also denote the Hamming weight~\({\#\{i|{x_{i}}\neq{0}\}}\) of~\(x\) by~\(\zeronorm{x}\).
The~\(n\)-dimensional vector of coefficients~\({(x_1,\ldots,x_n)}\) is denoted by~\({(x_i)}_{i\in\{{1}\ldots{n}\}}\) or more succinctly~\(\vectorif{x_i}{i}\) when the indexing is clear
from the context.
By extension, we denote~\({\norm{x}\times\norm{y}}\) by~\({\norm{x;y}}\)
and~\({\norm{\mat{L}}\times\norm{\mat{R}}\times\norm{\mat{P}}}\)
by~\({\norm{\mat{L};\mat{R};\mat{P}}}\).
\par
We will express the accuracy bounds in their full generality, using possibly two different norms in
the left and the right-hand side of the inequality of the bound. We will therefore usually denote
by~\(\pnorm{\cdot}\) the norm on the error term to be bounded, and by~\(\qnorm{\cdot}\) the norm
used appearing in the right-hand side of the bound. The dual norms will be respectively represented
by~\(\psnorm{\cdot}\) and~\(\qsnorm{\cdot}\).
\par
For a matrix~\(\mat{A}\) in~\(\RR^{{m}\times{n}}\), we recall that the map norm is defined by
\begin{equation}
  \norm{A}=\norm{\vectorif{\dualnorm{\row{\mat{A}}{i}}}{i}}.
\end{equation}
\subsection{Bounds on the output of  linear and bilinear maps}
\begin{lemma}\label{lem:operatornorms}
For a matrix~\(\mat{A}\) in~\(\RR^{\matrixsize{\firstdim}{\seconddim}}\) and vectors~\({x,y}\) in~\(\RR^{\seconddim}\) the following inequalities hold:
\begin{align}
  {|x \cdot y|} &\leq \dualnorm{x}\norm{y},\label{eq:dotprodbound}\\
  \norm{\mat{A}x} & \leq \norm{\mat{A}}\norm{x}.\label{eq:matvecbound}
\end{align}
\end{lemma}

We now introduce the growth factor~\(\GF{p}{q}\) that governs the growth in magnitude of a
bilinear map, and plays a central role in the forward error analysis of its computation.
\begin{definition}\label{def:gf}
The \emph{growth factor}~\(\GF{p}{q}\) of a bilinear map~\(\beta\) given by a representation~\(\HMRepresentation{\mat{L}}{\mat{R}}{\mat{P}}\), with respect to the norms~\(\pnorm{\cdot}\) and~\(\qnorm{\cdot}\) is defined by:
\begin{equation}
\GF{p}{q} =
\pnorm{\left(\sum_{i=1}^{r}\qsnorm{\row{\mat{L}}{i}}\qsnorm{\row{\mat{R}}{i}} |p_{j,i}|\right)_j}.
\end{equation}
\end{definition}
\begin{lemma}\label{lem:valuebound}
Given a bilinear map \(\beta\), the growth factor \(\GF{p}{q}\) associated to its \HMRepresentation{L}{R}{P} representation bounds its output as follows:
\begin{equation}
\pnorm{\beta(u,v)}\leq\GF{p}{q}\qnorm{u}\qnorm{v}\label{eq:boundbeta}.
\end{equation}
\end{lemma}
\begin{proof}
Let~\(\mat{D_{j}}=\text{Diag}_{i=1\dots r}(p_{j,i})\) and~\(c_j\) be the~\(j\)-th coefficient of~\(\beta(u,v)\).
Since
\begin{align}
  |c_j|&=|\Transpose{u} \Transpose{\mat{L}} \mat{D}_j\mat{R}{v}|
  \leq \qsnorm{\Transpose{u}\mat{L}\mat{D}_j\mat{R}}\qnorm{v}
  \leq\qsnorm{\Transpose{\mat{L}}\mat{D}_j\mat{R}}\qnorm{u}\qnorm{v},\\
  &\leq \qsnorm{\sum_{i=1}^r (\row{\mat{L}}{i}\otimes\row{\mat{R}}{i})p_{j,i}}\qnorm{u}\qnorm{v},
\end{align}
we have
\(|c_j|\leq\left(\sum_{i=1}^r\qsnorm{\row{\mat{L}}{i}}\qsnorm{\row{\mat{R}}{i}}|p_{j,i}|\right)\qnorm{u}\qnorm{v}\).
\end{proof}
However, this bound can be improved by additional knowledge on the result being computed.
For instance, in the case of a bilinear map computing a matrix multiplication~\({C={A}\times{B}}\) with~\(A\) in~\(\RR^{{m}\times{k}}\) and \(B\) in~\(\RR^{{k}\times{n}}\), the bound becomes
\begin{equation}\label{eq:mminf}
  \maxnorm{\beta(u,v)} = \maxnorm{C}\leq k\maxnorm{A}\maxnorm{B}=k\maxnorm{u}\maxnorm{v}
\end{equation}
when~\({p=q=\infty}\) or
\begin{equation}\label{eq:mmtwo}
  \twonorm{\beta(u,v)} = \twonorm{C}\leq \twonorm{A}\twonorm{B}=\twonorm{u}\twonorm{v}
\end{equation}
when~\({p=q=2}\).
Lastly, when~\({p}\) is~\(2\) and~\(q\) is~\(\infty\), the following inequalities hold:
\begin{align}
	|c_{i,j}|&\leq \onenorm{\row{A}{i}}\maxnorm{\col{B}{j}},\\
	&\leq \sqrt{k}\twonorm{\row{A}{i}}\maxnorm{B}.
\end{align}
Therefore, the matrix multiplicaiton output is bounded as follows:
\begin{equation}\label{eq:mmtwoinf}
\twonorm{C}\leq k^{3/2}\maxnorm{A}\maxnorm{B}.
\end{equation}
Consequently, we define now the amplification factor used in the sequel; this factor captures the growth in norm when applying a bilinear map.
\begin{definition}
The \emph{amplification factor} \(\AF{p}{q}\) of a bilinear map \(\beta\) with respect to the norms \(\pnorm{\cdot}\) and \(\qnorm{\cdot}\) is defined by:
\begin{enumerate}
  \item  when \(\beta\) is an \(m\times k\times n\) matrix multiplication map,
  \begin{equation}
     \left\{
     \begin{array}{ll}
       \AF{\infty}{\infty}&= k \\
       \AF{2}{2}&= 1 \\
       \AF{\infty}{2}&= 1\\
       \AF{2}{\infty}&= k^{3/2}\\
     \end{array}
     \right.
   \end{equation}
  \item \(  \AF{p}{q} = \GF{p}{q}\) otherwise.
\end{enumerate}
\end{definition}

Finally the growth in norm of the recursive application of a bilinear map is made explicit in the following Lemma:
\begin{lemma}\label{lem:betalbound}
  The output of \(\ell\) recursive application of a bilinear map \(\beta\) with base case \(\beta^{(0)}\), whose respective amplification factors are \(\AF{p}{q}\) and \(\AF{p}{q}^{0}\) is bounded as follows:
        \begin{equation}
          \pnorm{\tpow{\beta}{\ell}(u,v)} \leq \AF{p}{q}^{(0)}\AF{p}{q}^{\ell}\qnorm{u}\qnorm{v} \label{eq:boundbetal}\\
        \end{equation}
where
\begin{equation}
\left\{
\begin{array}{lll}
(\AF{\infty}{\infty}^{(0)},\AF{\infty}{\infty}) &=&  (k_0,k),\\
(\AF{2}{2}^{(0)},\AF{2}{2}) &=&  (1,1),\\
(\AF{\infty}{2}^{(0)},\AF{\infty}{2}) &=&  (1,1),\\
(\AF{2}{\infty}^{(0)},\AF{2}{\infty}) &=&  ({k_{0}}^{\!3/2},k^{3/2}),\\
\end{array}
\right.
\end{equation}
  when \(\beta\) is an \(m\times k\times n\) matrix multiplication map and
  \((\AF{p}{q}^{(0)},\AF{p}{q}) =  (\GF{p}{q}^{(0)},\GF{p}{q})\) otherwise
  where, \(\GF{p}{q}^{(0)}\) is such that
  \(\pnorm{\beta^{(0)}(u,v)} \leq \GF{p}{q}^{(0)} \qnorm{u}\qnorm{v}\).%
\end{lemma}

\begin{proof}
The first three cases follow from~\Cref{eq:mminf,eq:mmtwo,eq:mmtwoinf} and the last case  in the general setting  follows from~\Cref{lem:valuebound} by induction.
\end{proof}

\subsection{Forward error of bilinear algorithms}

We will consider the floating point arithmetic in the standard model of~\cite{Higham:2002}: the
expression \(\comp{x}\) denotes the approximated floating point value of a real number~\(x\) and the usual arithmetic maps satisfy~\({\comp{a\,\textup{op}\,b}=(a\,\textup{op}\,b)(1+\delta)}\)
for~\(\textup{op}\) in~\({\{+,-,\times,/\}}\) with~\(\delta\) the unit round off such
that~\({|\delta|\leq \ulp}\) (with \(\ulp\) denoting the machine precision).
We recall in the following Lemma some classical accuracy bounds for scalar products and summations.
\begin{lemma}[see~{\cite[Eq.~(3.5)]{Dai:2023aa}} and~{\cite[Eq.~(4.4)]{Higham:2002}}]\label{lem:dotprodbound}
For any vectors~\(u,v\in\RR^{\thirddim}\), the following forward error bounds hold:%
\begin{align}
\bigl|{\comp{{u}\cdot{v}}-{u}\cdot{v}}\bigr|&\leq\zeronorm{u}\sum_{i=1}^n{|u_{i}| |v_{i}|\ulp}+\bbigO{\ulp^2},\label{eq:dotproderrorfull}\\
&\leq\zeronorm{u}\dualnorm{u}\|v\|\ulp+\bbigO{\ulp^2},\label{eq:dotproderror}\\
\left|{\comp{\textstyle\sum_{i=1}^{n}u_{i}}-\textstyle\sum_{i=1}^{n}u_{i}}\right|&\leq
(n-1)\Bigl({\textstyle\sum_{i=1}^{n} |u_i|}\Bigr)\ulp +\bbigO{\ulp^2}.\label{eq:sumerror} %
\end{align}
\end{lemma}
In addition, we will need the following bounds on the accuracy of a scalar product with one approximate input vector~\({\comp{v}=v+\Delta_v}\):
\begin{align}
\left|{\comp{{u}\cdot{\comp{v}}}-{u}\cdot{v}}\right|&\leq\left|{\comp{{u}\cdot{\comp{v}}}-{u}\cdot{\comp{v}}}\right|+\bigl|{{u}\cdot{\Delta_{v}}}\bigr|,\\
&\leq\zeronorm{u}\sum_{i=1}^{n}{|u_{i}||v_{i}|\ulp}+\bigl|{{u}\cdot{\Delta_{v}}}\bigr|+\bbigO{\ulp^2},\label{eq:approxdotproddetail}\\
&\leq\sum_{i=1}^{n}{|u_{i}|\left({\zeronorm{u}|v_{i}|\ulp+\bigl|{\Delta_{v_i}}\bigr|}\right)}+\bbigO{\ulp^2},\label{eq:approxdotprodfull}\\
&\leq\zeronorm{u}\dualnorm{u}\norm{v}\ulp+\dualnorm{u}\norm{\Delta_v}+\bbigO{\ulp^{2}}.\label{eq:approxdotprod}
\end{align}
\begin{lemma}\label{lem:matvecerror}
The forward error resulting from applying an approximated vector~\({\hat{v}=v+\Delta_{v}}\) in~\(\RR^{n}\) to a matrix~\(\phi\) in~\(\RR^{{m}\times{n}}\) is bounded as follows:
\begin{equation}
\norm{\widehat{{\phi}\cdot{\hat{v}}}-{{\phi}\cdot{v}}}\leq{Q_{\phi}\norm{\phi}\norm{v}\ulp+\norm{\phi}\norm{\Delta_{v}}+\bbigO{\ulp^{2}}},
\end{equation}
where~\(Q_\phi\) denotes~\({{\maxnorm{(\zeronorm{\row{\phi}{i}})_{i}}}={\max_{i}\zeronorm{\row{\phi}{i}}}}\) and~\(\norm{\phi}\) is the usual operator norm~\(\norm{(\dualnorm{\row{\phi}{i}})_{i}}\).
\end{lemma}
\begin{lemma}\label{lem:lpowererror}
The forward error resulting from applying an approximated vector~\({\hat{v}=v+\Delta_{v}}\) in~\(\RR^{n^\ell}\) to the~\(\ell\)-th tensor power of a matrix~\(\phi\) in~\(\RR^{{m}\times{n}}\) is bounded as follows:
\begin{equation}
\norm{\comp{{\tpow{\phi}{\ell}}\cdot{\hat{v}}}-{\tpow{\phi}{\ell}}\cdot{v}}
\leq\ell Q_{\phi}\norm{\phi}^{\ell}\norm{v}\ulp+\norm{\phi}^{\ell}\norm{\Delta_v}+\bbigO{\ulp^{2}}.
\end{equation}
\end{lemma}

\par
In order to bound the forward error of recursive bilinear maps, we will first suppose that the
forward error  of the base case bilinear algorithm satisfies:
\begin{equation}
  \pnorm{\comp{\beta^{(0)}}(u,v)-\beta^{(0)}(u,v)} \leq \EF{p}{q}{0} \qnorm{u}\qnorm{v}\ulp +\bbigO{\ulp^2}.
\end{equation}
For instance, for a classical matrix multiplication base case: \Cref{eq:dotproderrorfull} implies
\(|\comp{c_{i,j}}-c_{i,j}|\leq k_0 \dualnorm{\row{A}{i}}\norm{\col{B}{j}}\), hence
\begin{equation}
  \left\{
  \begin{array}{lll}
	  \EF{\infty}{\infty}{0} &=& {k_{0}}^{\!2},\label{eq:basecaseinfty}\\ %
	  \EF{2}{2}{0} &=& k_{0}, \\
	  \EF{\infty}{2}{0} &=& k_{0}, \\
	  \EF{2}{\infty}{0} &=& {k_{0}}^{\!5/2}.
  \end{array}
  \right.
\end{equation}
The last inequality comes from the majoration \( \onenorm{\row{A}{i}}\leq \sqrt{k}\twonorm{\row{A}{i}}\).
\begin{theorem}\label{th:recbound}
The forward error in computing~\(\tpow{\beta}{\ell}\) satisfies
\begin{equation}
  \pnorm{\widehat{\tpow{\beta}{\ell}(u,v)}-\tpow{\beta}{\ell}(u,v)}\leq\EF{p}{q}{\ell}\qnorm{u}\qnorm{v}\ulp+\bbigO{\ulp^2}
\end{equation}
where%
\begin{equation}
  \EF{p}{q}{\ell} = \GF{p}{q}^\ell\left(\EF{p}{q}{0}+Q_0\AF{p}{q}^{(0)} \sum_{i=0}^{\ell-1} \left(\frac{\AF{p}{q}}{\GF{p}{q}}\right)^i\right)
\end{equation}
and~\(Q_0 = \max_j \left(\zeronorm{\row{\mat{P}}{j}}+ \max_i(\zeronorm{\row{\mat{L}}{i}}\!+\zeronorm{\row{\mat{R}}{i}}) \mathbb{1}_{p_{j,i}\neq0}\right)\).%
\end{theorem}

\begin{proof}
By induction, we will prove that the bound is of the form:
\begin{equation}
\pnorm{\Delta_{\tpow{\beta}{\ell}}} =\pnorm{\widehat{\tpow{\beta}{\ell}(u,v)} -\tpow{\beta}{\ell}(u,v)}
   \leq t_{\ell} \qnorm{u}\qnorm{v}\ulp +\bbigO{\ulp^2}
\end{equation}
for a parameter~\(t_{\ell}\) which value will be determined  in the process.
\par
Consider the~\(k\)-th coefficient~\(c_{k}\) of the output and write~\(k\) as~\({(G/g)J+j}\) with~\({0\leq{j}<G/g}\).
The coefficient \(c_k\) belongs to block \(C_J\) of~\({\lfrac{G}{g}=g_{0}g^{\ell-1}}\)
consecutive output coefficients: hence, \({C_J=\sum_{i=1}^r p_{J,i} \mat{H}^{(i)} }\), where~\(\mat{H}^{(i)}\) is equal to~\({\tpow{\beta}{\ell-1}(\row{\mat{L}}{i}\cdot{u},\row{\mat{R}}{i}\cdot{v})}\) for~\(i\) in~\({\{1\ldots{r}\}}\).
From~\Cref{eq:approxdotprodfull}
\begin{align}
  | \comp{c_k}-c_k|
  & \leq  \sum_{i=1}^{r} |p_{J,i}|
             \left( \zeronorm{\row{\mat{P}}{J}} \abs{\mat{H}^{(i)}_j}\ulp + \abs{\Delta_{\mat{H}^{(i)}_j}}\right)
             +\bigO{\ulp^2}\label{eq:deltaci},
\end{align}
and
\begin{align}
     \abs{\Delta_{\mat{H}^{(i)}_j}}
      \leq &\abs{\comp{\tpow{\beta}{\ell-1}}(\comp{\row{\mat{L}}{i}\cdot u}, \comp{\row{\mat{R}}{i} \cdot v})_j-
\tpow{\beta}{\ell-1}(\comp{\row{\mat{L}}{i}\cdot u}, \comp{\row{\mat{R}}{i} \cdot v})_j}\notag \\
&       + \abs{\tpow{\beta}{\ell-1}(\comp{\row{\mat{L}}{i}\cdot u}, \comp{\row{\mat{R}}{i} \cdot v})_j -
       \tpow{\beta}{\ell-1}(\row{\mat{L}}{i}\cdot u, \row{\mat{R}}{i} \cdot v)_j },\\
     \leq &\abs{\Delta_{(\tpow{\beta}{\ell-1})_j}}+
            \abs{\tpow{\beta}{\ell-1}(\Delta_L,\row{\mat{R}}{i}\cdot{v})_j}+\abs{\tpow{\beta}{\ell-1}(\row{\mat{L}}{i}\cdot{u},\Delta_{R})_j}\label{eq:deltahi}
   \end{align}
  by bilinearity of~\(\tpow{\beta}{\ell-1}\), where~\(\Delta_L,\Delta_R\) and~\(\Delta_{\tpow{\beta}{\ell-1}}\) are the approximations done in the computations of~\({\row{\mat{L}}{i}\cdot{u}}\), \({\row{\mat{R}}{i}\cdot{v}}\) and in the recursive computation of~\({\tpow{\beta}{\ell-1}(\comp{\row{\mat{L}}{i}\cdot u},\comp{\row{\mat{R}}{i}\cdot{v}})}\).
By induction hypothesis, the following inequalities hold:
\begin{align}
\pnorm{\Delta_{\tpow{\beta}{\ell-1}}} &\leq t_{\ell-1} \qnorm{\row{\mat{L}}{i}\cdot u+\Delta_\mat{L}}\qnorm{\row{\mat{R}}{i}\cdot v+\Delta_\mat{R}}\ulp +\bbigO{\ulp^2}, \\
  &\leq t_{\ell -1} \qsnorm{\row{\mat{L}}{i}} \qnorm{u} \qsnorm{\row{\mat{R}}{i}}\qnorm{v}\ulp + \bbigO{\ulp^2}.\label{eq:betalm1}
\end{align}
By~\Cref{lem:betalbound}, the following inequality holds:
\begin{equation}
  \pnorm{\tpow{\beta}{\ell-1}(\Delta_\mat{L},\row{\mat{R}}{i}\cdot v)}
  \leq \AF{p}{q}^{(0)} \AF{p}{q}^{\ell-1} \qnorm{\Delta_\mat{L}}\qsnorm{\row{\mat{R}}{i}}\qnorm{v},
\end{equation}
where, by~\Cref{eq:dotproderror}, the term~\(\qnorm{\Delta_\mat{L}}\) is bounded as follows:
\begin{equation}
  \qnorm{\Delta_\mat{L}} \leq  \zeronorm{\row{\mat{L}}{i}}\qsnorm{\row{\mat{L}}{i}}\qnorm{u}\ulp +\bbigO{\ulp^2}.%
\end{equation}
Hence
\begin{multline}\label{eq:betadeltaL}
   \pnorm{\tpow{\beta}{\ell-1}(\Delta_\mat{L},\row{\mat{R}}{i}\cdot v)}
   \leq \AF{p}{q}^{(0)} \AF{p}{q}^{\ell-1}
   \zeronorm{\row{\mat{L}}{i}}\qsnorm{\row{\mat{L}}{i}}\qsnorm{\row{\mat{R}}{i}}\qnorm{u}\qnorm{v}\ulp\\ +\bbigO{\ulp^2}
\end{multline}
Similarly,
\begin{multline}\label{eq:betadeltaR}
   \pnorm{\tpow{\beta}{\ell-1}(\row{\mat{L}}{i}\cdot{u},\Delta_\mat{R})}
   \leq \AF{p}{q}^{(0)} \AF{p}{q}^{\ell-1}
   \zeronorm{\row{\mat{R}}{i}}\qsnorm{\row{\mat{L}}{i}}\qsnorm{\row{\mat{R}}{i}}\qnorm{u}\qnorm{v}\ulp\\ +\bbigO{\ulp^2}
\end{multline}
Gathering~\Cref{eq:deltahi,eq:betalm1,eq:betadeltaL,eq:betadeltaR} we deduce that
\begin{multline}
     \pnorm{\Delta_{\mat{H}^{(i)}}} \leq
	\left({\AF{p}{q}^{(0)} \AF{p}{q}^{\ell-1}\left(\zeronorm{\row{\mat{L}}{i}}+\zeronorm{\row{\mat{R}}{i}}\right)+t_{\ell-1}}\right)\\
     \times \qsnorm{\row{\mat{L}}{i}} \qsnorm{\row{\mat{R}}{i}} \qnorm{u} \qnorm{v} \ulp + \bbigO{\ulp^{2}},
\end{multline}
then by~\Cref{eq:deltaci} and because
\begin{equation}
  \pnorm{\mat{H}^{(i)}} \leq \AF{p}{q}^{(0)}\AF{p}{q}^{\ell-1}
  \qsnorm{\row{\mat{L}}{i}}\qsnorm{\row{\mat{R}}{i}}\qnorm{u}\qnorm{v}\ulp +\bbigO{\ulp^2},
  \end{equation}
\begin{multline}
\pnorm{\comp{C_J}-C_J}\leq \sum_{i=1}^{r} \left(\AF{p}{q}^{(0)} \AF{p}{q}^{\ell-1}\left(\zeronorm{\row{\mat{L}}{i}}+\zeronorm{\row{\mat{R}}{i}}+\zeronorm{\row{\mat{P}}{J}}\right)+t_{\ell-1}\right)\\
\times\qsnorm{\row{\mat{L}}{i}}\qsnorm{\row{\mat{R}}{i}}\abs{p_{i,J}}\norm{u}\norm{v}\ulp+\bbigO{\ulp^{2}},
\end{multline}
and thus that
\begin{equation}
  \pnorm{\comp{c}-c}\leq\left({\AF{p}{q}^{(0)} \AF{p}{q}^{\ell-1}Q_{0}+t_{\ell-1}}\right)\GF{p}{q}\pnorm{u}\pnorm{v}\ulp+\bbigO{\ulp^2}.
  \end{equation}
We deduce that~\(t_\ell\) must then satisfy:
\begin{equation}
\left\{\begin{array}{llll}
t_{\ell} &=& \left({ \AF{p}{q}^{(0)} \AF{p}{q}^{\ell-1} Q_{0}+t_{\ell-1}}\right) \GF{p}{q} &\text{for}\ \ell>0,\\
t_{0}   &=& \EF{p}{q}{0}.%
\end{array}\right.
\end{equation}
This recurrence relation solves into
\begin{equation}
  {t_\ell =  \GF{p}{q}^{\ell}\EF{p}{q}{0} + Q_0 \AF{p}{q}^{(0)} \GF{p}{q}^\ell\sum_{i=0}^{\ell-1} {\left(\frac{\AF{p}{q}}{\GF{p}{q}}\right)}^i}
\end{equation}
which concludes the proof.
\end{proof}

\begin{corollary}\label{cor:bilinearerror}
Let \(Q_k\) be~\(Q_0\frac{\GF{p}{q}}{\GF{p}{q}-k}\).%
This term is increasing w.r.t.\ \(k\).
When \(\beta\) encodes a matrix multiplication algorithm, then
\begin{description}
  \item[in norms \( p=q=\infty\):]
  \begin{align}
	  \EF{\infty}{\infty}{\ell} &= \GF{\infty}{\infty}^\ell \left(k_0^2+k_0Q_k\right) - k_0k^\ell Q_k\\
	   &= \left(\frac{K}{k_0}\right)^{\log_k\GF{\infty}{\infty}} \left(k_0^2+k_0Q_k\right) - KQ_k, \label{case:mmii}
  \end{align}
\item[in norms  \( p=q=2\):]
  \begin{equation}
    \EF{2}{2}{\ell} = \GF{2}{2}^\ell \left(k_0+Q_1\right) -  Q_1
           = \left(\frac{K}{k_0}\right)^{\log_k \GF{2}{2}} \left(k_0+Q_1\right) - Q_1, \label{case:mm22}
  \end{equation}
\item[in norms  \( p=2,q=\infty\):]
  \begin{align}
    \EF{2}{\infty}{\ell} &= \GF{2}{\infty}^\ell k_0^{3/2}\left(k_0+Q_{k^{3/2}}\right) -  Q_{k^{3/2}} (k_0k^\ell)^{3/2}\\
           &= \left(\frac{K}{k_0}\right)^{\log_k \GF{2}{\infty}} k_0^{3/2}\left(k_0+Q_{k^{3/2}}\right) - Q_{k^{3/2}}K^{3/2}. \label{case:mm2i}
  \end{align}
  \item[in norms \( p=\infty,q=2\):]
  \begin{equation}
    \EF{\infty}{2}{\ell} = \GF{\infty}{2}^\ell \left(k_0+Q_1\right) -  Q_1
	   = \left(\frac{K}{k_0}\right)^{\log_k\GF{\infty}{2}} \left(k_0+Q_1\right) - Q_{1}. \label{case:mmi2}
  \end{equation}
\end{description}
  Otherwise, with a general tensor,
  \begin{align}
    \EF{p}{q}{\ell} &=  \GF{p}{q}^\ell (\EF{p}{q}{0} +\ell Q_0 \GF{p}{q}^{(0)}),\\
    &=   \left(\frac{G}{g_0}\right)^{\log_g\GF{p}{q}} (\EF{p}{q}{0} + \log_g\left(\frac{G}{g_0}\right) Q_0 \GF{p}{q}^{(0)}). \label{case:other}
  \end{align}
\end{corollary}
\begin{proof}
With a matrix multiplication tensor,~\({\AF{p}{q}<\GF{p}{q}}\), and the geometric progression in the bound of~\Cref{th:recbound} converges so that
  \begin{equation}
    \EF{p}{q}{\ell} = \GF{p}{q}^\ell \EF{p}{q}{0} + Q_0 \AF{p}{q}^{(0)} \GF{p}{q}\frac{\GF{p}{q}^\ell -\AF{p}{q}^\ell}{\GF{p}{q}-\AF{p}{q}}.
  \end{equation}
The values are then obtained using the specializations
\begin{align}
(\EF{\infty}{\infty}{0},\AF{\infty}{\infty}^{(0)},\AF{\infty}{\infty}) & = ({k_{0}}^{\!2},k_0,k)\ \text{for~\Cref{case:mmii},}\\
  (\EF{2}{2}{0},\AF{2}{2}^{(0)},\AF{2}{2})  & = (k_0,1,1)\ \text{for~\Cref{case:mm22},}\\
  (\EF{\infty}{2}{0},\AF{\infty}{2}^{(0)},\AF{\infty}{2}) & = (k_0,1,1)\ \text{for~\Cref{case:mmi2}},\\
(\EF{2}{\infty}{0},\AF{2}{\infty}^{(0)},\AF{2}{\infty}) & = ({k_{0}}^{\!5/2},{k_{0}}^{\!3/2},k^{3/2})\ \text{for~\Cref{case:mm2i}},\\
  (\EF{p}{q}{0},\AF{p}{q}^{(0)},\AF{p}{q}) & = (\EF{p}{q}{0},\GF{p}{q}^{(0)},\GF{p}{q})\ \text{for~\Cref{case:other}.}
\end{align}
\end{proof}

\begin{remark}
  Note that the \(Q_0\) factor is the same as in~\cite[Definition~1]{BBDLS16} and does not depend on
  the choice of norms for \(\pnorm{\cdot}\) and \(\qnorm{\cdot}\).
\end{remark}

\Cref{cor:bilinearerror} generalizes or improves on previous similar results in~\cite{brent:1970a,Higham:2002,demmel:2007a,BBDLS16,Dai:2023aa}.
In fact,~\cite{Dai:2023aa} considers a single recursive level without base case in
norms~\({p=q=2}\); the works~\cite{brent:1970a,Higham:2002} have similarly tight bounds but only for Strassen and
Winograd's algorithms in norms~\({p=q=\infty}\);
lastly~\cite{demmel:2007a,BBDLS16} involve an additional logarithmic factor likely due to a
looser bound  when bounding~\(\textbigmaxnorm{\tpow{\beta}{\ell-1}}\), i.e.\ not taking into account
the matrix multiplication nature of the tensor in~\Cref{lem:betalbound}.
\par
Even though the choice of  max-norms on both left- and right-handsides of the bound  is the usual
choice in forward error analysis, it is interesting to note that the asymptotic regime varies
depending on the choices made on \(p\) and \(q\).
In addition, the \((2,2)\) norm case produces a smoother growth factor expression, more amenable to optimizations, as
will be shown \Cref{sec:numericalStabilityMeasure} using an even smoother relaxation of it.

The hierarchy between the \(p\)-norms, \(\maxnorm{x}\leq \onenorm{x}\leq \twonorm{x}\leq \dots\),
implies that the error factor \({f_\text{ALG}}_{p,q}\) of an accuracy bound will grow as \(p\) (the
norm on the left-hand side)  decreases and \(q\) (the norm on the right-handside) increases.

\Cref{diag:gamma-norms} illustrates how the growth factors of the two considered algorithms depend on
these choice of norms. Each rectangle represents by its abscissa support the value \((p,q)\) and by
its height the value of the corresponding growth factor \(\GF{p}{q}\).

Hence the width of the rectangles is directly connected to the looseness of the bounds, and
therefore the rectangles tend to be either wide and flat or tall and skinny.

We report on this figure two variants for each algorithm: the standard one and a variant with
improved accuracy, obtained by minimizing the \(\GF{2}{2}\) factor, as will be presented in~\Cref{sec:numericalStabilityMeasure}.
Note that  the accurate variants not only  improves the theoretical accuracy in norms
\((p,q)=(2,2)\), but for any norm combination \(p,q\in\{2,\infty\}\) for Smirnov's algorithm, and
for any instances with \(q=2\) for Strassen's algorithm.
\begin{figure}[htbp]\centering
\begingroup
  \fontfamily{1}%
  \selectfont
  \makeatletter
  \providecommand\color[2][]{%
    \GenericError{(gnuplot) \space\space\space\@spaces}{%
      Package color not loaded in conjunction with
      terminal option `colourtext'%
    }{See the gnuplot documentation for explanation.%
    }{Either use 'blacktext' in gnuplot or load the package
      color.sty in LaTeX.}%
    \renewcommand\color[2][]{}%
  }%
  \providecommand\includegraphics[2][]{%
    \GenericError{(gnuplot) \space\space\space\@spaces}{%
      Package graphicx or graphics not loaded%
    }{See the gnuplot documentation for explanation.%
    }{The gnuplot epslatex terminal needs graphicx.sty or graphics.sty.}%
    \renewcommand\includegraphics[2][]{}%
  }%
  \providecommand\rotatebox[2]{#2}%
  \@ifundefined{ifGPcolor}{%
    \newif\ifGPcolor
    \GPcolortrue
  }{}%
  \@ifundefined{ifGPblacktext}{%
    \newif\ifGPblacktext
    \GPblacktexttrue
  }{}%
  \let\gplgaddtomacro\g@addto@macro
  \gdef\gplbacktext{}%
  \gdef\gplfronttext{}%
  \makeatother
  \ifGPblacktext
    \def\colorrgb#1{}%
    \def\colorgray#1{}%
  \else
    \ifGPcolor
      \def\colorrgb#1{\color[rgb]{#1}}%
      \def\colorgray#1{\color[gray]{#1}}%
      \expandafter\def\csname LTw\endcsname{\color{white}}%
      \expandafter\def\csname LTb\endcsname{\color{black}}%
      \expandafter\def\csname LTa\endcsname{\color{black}}%
      \expandafter\def\csname LT0\endcsname{\color[rgb]{1,0,0}}%
      \expandafter\def\csname LT1\endcsname{\color[rgb]{0,1,0}}%
      \expandafter\def\csname LT2\endcsname{\color[rgb]{0,0,1}}%
      \expandafter\def\csname LT3\endcsname{\color[rgb]{1,0,1}}%
      \expandafter\def\csname LT4\endcsname{\color[rgb]{0,1,1}}%
      \expandafter\def\csname LT5\endcsname{\color[rgb]{1,1,0}}%
      \expandafter\def\csname LT6\endcsname{\color[rgb]{0,0,0}}%
      \expandafter\def\csname LT7\endcsname{\color[rgb]{1,0.3,0}}%
      \expandafter\def\csname LT8\endcsname{\color[rgb]{0.5,0.5,0.5}}%
    \else
      \def\colorrgb#1{\color{black}}%
      \def\colorgray#1{\color[gray]{#1}}%
      \expandafter\def\csname LTw\endcsname{\color{white}}%
      \expandafter\def\csname LTb\endcsname{\color{black}}%
      \expandafter\def\csname LTa\endcsname{\color{black}}%
      \expandafter\def\csname LT0\endcsname{\color{black}}%
      \expandafter\def\csname LT1\endcsname{\color{black}}%
      \expandafter\def\csname LT2\endcsname{\color{black}}%
      \expandafter\def\csname LT3\endcsname{\color{black}}%
      \expandafter\def\csname LT4\endcsname{\color{black}}%
      \expandafter\def\csname LT5\endcsname{\color{black}}%
      \expandafter\def\csname LT6\endcsname{\color{black}}%
      \expandafter\def\csname LT7\endcsname{\color{black}}%
      \expandafter\def\csname LT8\endcsname{\color{black}}%
    \fi
  \fi
    \setlength{\unitlength}{0.0500bp}%
    \ifx\gptboxheight\undefined%
      \newlength{\gptboxheight}%
      \newlength{\gptboxwidth}%
      \newsavebox{\gptboxtext}%
    \fi%
    \setlength{\fboxrule}{0.5pt}%
    \setlength{\fboxsep}{1pt}%
    \definecolor{tbcol}{rgb}{1,1,1}%
\begin{picture}(7200.00,4320.00)%
    \gplgaddtomacro\gplbacktext{%
      \csname LTb\endcsname
      \put(592,479){\makebox(0,0)[r]{\strut{}$0$}}%
      \csname LTb\endcsname
      \put(592,990){\makebox(0,0)[r]{\strut{}$5$}}%
      \csname LTb\endcsname
      \put(592,1502){\makebox(0,0)[r]{\strut{}$10$}}%
      \csname LTb\endcsname
      \put(592,2014){\makebox(0,0)[r]{\strut{}$15$}}%
      \csname LTb\endcsname
      \put(592,2525){\makebox(0,0)[r]{\strut{}$20$}}%
      \csname LTb\endcsname
      \put(592,3037){\makebox(0,0)[r]{\strut{}$25$}}%
      \csname LTb\endcsname
      \put(592,3548){\makebox(0,0)[r]{\strut{}$30$}}%
      \csname LTb\endcsname
      \put(592,4060){\makebox(0,0)[r]{\strut{}$35$}}%
      \csname LTb\endcsname
      \put(1222,239){\makebox(0,0){\strut{}$p=\infty$}}%
      \csname LTb\endcsname
      \put(2990,239){\makebox(0,0){\strut{}$p=2$}}%
      \csname LTb\endcsname
      \put(4757,239){\makebox(0,0){\strut{}$q=\infty$}}%
      \csname LTb\endcsname
      \put(6524,239){\makebox(0,0){\strut{}$q=2$}}%
    }%
    \gplgaddtomacro\gplfronttext{%
      \csname LTb\endcsname
      \put(6077,3844){\makebox(0,0)[r]{\strut{}Strassen}}%
      \csname LTb\endcsname
      \put(6077,3604){\makebox(0,0)[r]{\strut{}Strassen accurate}}%
      \csname LTb\endcsname
      \put(195,2269){\rotatebox{-270.00}{\makebox(0,0){\strut{}$\gamma_{p,q}$}}}%
    }%
    \gplbacktext
    \put(0,0){\includegraphics[width={360.00bp},height={216.00bp}]{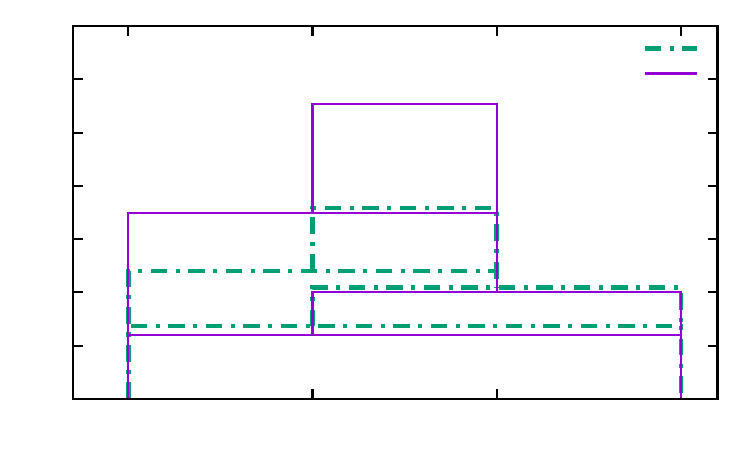}}%
    \gplfronttext
  \end{picture}%
\endgroup\\
\begingroup
  \fontfamily{1}%
  \selectfont
  \makeatletter
  \providecommand\color[2][]{%
    \GenericError{(gnuplot) \space\space\space\@spaces}{%
      Package color not loaded in conjunction with
      terminal option `colourtext'%
    }{See the gnuplot documentation for explanation.%
    }{Either use 'blacktext' in gnuplot or load the package
      color.sty in LaTeX.}%
    \renewcommand\color[2][]{}%
  }%
  \providecommand\includegraphics[2][]{%
    \GenericError{(gnuplot) \space\space\space\@spaces}{%
      Package graphicx or graphics not loaded%
    }{See the gnuplot documentation for explanation.%
    }{The gnuplot epslatex terminal needs graphicx.sty or graphics.sty.}%
    \renewcommand\includegraphics[2][]{}%
  }%
  \providecommand\rotatebox[2]{#2}%
  \@ifundefined{ifGPcolor}{%
    \newif\ifGPcolor
    \GPcolortrue
  }{}%
  \@ifundefined{ifGPblacktext}{%
    \newif\ifGPblacktext
    \GPblacktexttrue
  }{}%
  \let\gplgaddtomacro\g@addto@macro
  \gdef\gplbacktext{}%
  \gdef\gplfronttext{}%
  \makeatother
  \ifGPblacktext
    \def\colorrgb#1{}%
    \def\colorgray#1{}%
  \else
    \ifGPcolor
      \def\colorrgb#1{\color[rgb]{#1}}%
      \def\colorgray#1{\color[gray]{#1}}%
      \expandafter\def\csname LTw\endcsname{\color{white}}%
      \expandafter\def\csname LTb\endcsname{\color{black}}%
      \expandafter\def\csname LTa\endcsname{\color{black}}%
      \expandafter\def\csname LT0\endcsname{\color[rgb]{1,0,0}}%
      \expandafter\def\csname LT1\endcsname{\color[rgb]{0,1,0}}%
      \expandafter\def\csname LT2\endcsname{\color[rgb]{0,0,1}}%
      \expandafter\def\csname LT3\endcsname{\color[rgb]{1,0,1}}%
      \expandafter\def\csname LT4\endcsname{\color[rgb]{0,1,1}}%
      \expandafter\def\csname LT5\endcsname{\color[rgb]{1,1,0}}%
      \expandafter\def\csname LT6\endcsname{\color[rgb]{0,0,0}}%
      \expandafter\def\csname LT7\endcsname{\color[rgb]{1,0.3,0}}%
      \expandafter\def\csname LT8\endcsname{\color[rgb]{0.5,0.5,0.5}}%
    \else
      \def\colorrgb#1{\color{black}}%
      \def\colorgray#1{\color[gray]{#1}}%
      \expandafter\def\csname LTw\endcsname{\color{white}}%
      \expandafter\def\csname LTb\endcsname{\color{black}}%
      \expandafter\def\csname LTa\endcsname{\color{black}}%
      \expandafter\def\csname LT0\endcsname{\color{black}}%
      \expandafter\def\csname LT1\endcsname{\color{black}}%
      \expandafter\def\csname LT2\endcsname{\color{black}}%
      \expandafter\def\csname LT3\endcsname{\color{black}}%
      \expandafter\def\csname LT4\endcsname{\color{black}}%
      \expandafter\def\csname LT5\endcsname{\color{black}}%
      \expandafter\def\csname LT6\endcsname{\color{black}}%
      \expandafter\def\csname LT7\endcsname{\color{black}}%
      \expandafter\def\csname LT8\endcsname{\color{black}}%
    \fi
  \fi
    \setlength{\unitlength}{0.0500bp}%
    \ifx\gptboxheight\undefined%
      \newlength{\gptboxheight}%
      \newlength{\gptboxwidth}%
      \newsavebox{\gptboxtext}%
    \fi%
    \setlength{\fboxrule}{0.5pt}%
    \setlength{\fboxsep}{1pt}%
    \definecolor{tbcol}{rgb}{1,1,1}%
\begin{picture}(7200.00,4320.00)%
    \gplgaddtomacro\gplbacktext{%
      \csname LTb\endcsname
      \put(692,479){\makebox(0,0)[r]{\strut{}$0$}}%
      \csname LTb\endcsname
      \put(692,990){\makebox(0,0)[r]{\strut{}$100$}}%
      \csname LTb\endcsname
      \put(692,1502){\makebox(0,0)[r]{\strut{}$200$}}%
      \csname LTb\endcsname
      \put(692,2014){\makebox(0,0)[r]{\strut{}$300$}}%
      \csname LTb\endcsname
      \put(692,2525){\makebox(0,0)[r]{\strut{}$400$}}%
      \csname LTb\endcsname
      \put(692,3037){\makebox(0,0)[r]{\strut{}$500$}}%
      \csname LTb\endcsname
      \put(692,3548){\makebox(0,0)[r]{\strut{}$600$}}%
      \csname LTb\endcsname
      \put(692,4060){\makebox(0,0)[r]{\strut{}$700$}}%
      \csname LTb\endcsname
      \put(1315,239){\makebox(0,0){\strut{}$p=\infty$}}%
      \csname LTb\endcsname
      \put(3053,239){\makebox(0,0){\strut{}$p=2$}}%
      \csname LTb\endcsname
      \put(4791,239){\makebox(0,0){\strut{}$q=\infty$}}%
      \csname LTb\endcsname
      \put(6530,239){\makebox(0,0){\strut{}$q=2$}}%
    }%
    \gplgaddtomacro\gplfronttext{%
      \csname LTb\endcsname
      \put(6077,3844){\makebox(0,0)[r]{\strut{}Smirnov}}%
      \csname LTb\endcsname
      \put(6077,3604){\makebox(0,0)[r]{\strut{}Smirnov accurate}}%
      \csname LTb\endcsname
      \put(195,2269){\rotatebox{-270.00}{\makebox(0,0){\strut{}$\gamma_{p,q}$}}}%
    }%
    \gplbacktext
    \put(0,0){\includegraphics[width={360.00bp},height={216.00bp}]{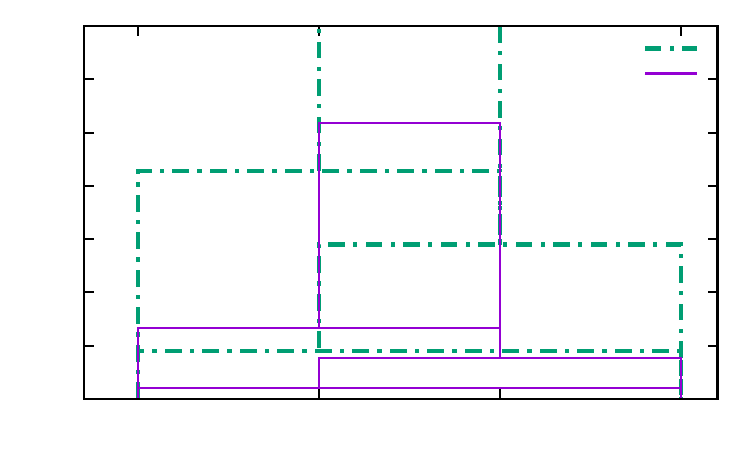}}%
    \gplfronttext
  \end{picture}%
\endgroup
\caption{Growth factors for Strassen~\(\FMMA{2}{2}{2}{40}\) (above) Smirnov's~\(\FMMA{3}{3}{6}{40}\) (below) algorithms and their accurate version, depending on the choice of output norm~\(p\) and input norm~\(q\).}\label{diag:gamma-norms}
\end{figure}
\begin{remark}
Note that in the case~\({(p,q)=(\infty,2)}\), \Cref{case:mmi2} is a significant improvement over the alternative way combining the bound of \Cref{case:mmii} for~\({(p,q)=(\infty,\infty)}\) with the inequality~\({\maxnorm{u}\maxnorm{v}\leq \twonorm{u}\twonorm{v}}\).
  Indeed, when comparing
  \(\EF{\infty}{2}{\ell}= \GF{\infty}{2}^\ell(k_0+Q_1) -Q_1\) with
  \(\EF{\infty}{\infty}{\ell}= \GF{\infty}{\infty}^\ell(k_0^2+k_0Q_k)-k_0k^\ell Q_k\)
  not only is the factor \(\GF{\infty}{2}\)  much smaller  than  \(\GF{\infty}{\infty}\), but also \(Q_1<Q_k\).
\end{remark}
We now extend the bound on the forward error of bilinear maps to also handle approximated input.
\begin{corollary}\label{cor:approxbilinear}
Given approximated input~\({\comp{u}=u+\Delta_u}\) and~\({\comp{v}=v+\Delta_v}\), the forward error in
  computing~\(\tpow{\beta}{\ell}(\comp{u},\comp{v})\) is bounded as follows
\begin{multline}
\pnorm{\widehat{\tpow{\beta}{\ell}(\hat u,\hat v)}-\tpow{\beta}{\ell}(u,v)} \leq \EF{p}{q}{\ell}\qnorm{u}\qnorm{v}\ulp\\+\mu(\qnorm{\Delta_u}\qnorm{v}+\qnorm{u}\qnorm{\Delta_v})+\bbigO{\ulp^2}\end{multline}
where
\begin{equation}
  \EF{p}{q}{\ell} = \GF{p}{q}^\ell\left(\EF{p}{q}{0}+Q_0\AF{p}{q}^{(0)} \sum_{i=1}^\ell
  \left(\frac{\AF{p}{q}}{\GF{p}{q}}\right)^i\right),
  \mu = \AF{p}{q}^{(0)}\AF{p}{q}^\ell
\end{equation}

and~\(Q_0={\max_j \bigl(\zeronorm{\row{\mat{P}}{j}} + \max_i(\zeronorm{\row{\mat{L}}{i}}\!+\zeronorm{\row{\mat{R}}{i}}) \mathbb{1}_{p_{j,i}\neq0}\bigr)}\).%
\end{corollary}
\begin{proof}
\begin{multline}
\pnorm{\widehat{\tpow{\beta}{\ell}(\hat u,\hat v )}-\tpow{\beta}{\ell}(u,v )} \leq
\pnorm{\widehat{\tpow{\beta}{\ell}(\hat u, \hat v)} - \tpow{\beta}{\ell}(\hat u,\hat v)}\\
	+\pnorm{\tpow{\beta}{\ell}(\hat u, \hat v) - \tpow{\beta}{\ell}(u,v)}.
\end{multline}
By~\Cref{th:recbound}, \(    \pnorm{\widehat{\tpow{\beta}{\ell}(\hat u, \hat v)} -
  \tpow{\beta}{\ell}(\hat u,\hat v)} \leq \EF{p}{q}{\ell} \qnorm{\hat u}\qnorm{\hat v}\ulp + \bbigO{\ulp^2}\).
\par
Then, by bilinearity and applying~\Cref{lem:betalbound},
\begin{align}
  \pnorm{\tpow{\beta}{\ell}(\hat u,\hat v)-\tpow{\beta}{\ell}(u,v)} &\leq \pnorm{\tpow{\beta}{\ell}(\Delta_u,\hat v)}  +\pnorm{\tpow{\beta}{\ell}(\hat u,\Delta_v)},\\
&\leq \AF{p}{q}^{(0)}\AF{p}{q}^\ell (\qnorm{\Delta_u}\qnorm{v}+\qnorm{u}\qnorm{\Delta_v}).
\end{align}
\end{proof}

\subsection{Forward error of alternative basis matrix multiplication algorithms}
Consider an alternative basis matrix multiplication algorithm represented by the recursive bilinear formula:
\begin{equation}
\tpow{\beta_\mathsc{mmab}}{\ell}:
\begin{array}[t]{cl}
\RR^{e_0e^\ell}\times\RR^{f_0f^\ell}&\rightarrow\RR^{g_0g^\ell},\\
(u,v) &\mapsto \Transpose{\nu} \cdot \sum_{i=1}^r \col{\mat{P}^\nu}{i} \cdot \tpow{\beta_\mathsc{mmab}}{\ell-1}(\row{\mat{L}^\phi}{i} \cdot \phi \cdot u,\row{\mat{R}^\psi}{i}\cdot \psi \cdot v)
\end{array}
\end{equation}
where~\(\phi\) is in~\(\RR^{{e'}\times{e}}\),~\(\psi\) in~\(\RR^{{f'}\times{f}}\) and~\(\nu\) in~\(\RR^{{g'}\times{g}}\).
\par
Let us define the corresponding sparsified bilinear map:
  \begin{equation}
\tpow{\beta_\mathsc{ab}}{\ell}:
\begin{array}[t]{cl}
\RR^{e_0e'^\ell}\times\RR^{f_0f'^\ell}&\rightarrow\RR^{g_0g'^\ell},\\
(u,v) &\mapsto \sum_{i=1}^r \col{\mat{P}^\nu}{i} \cdot \tpow{\beta_\mathsc{ab}}{\ell-1}(\row{\mat{L}^\phi}{i} \cdot u,\row{\mat{R}^\psi}{i} \cdot v).
\end{array}
  \end{equation}
Then, by commutativity of tensor multiplications, the following equality holds:
\begin{equation}\label{eq:mmab}
\tpow{\beta_\mathsc{mmab}}{\ell}(u,v) = \Transpose{\left({\tpow{\nu}{\ell}}\right)} \tpow{\beta_\mathsc{ab}}{\ell}(\tpow{\phi}{\ell} \cdot u, \tpow{\psi}{\ell}    \cdot v).
\end{equation}
A matrix multiplication algorithm using alternate basis, follows the setting of~\Cref{eq:mmab}:
it first computes \(\tpow{\phi}{\ell}(u)\) and \(\tpow{\psi}{\ell}(v)\) which are then passed as input to \(\ell\) recursive
levels of the bilinear algorithm \(\beta_\mathsc{ab}\)  before applying~\(\tpow{(\Transpose{\nu})}{\ell}\) to
the output. Note that, since that \(\beta_\mathsc{ab}\) is not a matrix multiplication formula, its forward
error bound follows~\Cref{case:other}.

\begin{theorem}\label{thm:altbase}

  \begin{equation}
    \pnorm{\comp{\tpow{\beta_\mathsc{mmab}}{\ell}(u,v)}-\tpow{\beta_\mathsc{mmab}}{\ell}(u,v)}
    \leq \EF{p}{q}{\ell}\qnorm{u}\qnorm{v}\ulp+\bbigO{\ulp^2}
  \end{equation}
  where
  \begin{multline}
    \EF{p}{q}{\ell}= (\GF{p}{q}(\mathsc{mmab}))^\ell (\EF{p}{q}{0} + \GF{p}{q}^{(0)}\ell(Q_0+Q_\phi+Q_\psi+Q_{\Transpose{\nu}})),\\
  \textup{with}\ \GF{p}{q}(\mathsc{mmab}) = \GF{p}{q}(\mathsc{ab}) \qnorm{\phi}\qnorm{\psi}\pnorm{\Transpose{\nu}}.
  \end{multline}
\end{theorem}
Note that the error factor \(\GF{p}{q}(\mathsc{mmab})\) is necessarily always worse than that of the
underlying sparsified bilinear map \(\GF{p}{q}(\mathsc{ab})\) since the transformation
matrices \(\phi,\psi,\nu\)  have a norm greater than 1, as they perform an elimination to
produce a sparser map.
\begin{proof}
Following~\Cref{eq:matvecbound}, the input can be bounded as follows:
\begin{equation}
 \qnorm{\tpow{\phi}{\ell} \cdot u} \leq \qnorm{\phi}^\ell\qnorm{u},\quad
  \qnorm{\tpow{\psi}{\ell}\cdot v} \leq \qnorm{\psi}^\ell\qnorm{v}.
\end{equation}
By~\Cref{lem:lpowererror}, the error in these input can be bounded by:
  \begin{align}
    \qnorm{\comp{\tpow{\phi}{\ell}\cdot u} - \tpow{\phi}{\ell}\cdot u}
  &\leq \ell Q_\phi\qnorm{\phi}^\ell \qnorm{u}\ulp +\bbigO{\ulp^2},
  \end{align}
where \(Q_\phi=\max_i \zeronorm{\row{\phi}{i}}\).
Hence, by~\Cref{cor:approxbilinear}
\begin{align}
\pnorm{\Delta_{\tpow{\beta_\mathsc{ab}}{\ell}}}&=
\pnorm{\comp{\tpow{\beta_\mathsc{ab}}{\ell}}(\comp{\tpow{\phi}{\ell}\cdot u},\comp{\tpow{\psi}{\ell}\cdot v})-\tpow{\beta_\mathsc{ab}}{\ell}(\tpow{\phi}{\ell}\cdot u,\tpow{\psi}{\ell}\cdot v)} \\
&\leq  \EF{p}{q}{\ell}(\mathsc{ab}) \qnorm{\tpow{\phi}{\ell}\cdot u}\qnorm{\tpow{\psi}{\ell}\cdot v}\notag \\
	&\quad + \mu \left(\qnorm{\Delta_{\tpow{\phi}{\ell}}}\qnorm{\tpow{\psi}{\ell}\cdot v}
	+ \qnorm{\tpow{\phi}{\ell}\cdot u}\qnorm{\Delta_{\tpow{\psi}{\ell}}}\right)\\
     &\leq  \kappa_\textsc{ab} \qnorm{u}\qnorm{v}  \ulp + \bbigO{\ulp^2}
  \end{align}
  where
 \begin{align}
 \kappa_\mathsc{ab}&= \left({\EF{p}{q}{\ell}(\mathsc{ab}) +\ell\mu(Q_\phi + Q_\psi)}\right)\qnorm{\phi}^{\ell}\qnorm{\psi}^{\ell},\\
 &= {\left({\GF{p}{q}(\mathsc{ab}) \qnorm{\phi}\qnorm{\psi}}\right)}^{\ell}\left({\EF{p}{q}{0} +\ell\GF{p}{q}^{(0)}(Q_0+Q_\phi+Q_\psi)}\right).
  \end{align}
Lastly,~\({\tpow{\beta_\mathsc{mmab}}{\ell}(u,v)=\Transpose{\tpow{\nu}{\ell}}\cdot \tpow{\beta_\mathsc{ab}}{\ell}\left({\tpow{\phi}{\ell}\cdot{u},\tpow{\psi}{\ell}\cdot{v}}\right)}\).
Hence, applying~\Cref{lem:lpowererror} leads to
\begin{align}
\pnorm{\Delta_{\tpow{\beta_\mathsc{mmab}}{\ell}}} &\leq
  \ell Q_{\Transpose{\nu}}
	\pnorm{\Transpose{\nu}}^\ell \GF{p}{q}^{(0)}{\left({\GF{p}{q}(\mathsc{ab})}\right)}^\ell\qnorm{\phi}^\ell\qnorm{\psi}^\ell\qnorm{u}\qnorm{v}\ulp\notag\\
	&\quad+ \pnorm{\Transpose{\nu}}^\ell{\left({\GF{p}{q}(\mathsc{ab})\qnorm{\phi}\qnorm{\psi}}\right)}^{\ell}\notag\\
	&\quad\times\left({\EF{p}{q}{0}+\ell \GF{p}{q}^{(0)}(Q_0+Q_\phi+Q_\psi)}\right)\qnorm{u}\qnorm{v}\ulp+\bbigO{\ulp^2},\\
	&\leq {\left({\GF{p}{q}(\mathsc{ab})\qnorm{\phi}\qnorm{\psi}\pnorm{\phi}}\right)}^{\ell}\notag\\
	&\quad\times\left({\EF{p}{q}{0}+\ell
	{\GF{p}{q}}^{(0)}(Q_{0}+Q_{\phi}+Q_\psi+Q_{\Transpose{\nu}})}\right)
        \qnorm{u}\qnorm{v}\ulp\notag\\ &\quad+\bbigO{\ulp^2}.
\end{align}
That concludes the proof.
\end{proof}

\begin{remark}
  Note that the bound in~\Cref{thm:altbase} on an alternative basis algorithm is not as good as that
  of the corresponding original algorithm for several reasons: first, because the \emph{sparsified}
  bilinear map at the core of the recursion is no longer a matrix multiplication formula, which
  induces the additional logarithmic factor of~\Cref{case:other} ; second, because the contribution
  of the change of bases \(\phi,\psi,\nu\) impact multiplicatively the growth factor.
  We have been unable to reproduce the proof of~\cite[Corollary III.9]{Schwartz:2024aa} stating that
  the growth factor remains unchanged in the alternative basis variant of any bilinear matrix
  multiplication algorithm. We report these bounds in~\cref{tab:accbounds} and ours for the sake of completeness.
\end{remark}
\par
\Cref{tab:accbounds} presents the effective values of the growth factor and  accuracy bounds obtained in this section
for matrix multiplication algorithms build on \(\FMMA{2}{2}{2}{7}\) schemes (Strassen, Winograd,
and the accurate variant of~\cref{sec:numericalStabilityMeasure} denoted by <2,2,2> acc) and \(\FMMA{3}{3}{6}{40}\) schemes
(Smirnov and the accurate variant of~\cref{sec:numericalStabilityMeasure} denoted by <3,3,6>
acc). The Alternative Basis variants of these schemes also displayed, and denoted with the
\emph{AltB} suffix.
These bounds are computed for any choice of \(p,q\in\{2,\infty\}\).
\begin{sidewaystable*}[!htbp]
\small
 \centering
  \begin{tabular}{l|l|r|rrrr}
\toprule
 \multicolumn{2}{c|}{Analysis from}     &
 \cite{Schwartz:2024aa,BBDLS16,bini:1980,demmel:2007a}\footnote{\cite{bini:1980,demmel:2007a} achieve a slightly smaller~\(Q_0\) constant by assuming all additions are performed following a balanced tree, instead of a worst case estimate as done in all other formulas.}
& \multicolumn{4}{c}{Here}\\
\midrule
  \multicolumn{2}{c|}{\((p,q)\)} &  \((\infty,\infty)\) & \((\infty,\infty)\) &  \((2,2)\) & \((\infty,2)\) & \((2,\infty)\) \\
  \midrule
  \multicolumn{2}{c|}{\(\AF{p}{q}\)}  & \(\GF{\infty}{\infty}^{\ell}\) & \(K\)   &   \(1\)    &  \(1\)         & \(K^{3/2}\) \\
  \midrule
  \multirow{11}{*}{\(\GF{p}{q}\)}
  &  Classic               & 4              &  4    &   2      &  2      & \(5.66\)    \\
  &  Winograd              & 18             &  18   &   14     &  8      &  31.3 \\
  &  Strassen              & 12             &  12   &   10.46  &  6.83   &   17.89\\
  &  <2,2,2> acc                 &                & 17.48 &   10.01  &  5.97   &   27.71\\
  &  Winograd AltB~\cite{BCHKS:2020} &\color{red}{18} & 270   &  116.95  & 108     & 292.47\\
  &  Strassen AltB &\color{red}{12} &  80   &  146.40  & 84      & 137.99\\
  &  <2,2,2> acc AltB            &                &183.54 & 115.38   & 95.55   & 233.43  \\
  & <3,3,6>\cite{smirnov:2013a}         &                & 428   & 289.19   & 90.17   & 1387\\
  & <3,3,6> acc            &                & 134   & 76.95    & 20.00   & 518.16\\
  & <3,3,6> AltB           &                & 30567 & 42774    & 15828   & 87198\\
  & <3,3,6>acc AltB        &                & 7050 & 11526    & 3812   & 25020\\
  \midrule
  \multirow{8}{*}{\(\EF{p}{q}{\ell}\)}
  &  Classic & \(K^2\)    &   \(K^2\) &     \(K\)    &  \(K\)                        & \(K^{5/2}\)    \\
  &  Formula & \((1+ Q_0\lg_kK)K^{\lg_k\gamma}\) & \((1+Q_k)K^{\lg_k\gamma}\) & \((1+Q_1)K^{\lg_k\gamma}\)&\((1+Q_1)K^{\lg_k\gamma}\)&\((1+Q_{k^{3/2}})K^{\lg_k\gamma}\)\\
  &  Winograd& \((1+10\lg_2 K)K^{4.17}\)& \(12.25 K^{4.17} \)    & \(11.77K^{3.81}\) & \(12.43K^{3}\) &  \(12.00K^{4.97}\) \\
  &  Strassen& \((1+8\lg_2K)K^{3.59}\)  & \(10.60 K^{3.59}\)     & \(9.85K^{3.39}\)   & \(10.38K^{2.78}\)  & \(10.51K^{4.17}\) \\
  &  <2,2,2> acc   & &\(17.94 K^{4.13}\)   & \(16.48K^{3.33}\) & \(17.74K^{2.58}\) & \(16.52K^{4.80}\)  \\
  &  Winograd AltB \cite{BCHKS:2020} &\color{red}{\({(1+15\lg_2 K)K^{4.17}}\)}   &  \((1+15\lg_2K)K^{8.08}\)&  \((1+15\lg_2K)K^{6.87}\)&  \((1+15\lg_2K)K^{6.76}\)&  \((1+15\lg_2K)K^{8.20}\) \\
  &  Strassen AltB  &\color{red}{\((1+15\lg_2 K)K^{3.59}\)}   &  \((1+14\lg_2K)K^{6.33}\) &  \((1+14\lg_2K)K^{7.20}\) &  \((1+14\lg_2K)K^{6.40}\) &  \((1+14\lg_2K)K^{7.11}\) \\
  &  <2,2,2> acc AltB &               &  \((1+20\lg_2K)K^{7.52}\) &  \((1+20\lg_2K)K^{6.86}\) &  \((1+20\lg_2K)K^{6.58}\) &  \((1+14\lg_2K)K^{7.87}\) \\
  & <3,3,6>\cite{smirnov:2013a}     &            & \(40.28K^{5.52}\)   & \(40.14K^{5.16}\)   & \(40.14K^{4.09}\)   & \(40.15K^{6.59}\)\\
  & <3,3,6> acc        &            & \(33.23K^{4.46}\)   & \(32.92K^{3.96}\)   & \(34.16K^{2.73}\)   & \(32.82K^{5.69}\)\\
  & <3,3,6> AltB      &            & \((1+49\lg_3K)K^{9.41}\)   & \((1+49\lg_3K)K^{9.71}\)   & \((1+49\lg_3K)K^{8.81} \)  & \((1+49\lg_3K)K^{10.36}\)\\
  & <3,3,6> acc AltB  &            & \((1+53\lg_3K)K^{8.07}\)   & \((1+53\lg_3K)K^{8.52}\)   & \((1+53\lg_3K)K^{7.51} \)  & \((1+53\lg_3K)K^{9.22}\)\\
  \bottomrule
\end{tabular}
  \caption{Comparing growth factors and accuracy bounds for matrix multiplication algorithms based
    on various \(\FMMA{2}{2}{2}{7}\) and \(\FMMA{3}{3}{6}{40}\) schemes under various choices of
    norms. For the sake of clarity, the base case dimension is set to \(k_0=1\) and the negative
    terms are omitted.} \label{tab:accbounds}
\end{sidewaystable*}

\newpage
\section{Growth factor along orbits of matrix multiplication tensor decomposition}\label{sec:numericalStabilityMeasure}
In the footstep of~\cite{bini:1980}, we aim to find matrix product tensor decomposition with improved accuracy achieved by minimizing the growth factor in the orbit of given tensor decomposition.
We present here this strategy applied on two examples: the well-known Strassen multiplication tensor already treated this way in~\cite{jgd:2024:accurate} and Smirnov's~\(\FMMA{3}{3}{6}{40}\) introduced in~\cite{smirnov:2013a} that is not square and whose asymptotic complexity exponent is less than~\(2.7743\).
\begin{remark}
In order to perform  this optimization on a function as smooth as possible, we perform two
consecutive relaxations: first choosing to work with the norms \(p=q=2\) in a similar way as in~\cite{Dai:2023aa}, hence optimizing
\begin{equation}
  \GF{2}{2}=\twonorm{\vectorif{\textstyle\sum_{i=1}^r\twonorm{\row{\mat{L}}{i}}\twonorm{\row{\mat{R}}{i}}|p_{k,i}|}{k}},
  \end{equation}
which can then be bounded from above by
\begin{equation}\label{def:gammas}
\growthfactor = \sum_{i=1}^r \twonorm{\row{\mat{L}}{i}}\twonorm{\row{\mat{R}}{i}}\twonorm{\row{\mat{P}}{i}}.
\end{equation}
From now, in order to simplify forthcoming presentation, we systematically use the term~\emph{growth factor} to designate the above \emph{relaxed} growth factor and we denote it by~\(\growthfactor\).
\end{remark}
\Cref{thm:IsotropiesActTransitivelyOnOptimalAlgorithm} shows that all fast~\(\matrixsize{2}{2}\) matrix product algorithms are in the same orbit under isotropies action introduced in~\Cref{lem:sandwiching}.
This is no more true in general e.g.\ for Smirnov's~\(\FMMA{3}{3}{6}{40}\) tensor decomposition (denoted~\(\tensor{R}\) in the sequel).
Though, \Cref{thm:groot} still applies and even if we do not have the whole picture in this case, we could study the growth factor along the orbit defined by the action of a subgroup of the associated isotropy group on a given tensor decomposition.
\subsection{A subgroup of symmetries impacting the growth factor}
While the tensor rank is invariant under the action described in \Cref{thm:groot} the growth factor is generally not.
As its definition is based on Frobenius norm, some isotropies leave it invariant as stated in the following lemma:
\begin{lemma}\label{lem:actionLeavingGrowthFactorInvariant}
The growth factor~\(\growthfactor\) %
is invariant under the action of the semidirect product~\({{\left({{{\emph{\textsc{so}}}^{\pm}({{\Field}^{\firstdim}})}\times{{\emph{\textsc{so}}}^{\pm}({{\Field}^{\seconddim}})}\times{{\emph{\textsc{so}}}^{\pm}({{\Field}^{\thirddim}})}}\right)}\!\rtimes{\mathfrak{S}_{3}}}\) induced by the special orthogonal group and the permutation group~\(\mathfrak{S}_{3}\).
\end{lemma}
\begin{proof}
By~\Cref{def:FrobeniusInnerProduct}, Frobenius norms are invariant under any orthogonal transformations and so is~\(\growthfactor\) by \Cref{def:gammas}.
\Cref{lem:actionLeavingGrowthFactorInvariant} is then derived from \Cref{eq:isotropy,eq:isotropyActionOnHMRepresentation}.
\end{proof}
As it is useless to consider isotropies leaving the growth factor invariant, we limit our search in the real setting to isotropies of the form described in the following lemma.
\begin{lemma}\label{lem:IwasawaDecomposition}
Only the action of~\({{({\textup{h}_{\firstdim}}\cdot{\textup{p}_{\firstdim}})}\times {({\textup{h}_{\seconddim}}\cdot{\textup{p}_{\seconddim}})}\times {({\textup{h}_{\thirddim}}\cdot{\textup{p}_{\thirddim}})}}\) impacts the growth factor~\(\growthfactor\) with:
\begin{equation}
{{\textup{h}_{s}}=\left\lbrace{\mat{H}_{s}(\rho_{i})=
\begin{smatrix}
	\rho_{1} & 0 &&\cdots&0\\
	0&\rho_{2}&0&\cdots&0\\
	\vdots&\ddots&\ddots&\ddots&\vdots\\
	0&\cdots&0&\rho_{s-1}&0\\
	0&\cdots&&0&\frac{1}{\prod_{i=1}^{s-1}\rho_{i}}
\end{smatrix}\Bigg\vert\ {{\rho}>{0}}}\right\rbrace}
\end{equation}
and
\begin{equation}
	{{\textup{p}_{s}}=\left\lbrace{\mat{P}_{s}(\xi_{ij})=
\begin{smatrix}
	1 & \xi_{11} &&\cdots&\xi_{1(s-1)}\\
	0&1&\xi_{21}&\cdots&\xi_{2(s-2)}\\
	\vdots&\ddots&\ddots&\ddots&\vdots\\
	0&\cdots&0&1&\xi_{(s-1)1}\\
	0&\cdots&&0&1
\end{smatrix}\Bigg\vert\ {\xi_{ij}\in\RR} }\right\rbrace}.
\end{equation}
\end{lemma}
\begin{proof}
\Cref{eq:isotropyActionOnHMRepresentation} shows that the product of any action, say~\(\mat{U}\), by a non-zero scalar affects the growth factor once in~\(\mat{U}\) and once, inverted, in~\(\Inverse{\mat{U}}\), as norms are absolutely homogeneous.
Thus, it is sufficient to consider matrices with determinant~\(1\).
\Cref{lem:actionLeavingGrowthFactorInvariant} states that orthogonal matrices do not have any effect.
From the \textsc{qr} decomposition of any invertible matrices, there remains just the~\(\textsc{r}\)-part~\(({{\textup{h}_{i}}\times{\textup{p}_{i}}})\) of~\({\textsc{psl}^{\pm}\bigl({\RR}^{i}\bigr)}\)'s Iwasawa decomposition in~\Cref{thm:IsotropiesActTransitivelyOnOptimalAlgorithm}.
\end{proof}
We consider now the action of isotropies described in \Cref{lem:IwasawaDecomposition} on Strassen and on Smirnov's tensor decompositions in order to find variants with the smaller possible growth factor~\(\growthfactor\).
\subsection{A heuristic for reducing growth factor along tensor decomposition's orbits}\label{seq:heuristic}
First, given any tensor decomposition~\(\tensor{S}\) (for example taken from~\cite{FMMdb:2019}),
we compute symbolically using any computer algebra system the~\(\GrowthFactor{\IsotropyAction{\Isotropy{g}}{\tensor{S}}}\) with a completely generic isotropy~\(\Isotropy{g}\) in~\({ {({\textup{h}_{\firstdim}}\cdot{\textup{p}_{\firstdim}})}\times {({\textup{h}_{\seconddim}}\cdot{\textup{p}_{\seconddim}})}\times {({\textup{h}_{\thirddim}}\cdot{\textup{p}_{\thirddim}})}}\) described in \Cref{lem:IwasawaDecomposition}.
This gives us a function depending on:
\begin{equation}
\frac{(\firstdim+2)(\firstdim-1)}{2}+\frac{(\seconddim+2)(\seconddim-1)}{2}+\frac{(\thirddim+2)(\thirddim-1)}{2}
\end{equation}
parameters compared to~\({\firstdim^{2}+\seconddim^{2}+\thirddim^{2}}\) without using \Cref{lem:IwasawaDecomposition}.
\par
Then, using any numerical optimization toolbox (for example~\cite{MatlabOptToolbox:2023}), we perform numerical minimization on this function to find a set of parameters, that is an isotropy defining an equivalent tensor decomposition with smaller growth factor.
\par
This numerical step results could be supervised (e.g.\ replacing floating numbers by algebraic numbers or their continued fraction's convergents) using any computer algebra system in order to obtain better-looking output, gain a better control on this process and eventually refine it.
\par
Even if this way we could systematically find tensor decomposition variants using this heuristic, the main drawback of the current state of this process is that we have no guarantee that we reach the minimal possible growth factor for the considered fast matrix multiplication tensor decomposition orbit.
\subsection{Strassen's and Smirnov's more accurate tensor decomposition}\label{seq:OurStrassenAndSmirnov}
First let us explicit considered original tensor decompositions growth factor.
The Smirnov's tensor decomposition growth factor is:
\begin{equation}
\growthfactor(\tensor{R})=\sqrt{{17}\cdot{257}}+\sqrt{{2}\cdot{97}\cdot{131}}+\sqrt{{3}\cdot{11}\cdot{43}}\,\frac{9}{2}\approx{395.03}.
\end{equation}
Using the heuristic presented in \Cref{seq:heuristic}, we find the following isotropy:
\begin{equation}
\Isotropy{d}=
{\begin{bmatrix}1&0&0\\0&1&0\\0&0&1\end{bmatrix}}
\tensorproduct
{\begin{bmatrix}1&0&0\\0&1&0\\0&0&1\end{bmatrix}}
\tensorproduct
	{\begin{bmatrix}
		\frac{1}{4}&0&0&0&0&0\\
		0&2&0&0&0&0\\
		0&0&2&0&0&0\\
		0&0&0&2&0&0\\
		0&0&0&0&\frac{1}{4}&0\\
		0&0&0&0&0&2
	\end{bmatrix}}
\end{equation}
The resulting accurate \textsc{hm}-representation with three matrices
is given in appendix, in~\cref{fig:hm336acc}.
The growth factor of the Smirnov's tensor decomposition variant~\({\IsotropyAction{\Isotropy{d}}{\tensor{R}}}\) is equal to~\({60+18\sqrt{6}}\) that is less than~\(104.091\).
the representation of the variant \(\FMMA{3}{6}{3}{40}\),
\(\FMMA{6}{3}{6}{40}\), as well as their accurate versions can be
found in the \plinopt~\href{https://github.com/jgdumas/plinopt/tree/main/data}{data} directory\footnote{%
\begingroup\setlength{\tabcolsep}{4pt}\begin{tabular}{ccc}
\begin{tabular}{c}
\href{https://github.com/jgdumas/plinopt/blob/main/data/L_3x3x6.sms}{L\_3x3x6.sms},\\
\href{https://github.com/jgdumas/plinopt/blob/main/data/L_3x3x6_accurate.sms}{L\_3x3x6\_accurate.sms},\\
\href{https://github.com/jgdumas/plinopt/blob/main/data/L_3x6x3.sms}{L\_3x6x3.sms},\\
\href{https://github.com/jgdumas/plinopt/blob/main/data/L_3x6x3_accurate.sms}{L\_3x6x3\_accurate.sms},\\
\href{https://github.com/jgdumas/plinopt/blob/main/data/L_6x3x3.sms}{L\_6x3x3.sms},\\
\href{https://github.com/jgdumas/plinopt/blob/main/data/L_6x3x3_accurate.sms}{L\_6x3x3\_accurate.sms},
\end{tabular}&\begin{tabular}{c}
\href{https://github.com/jgdumas/plinopt/blob/main/data/R_3x3x6.sms}{R\_3x3x6.sms},\\
\href{https://github.com/jgdumas/plinopt/blob/main/data/R_3x3x6_accurate.sms}{R\_3x3x6\_accurate.sms},\\
\href{https://github.com/jgdumas/plinopt/blob/main/data/R_3x6x3.sms}{R\_3x6x3.sms},\\
\href{https://github.com/jgdumas/plinopt/blob/main/data/R_3x6x3_accurate.sms}{R\_3x6x3\_accurate.sms},\\
\href{https://github.com/jgdumas/plinopt/blob/main/data/R_6x3x3.sms}{R\_6x3x3.sms},\\
\href{https://github.com/jgdumas/plinopt/blob/main/data/R_6x3x3_accurate.sms}{R\_6x3x3\_accurate.sms},
\end{tabular}&\begin{tabular}{c}
\href{https://github.com/jgdumas/plinopt/blob/main/data/P_3x3x6.sms}{P\_3x3x6.sms},\\
\href{https://github.com/jgdumas/plinopt/blob/main/data/P_3x3x6_accurate.sms}{P\_3x3x6\_accurate.sms},\\
\href{https://github.com/jgdumas/plinopt/blob/main/data/P_3x6x3.sms}{P\_3x6x3.sms},\\
\href{https://github.com/jgdumas/plinopt/blob/main/data/P_3x6x3_accurate.sms}{P\_3x6x3\_accurate.sms}, \\
\href{https://github.com/jgdumas/plinopt/blob/main/data/P_6x3x3.sms}{P\_6x3x3.sms},\\
\href{https://github.com/jgdumas/plinopt/blob/main/data/P_6x3x3_accurate.sms}{P\_6x3x3\_accurate.sms}.
\end{tabular}
\end{tabular}\endgroup}.
Their practical accuracy is presented in~\cref{ssec:54}.
\par
Concerning Strassen's tensor decomposition, its growth factor is~\({12+2\sqrt{2}}\) that is greater than~\(14.828\) and we have introduced in~\cite{jgd:2024:accurate} one of its variant with growth factor~\({{2\sqrt{2}}+\lfrac{16}{\sqrt{3}}}\) that is less than~\(12.066032\).
This variant is defined by the following \textsc{hm}-representation:
\begin{equation}\label{eq:asopt}%
\begin{smatrix}
\frac{\sqrt{3}}{2}&\frac{1}{2}&\frac{1}{2}&\frac{\sqrt{3}}{6}\\
0&0&1&{-\frac{\sqrt{3}}{3}}\\
0&1&0&\frac{\sqrt{3}}{3}\\
0&0&0&{-}\frac{2}{\sqrt{3}}\\
{-}\frac{\sqrt{3}}{2}&{-}\frac{1}{2}&\frac{1}{2}&{-}\frac{\sqrt{3}}{2}\\
{-}\frac{\sqrt{3}}{2}&{-}\frac{1}{2}&\frac{1}{2}&\frac{\sqrt{3}}{6}\\
{-}\frac{\sqrt{3}}{2}&\frac{1}{2}&\frac{1}{2}&{-}\frac{\sqrt{3}}{6}\\
\end{smatrix};\
\begin{smatrix}
0&\frac{2}{\sqrt{3}}&0&0\\
-1&\frac{\sqrt{3}}{3}&0&0\\
0&\frac{\sqrt{3}}{3}&0&-1\\
\frac{1}{2}&{-}\frac{\sqrt{3}}{6}&\frac{\sqrt{3}}{2}&{-}\frac{1}{2}\\
{-}\frac{1}{2}&\frac{\sqrt{3}}{2}&{-}\frac{\sqrt{3}}{2}&{-}\frac{1}{2}\\
\frac{1}{2}&\frac{\sqrt{3}}{6}&\frac{\sqrt{3}}{2}&\frac{1}{2}\\
\frac{1}{2}&\frac{\sqrt{3}}{6}&{-}\frac{\sqrt{3}}{2}&{-}\frac{1}{2}\\
\end{smatrix};\
\Transpose{\begin{smatrix}
\frac{\sqrt{3}}{6}&\frac{1}{2}&\frac{1}{2}&\frac{\sqrt{3}}{2}\\
{-}\frac{\sqrt{3}}{3}&0&-1&0\\
\frac{\sqrt{3}}{3}&-1&0&0\\
\frac{\sqrt{3}}{6}&{-}\frac{1}{2}&{-}\frac{1}{2}&\frac{\sqrt{3}}{2}\\
\frac{\sqrt{3}}{2}&{-}\frac{1}{2}&\frac{1}{2}&\frac{\sqrt{3}}{2}\\
{-}\frac{\sqrt{3}}{6}&{-}\frac{1}{2}&\frac{1}{2}&\frac{\sqrt{3}}{2}\\
{-}\frac{2}{\sqrt{3}}&0&0&0\\
\end{smatrix}}.%
\end{equation}
For this tensor decomposition, more results are available then a variant with smaller growth factor.
\par
The numerical minimization involves~\(6\) indeterminates by~\Cref{lem:IwasawaDecomposition} and suggests that a suitable isotropy to reach a fast matrix product tensor decomposition with minimal~\(\growthfactor\) could be of the form~\({(\mat{U}\times\mat{U}\times\mat{U})}\) (involving only~\(2\) indeterminates).
The following proposition states precisely this possibility.
\begin{proposition}[{\cite[Prop.~15]{jgd:2024:accurate}}] \label{prop:BestGrowthFactor}
Consider the matrices~\({\mat{U}(\rho,\xi)=\MatrixProduct{\mat{H}_{\rho}}{\mat{P}_{\xi}}}\) and the isotropies~\(\Isotropy{g}_{\rho,\xi}\) defined by~\({{\mat{U}(\rho,\xi)}{}^{\times{3}}}\).
The minimal value on the
orbit~\(\IsotropyAction{\Isotropy{g}_{\rho,\xi}}{\tensor{S}}\) of the
growth
factor~\(\growthfactor\bigl({\IsotropyAction{\Isotropy{g}_{\rho,\xi}}{\tensor{S}}}\bigr)\)
is~\({{2\sqrt{2}+\lfrac{16}{\sqrt{3}}}>{12.06603}}\), reached at the point~\({{(\rho,\xi)}={\bigl(\sqrt[4]{\lfrac{4}{3}},{-\lfrac{1}{2}}\bigr)}}\).
\end{proposition}
The algorithm corresponding to the point~\({(\rho,\xi)}\) with
minimal~\(\growthfactor\) on Strassen's orbit is given
in~\Cref{eq:asopt}.

We gather in~\Cref{tab:frobenius} values for~\(\growthfactor\) of some~\(\matrixsize{2}{2}\)-matrix product tensor decompositions, together with the result obtained in~\Cref{prop:BestGrowthFactor}.
In~\Cref{sec:implem}, we compare the implementation of algorithms associated to these tensor decompositions in order to confirm that their numerical accuracy is correlated to their respective~\(\growthfactor\) growth factor.
\par
Now, let gather us in the following section some results about the the~\(\growthfactor\) of our new Strassen's tensor decomposition variant and the lower bound~\(11.755\) on any~\(\growthfactor\) in Strassen's orbit.
\subsection{Upper and lower bounds on more accurate Strassen's tensor decomposition}\label{sec:holder}
We explore in this section some bounds on the norm of each component of an \textsc{hm} representation.
By the multiplicativity of~\(L_{p,q}\) norms (even generalized to negative H\"older conjugates), this will always give alternative bounds on the error, a priori less accurate, but potentially easier to apprehend.
\begin{lemma}[{\cite[Lem.~16]{jgd:2024:accurate}}]\label{lem:equiv}
For any \emph{\textsc{hm}} representation~\(\mathcal{H}\), with matrices~\({\mat{L},\mat{R},\mat{P}}\) in~\({\Field^{r{\times}n}}\), let~\({\gamma_{\mathcal{H}}}\) be its~\(\growthfactor\) growth factor~\({\growthfactor(\mathcal{H})}\), as in~\Cref{def:gammas}.
Then for any strictly positive~\(y\) and~\(z\), we have both:
\begin{gather}\label{eq:normsgammaupper}
\gamma_{\mathcal{H}}\leq\xnorm{\mathcal{H}}{2,3}\leq\Fnorm{\mathcal{H}}\quad\quad\text{and}\\
\max\Bigl\lbrace{r^{1+3z}}\xnorm{\mathcal{H}}{2,{-\frac{1}{z}}};\xnorm{\mat{L}}{2,-\frac{1}{y}}\!{\cdot}\xnorm{\mat{R}}{2,-\frac{1}{z}}\!{\cdot}\xnorm{\mat{P}}{2,\frac{1}{1+y+z}}\Bigr\rbrace
\leq\gamma_{\mathcal{H}.}\label{eq:normsgammalower}
\end{gather}
\end{lemma}
\begin{table}[ht]%
\centering\setlength{\tabcolsep}{3pt}\renewcommand{\arraystretch}{1.2}
\begin{tabular}{lrcc}
\toprule
Algorithm&\multicolumn{1}{c}{\(\growthfactor({\mathcal{H}})\)}&\(\xnorm{\mathcal{H}}{2,3}\)&\(\Fnorm{\mathcal{H}}\)\\
\midrule
Winograd &${7{+}\frac{8}{\sqrt{2}}{+}\frac{9}{\sqrt{3}}}\approx{17.853}$&
$11{+}\frac{8}{\sqrt{2}}{+}\frac{9}{\sqrt{3}}$ & $\sqrt{14}^3$\\
Strassen &$12{+}\frac{4}{\sqrt{2}}\approx{14.828}$&
$2{+}\frac{20}{\sqrt{2}}$ & $\sqrt{12}^3$ \\
Eq.~(\ref{eq:powers}) &  $\frac{75}{8}{+}\frac{4}{\sqrt{2}}\approx{12.203}$& $\frac{125}{32}{+}\frac{4}{\sqrt{2}}{+}\frac{25}{2\sqrt{5}}$   & $\sqrt{\frac{162}{16}}^3$\\
Eq.~(\ref{eq:powrot})  & $\frac{75}{8}{+}\frac{4}{\sqrt{2}}\approx{12.203}$& $\frac{125}{32}{+}\frac{4}{\sqrt{2}}{+}\frac{25}{2\sqrt{5}}$& $\sqrt{10}\sqrt{\frac{162}{16}}\frac{810}{80\sqrt{10}}$ \\
Eq.~(\ref{eq:asopt}) & \color{teal}\bf $\frac{16}{\sqrt{3}}{+}\frac{4}{\sqrt{2}}\approx{12.066}$& \color{teal}\bf $\frac{16}{\sqrt{3}}{+}\frac{4}{\sqrt{2}}$ & $\sqrt{10}^3$ \\
Conventional &  $8.000$& $8$ & $\sqrt{8}^3$ \\
\bottomrule
\end{tabular}
\caption{Illustration of~\Cref{eq:normsgammaupper} on several~\(\mathcal{H}=\HMRepresentation{\mat{L}}{\mat{R}}{\mat{P}}\)}\label{tab:frobenius}%
\end{table}
\Cref{tab:frobenius} gives the Frobenius and~\({(2,3)}\)-norms of each of the three matrices defining the \textsc{hm} representation of several matrix product algorithms, as well as their~\(\growthfactor\) growth factor.
\par
In the following proposition, up to orthogonal transformations, we show that~\(\sqrt{10}\) is the minimum of the Frobenius norm of each of the three \textsc{hm} representation components defining a fast~\(\matrixsize{2}{2}\)-matrix multiplication algorithm.
\begin{proposition}[{\cite[Prop.~17]{jgd:2024:accurate}}]\label{prop:tencubetwonorm}
The minimal product~\(\Fnorm{\mathcal{H}}\) of the three Frobenius
norms of the \emph{\textsc{hm}} representation of any bilinear
algorithm for matrix multiplication with~\(7\) multiplications,
is~\(\sqrt{10}{}^3\).
\end{proposition}
Remark that this lower bound is reached by the algorithm whose \textsc{hm} representation is given in~\Cref{eq:asopt}.
\begin{remark}
Similarly, we have that~\({\bigl(\lfrac{\sqrt[4]{3}}{\sqrt{2}},\lfrac{\sqrt[4]{3}}{\sqrt{6}},\lfrac{\sqrt[4]{3}}{\sqrt{2}},{-}\lfrac{\sqrt[4]{3}}{\sqrt{6}}\bigr)}\)
is a minimum
of~\({\xnormexp{\mat{L}\cdot(\mat{W}\otimes\mat{V})}{2,3}{3}}\) as
in~\Cref{prop:tencubetwonorm}
for~\({\xnormexp{\mat{L}\cdot(\mat{W}\otimes\mat{V})}{2,2}{2}}\).
It turns out that this value
is~\({\lfrac{16}{\sqrt{3}}+2\sqrt{2}}\), the same as the
\(\growthfactor\) growth factor at this point, proving that our upper
bound is reached.
\end{remark}
We now turn to potential lower bounds.
\begin{lemma}[{\cite[Lem.~19]{jgd:2024:accurate}}] \label{lem:minimum}
With~\({\mat{W}=\begin{smatrix}r&x\\0&r^{-1}\end{smatrix}}\),~\({\mat{V}=\begin{smatrix}s&y\\0&s^{-1}\end{smatrix}}\),
with~\(\mat{L}\)
the first component of Strassen's \emph{\textsc{hm}} representation given
in~\Cref{eq:StrassenHMRepresentation} and for any~\({z\geq{0.5171}}\),
the
point:
\[{\bigl(\lfrac{\sqrt[4]{3}}{\sqrt{2}},\lfrac{\sqrt[4]{3}}{\sqrt{6}},\lfrac{\sqrt[4]{3}}{\sqrt{2}},{-}\lfrac{\sqrt[4]{3}}{\sqrt{6}}\bigr)}\]
is a local minimum
of~\({\xnorm{\mat{L}\cdot(\mat{W}\otimes\mat{V})}{2,-\lfrac{1}{z}}}\)
as a function of~\({r, x, s}\) and~\(y\).
\end{lemma}
\begin{corollary}[{\cite[Cor.~20]{jgd:2024:accurate}}] \label{cor:lowerbound}
\({11.7554696<\frac{28}{9}2^{\frac{11}{14}}3^{\frac{5}{7}}}\) is a lower bound for the~\(\growthfactor\) growth factor of an \emph{\textsc{hm}} formula using~\(7\) products.
\end{corollary}
\Cref{cor:lowerbound}, for instance shows that the~\(\growthfactor\) growth factor of the conventional algorithm (for which~\({\growthfactor=8}\)) can not be attained by such fast algorithms.
The corollary also shows that the numerical minimum we obtained~(\({\growthfactor\approx{12.066032}}\)) could not be improved by more than~\(2.6\)\%.
\par
In~\Cref{sec:implem}, we compare the implementation of algorithms associated to these tensor decompositions in order to confirm that their numerical accuracy is correlated to their respective~\(\growthfactor\) growth factor.
\section{Algorithms into practice}\label{sec:implem}
In this section, we present several techniques to lower the number of operations used in our algorithms and thus, lower the constant in complexity bounds and potentially obtain an even better accuracy.
\par
Determining actual complexity bounds requires estimating the number of operations required to implement a given formula.
Considering an \textsc{hm} representation, a direct upper bound can be obtained by: first count the number of coefficients different from~\({0,\pm{1}}\) to upper bound the number of multiplications/divisions; second count the number of non-zero coefficients, minus the number of rows, to get an upper bound on the number of additions/subtractions.
\par
To obtain lower operation counts, we use the following techniques:
first, we select among equivalently accurate algorithms those with
better coefficient patterns: this is presented in~\Cref{ssec:rotations};
second, we factor as much as possible the computations between rows of the \textsc{hm} representations, as in~\Cref{ssec:cancel};
third, we use dependent rows as more opportunities for factorization, as in~\Cref{ssec:kernel}.
We then present some good candidates
and we eventually look at some potential sparse alternative change of basis in~\Cref{ssec:schwartz}.
\subsection{Sparsifying via rotations}\label{ssec:rotations}
We have seen in~\Cref{lem:actionLeavingGrowthFactorInvariant} that orthogonal transformations leave the Frobenius norm invariant and thus, the~\(\growthfactor\) growth factor.
Therefore, one can apply~\(\matrixsize{4}{4}\) generic Kronecker products of orthogonal~\(\matrixsize{2}{2}\) (rotation) matrices using~\Cref{lem:actionOnHMRepresentation} and try to optimize the considered \textsc{hm} representation for several possible goals:
\begin{itemize}
\item
a smaller number of non-zero coefficients in \textsc{hm} representation components (see the techniques of~\Cref{ssec:cancel});
\item
a non-zero pattern better suited to factorization (see more precisely the technique of~\Cref{ssec:cancel});
\item
a triangular (sparse) subset of independent rows (see more precisely the technique of~\Cref{ssec:kernel}).
\end{itemize}
\par
For instance, to obtain~\Cref{eq:asopt}, we solve for the minimal values of the Frobenius norms as in~\Cref{prop:tencubetwonorm} and then for orthogonal transformations that produce as many vectors of the canonical basis as possible.
Doing so, we found that with~\(\growthfactor\) set to~\({\lfrac{16}{\sqrt{3}}+2\sqrt{2}}\) and \textsc{hm} representation component Frobenius norms set to~\(\sqrt{10}\), the maximal possible number of canonical vectors was~$1$.
\Cref{eq:asopt} is one of those.
Similarly, \Cref{eq:powrot} is an orthogonal optimization of~\Cref{eq:powers}, with one canonical vector in each of components of the \textsc{hm} representation.
A \textsc{c}++ implementation of these optimizing tools is available in the \plinopt~library~\cite{jgd:2024:plinopt}.

\subsection{Cancellation-free search to reduce additions and multiplications}\label{ssec:cancel}
For the implementation of a given linear operator (in this work one of
the matrices in the \textsc{hm} representation) one can try to find
the shortest straight-line program for its computation.
The problem is \textsc{np}-hard in general (see
e.g.~\cite[\S~3.1.1]{Boyar:2013aa}); but for small matrices, and over
the field with~\(2\) elements,~\cite{Boyar:2013aa} and references
therein propose several heuristics that potentially reduce the number
of operations.
\par
Not all of them are applicable to fields with more elements, but we
use a kind of common sub-expression eliminations, the
``cancellation-free'' search, described in~\Cref{alg:factoringout} and
implemented in
\href{https://github.com/jgdumas/plinopt/blob/main/optimizer.cpp}{\texttt{plinopt/optimizer\,-D}}~\cite{jgd:2024:plinopt}.

The first goal of our algorithms is to deal with additions.
There, finding the shortest linear straight-line program is
hard, in part since potentially all minors could have to be
explored~\cite{Morgenstern:1975:Linear}.
Thus in practice we here only explore the extremal relations, namely
$2{\times}2$ minor or whole row/column dependencies:
\begin{itemize}
\item Zero $2{\times}2$ minors identifies a common addition that can
  be eliminated, as shown in~\Cref{ssec:cancel};
\item Selecting a subset of rows or columns of maximal rank implies
  that the other rows/columns can be obtained by, potentially sparse,
  linear combinations of this full rank subset, as shown
  in~\Cref{ssec:kernel}.
\end{itemize}

The second goal is will then be to factor out
multiplications/divisions by the same
constants.

First, \Cref{alg:factoringout} describes a common sub-expression
elimination heuristic that reduced the number of operations in our
algorithms: both in terms of additions/subtractions and
multiplications/divisions. We want to produce a straight-line program
that realizes the product of a given matrix $M$ by any vector.
\Cref{alg:factoringout} works by extracting common computations from
$M$: these common operations are computed in new temporary variables
while~\(M\) is extended with new columns corresponding to these
temporaries.
That way the remaining $M$ is larger but sparser and can be directly
transformed to a very simple straight-line program.

\paragraph{Reducing the number of additions with a co-linear multiple}
Classical common sub-expression elimination on a matrix-vector product can be applied to the matrix as follows: if a~\({{2}\times{2}}\) minor is zero, it means that two \mbox{dimension-$2$}
subvectors are colinear,~\({\begin{bmatrix}c&d\end{bmatrix}=\lambda\begin{bmatrix}a&b\end{bmatrix}}\).
They therefore recompute the same addition, potentially up to a constant factor.
Thus we replace the corresponding minor~\({\begin{bmatrix} a& b\\c & d\end{bmatrix}\Transpose{\begin{bmatrix} u & v\end{bmatrix}}}\) by instead~\({\begin{bmatrix} 0&0&1 \\ 0 & 0 & \lambda \end{bmatrix}\Transpose{\begin{bmatrix}u&v&w\end{bmatrix}}}\) with a new variable~\({w=au+bv}\).
This new variable is used to store the first computation of one addition~\({w=au+bv}\), while the new column allows to reuse this computation for the other rows (thus saving one addition overall, as now~\({\lambda{w}=\lambda(au+bv)}\), that is~\({\lambda{a}u+\lambda{b}v=cu+dv}\)).
\paragraph{Reducing the number of multiplications up to signs}
Second, \Cref{alg:factoringout} describes savings on
multiplications/divisions by constants.
The idea is to identify identical constants in rows or columns,
but also composed by row and columns (which we denote by a
\emph{triangle reduction}):
\begin{itemize}
\item Reducing the number of multiplications by columns: replace the product~\({\begin{bmatrix}a\\\vdots\\\pm{a}\end{bmatrix}\Transpose{\begin{bmatrix} u \end{bmatrix}}}\) by~\({\begin{bmatrix}0&1\\ \vdots&\vdots\\0&\pm{1}\end{bmatrix}\Transpose{\begin{bmatrix} u&w \end{bmatrix}}}\) with a new variable~\({w=au}\) associated to a new column of the matrix-vector product.
\item Reducing the number of multiplications by rows:
replace the product~\({\begin{bmatrix}a&\pm{a}\end{bmatrix}\Transpose{\begin{bmatrix} u & v\end{bmatrix}}}\) by instead~\({\begin{bmatrix}0&0&a\end{bmatrix}\Transpose{\begin{bmatrix}u&v&w\end{bmatrix}}}\) with a new variable~\({w=u\pm{v}}\).
This saves one multiplication.
Note that this should be performed only after the previous step
of reducing the number of additions with all co-linear multiples:
this ensures that, at that point, $w=u+v$ is present only once in the
remaining matrix-vector product.
\item Reducing the number of multiplications by triangles:
for $a\neq\pm{1}$, $b\neq\pm{1}$ and $c\neq\pm{1}$,
replace
$\begin{bmatrix}\pm{ab}& b\\a&c\end{bmatrix}\Transpose{\begin{bmatrix} u & v\end{bmatrix}}$
by
$\begin{bmatrix}0&b&\pm{b}\\0&c&1\end{bmatrix}\Transpose{\begin{bmatrix}u&v&w\end{bmatrix}}$
with a new variable $w=au$.
Of course then proceed with a reduction by row of $b$, as previously,
Overall these two reductions also save a single multiplication.
Again, this should also be performed only after the previous step
of reducing the number of additions with a co-linear multiple (this
ensures, e.g., that $c$ can in fact never be $\pm{1}$).
\end{itemize}

This is summarized in~\Cref{alg:factoringout} thereafter.

\begin{algorithm}[htbp]
\caption{Cancellation-free optimization of a linear operator}\label{alg:factoringout}
\begin{algorithmic}[1]
\Require{$\mat{M}\in\Field^{m{\times}n}$.}
\Ensure{A straight-line program computing $x\rightarrow{\mat{M}}{\cdot}x$.}
\Repeat\hfill\Comment{Reduce additions by transforming colinear pairs}
\State{In each row, list all pairs of indices of non-zero coefficients;}
\State{Among all the rows, find the pair(s) with the maximal number of co-linear representatives;}
\State{In case of ties, exhaust all the possibilities with maximal
  pairs (or choose one pair $\lbrace{i,j}\rbrace$ using a score like
  that of~\cite[\S~3.2]{Boyar:2013aa});}
\State{Precompute the chosen pair with the variables of the columns
  $i$ and $j$, and factor it out of all the rows.
  That is, anywhere applicable, replace
  $\begin{bmatrix}m_{r,i}&m_{r,j}\\\lambda{m_{r,i}}&\lambda{m_{r,j}}\end{bmatrix}\Transpose{\begin{bmatrix}x_i&x_j\end{bmatrix}}$
  by instead
  $\begin{bmatrix}0&0&1\\0&0&\lambda\end{bmatrix}\Transpose{\begin{bmatrix}x_i&x_j&x_{n+1}\end{bmatrix}}$
  using a new variable $x_{n+1}=m_{r,i}x_i+m_{r,j}x_j$;}
\Until{no pair has more than $1$ representative}
\Statex\Comment{Multipliers by columns:}
\ForAll{equal coefficients in a column $j$ (up to sign)}
\State{Compute the product by the absolute value in a temporary
  variable: if $m_{r,i}=\pm{m_{i,j}}$, then let $x_{n+1}=m_{i,j}x_j$;}
\State{Factor this coefficient out: replace the product~\({\begin{bmatrix}m_{i,j}\\m_{r,i}\end{bmatrix}\Transpose{\begin{bmatrix}x_j\end{bmatrix}}}\) by~\({\begin{bmatrix}0&1\\0&\pm{1}\end{bmatrix}\Transpose{\begin{bmatrix}x_j&x_{n+1}\end{bmatrix}}}\);}
\EndFor
\Statex\Comment{Multipliers by triangle relations:}
\ForAll{Triangle relations $m_{i,j}=\pm{m_{i,k}m_{r,j}}$}
\State{Compute the sum (or subtraction) of variables with that same
  coefficients in a temporary variable: $x_{n+1}=m_{r,j}x_j$;}
\State{Factor the coefficient out:
replace
$\begin{bmatrix}m_{i,j}&m_{i,k}\\m_{r,j}&\star\end{bmatrix}\Transpose{\begin{bmatrix}x_j&x_k\end{bmatrix}}$
by instead
$\begin{bmatrix}0&m_{i,k}&\pm{m_{i,k}}\\0&\star&1\end{bmatrix}\Transpose{\begin{bmatrix}x_j&x_k&x_{n+1}\end{bmatrix}}$;
}
\EndFor
\Statex\Comment{Multipliers by rows:}
\ForAll{equal coefficients in a row $i$ (up to sign)}
\State{Compute the sum (or subtraction) of variables with that same
  coefficients in a new temporary variable: if $m_{i,k}=\pm{m_{i,j}}$,
    then let $x_{n+1}=x_j\pm{x_k}$;}
\State{Factor the coefficient out:
replace
$\begin{bmatrix}m_{i,j}&m_{i,k}\end{bmatrix}\Transpose{\begin{bmatrix}x_j&x_k\end{bmatrix}}$
by instead
$\begin{bmatrix}0&0&m_{i,j}\end{bmatrix}\Transpose{\begin{bmatrix}x_j&x_k&x_{n+1}\end{bmatrix}}$;
}
\EndFor
\Statex\Comment{Now the matrix has been simplified}
\State{Apply the remaining linear operations of the matrix to produce
  the remaining simple \textsc{slp}.}
\end{algorithmic}
\end{algorithm}
\subsection{Kernel computation and Tellegen's principle}\label{ssec:kernel}
If the rank of the linear operator is lower than its number of rows, then an additional strategy has proven useful:
compute first some independent rows, then express the dependent ones by their linear relations.
For this, \Cref{alg:kernel} computes a left kernel of the linear operator and uses it to compute the dependent rows via linear combinations of the independent ones.
This is sometimes faster than directly computing the dependent rows.
Of course, if the matrix's rank is lower than the number of \emph{columns}, one can apply~\Cref{alg:kernel} to the transposed matrix and then apply the \emph{Tellegen's transposition} principle to recover the transposed linear dependencies (e.g.\ see~\cite{bostan:2003} and references therein).
\begin{algorithm}[ht]
\caption{Kernel decomposition of a linear operator}\label{alg:kernel}
\begin{algorithmic}[1]
\Require{\(\mat{M}\) in~\(\Field^{m{\times}n}\) such that~\({r=\MatrixRank{\mat{M}}}\).}
\Ensure{A straight line program computing~\(\vec{u}\leftarrow{\mat{M}}{\cdot}\vec{v}\).}
\State{By Gaussian elimination, compute~\({\mat{M}={\mat{P}\cdot\mat{L}\cdot\mat{U}\cdot\mat{Q}}}\) with~\(\mat{P}\) a permutation matrix,~\(\mat{L}\) in~\(\Field^{m{\times}r}\) be~\({\begin{bmatrix}\mat{L}_1\\ \mat{L}_2\end{bmatrix}}\) unit upper triangular and~\(\mat{L}_1\) in~\(\Field^{r{\times}r}\); choosing~\(\mat{P}\) so that~(1) the first~\(r\) rows of~\(\mat{P}^{-1}\mat{M}\) are sparsest;~(2)~\(\mat{L}_1\) is the sparsest;~(3)~\(\mat{L}_2\) is the sparsest;}
\State{Let~\(\sigma\) be the permutation represented by~\(\mat{P}\);}
\State{Apply Alg.~\ref{alg:factoringout} to~\({\Transpose{\begin{bmatrix}u_{\sigma(1)}\ldots{u_{\sigma(r)}}\end{bmatrix}}\leftarrow\begin{bmatrix}I_r 0\end{bmatrix}\cdot\mat{P}\cdot\mat{M}\cdot\vec{v}}\);}
\Statex\Comment{\(\begin{bmatrix}-{\mat{L}_{2}}\cdot{{\mat{L}_{1}}^{\!{-1}}}&\mat{I}_{m-r}\end{bmatrix}\) is a (sparse) left kernel of~\(\mat{M}\) and provides the linear dependencies of the remaining rows}
\State{Apply Alg.~\ref{alg:factoringout} to~\({\Transpose{\begin{bmatrix}u_{\sigma(r+1)}\ldots{u_{\sigma(m)}}\end{bmatrix}}\leftarrow{\mat{L}_{2}}\cdot{{\mat{L}_{1}}^{\!{-1}}}\Transpose{\begin{bmatrix}u_{\sigma(1)}\ldots{u_{\sigma(r)}}\end{bmatrix}}}\).}
\end{algorithmic}
\end{algorithm}
\par
\Cref{alg:kernel} is implemented in \href{https://github.com/jgdumas/plinopt/blob/main/optimizer.cpp}{\texttt{plinopt/optimizer\,-K}}.
The Tellegen's transposition principle applied to such straight-line programs is implemented in \href{https://github.com/jgdumas/plinopt/blob/main/transpozer.cpp}{\texttt{plinopt/transpozer}}.
The implementations can be found in~\cite{jgd:2024:plinopt,jgd:2024:mFMM}.
\par
We have applied these heuristics to some matrix multiplication
algorithms and present the obtained reductions in~\Cref{tab:CSE}.
\begin{table}[htbp]\centering
\begin{tabular}{lrrr}
\toprule
\multicolumn{1}{c}{Alg.}  & Original &  \cite[Table~3]{Benson:2015:framework} & \plinopt, here\\
\midrule
3${\times}$3${\times}$3 (Laderman) &    98 &    - &    62\\
3${\times}$3${\times}$3 (Grey-152) &    97 &    70 &    63\\
4${\times}$2${\times}$4 (Grey-257) &   189 &   138 &    97\\
4${\times}$3${\times}$2 (Grey-144) &    96 &    72 &    62\\
4${\times}$3${\times}$3 (Grey-234) &   164 &   125 &    98\\
5${\times}$2${\times}$2 (Grey-99) &    53 &    43 &    40\\
\bottomrule
\end{tabular}
\caption{Number of Additions for some matrix multiplications}\label{tab:CSE}
\end{table}

For instance, the considered~\({{4}{\times}{2}{\times}{4}}\) algorithm in~\Cref{tab:CSE} has tensor rank~\(26\).
Its \textsc{hm} representation has~\(L\) a~\({{26}\times{8}}\) matrix with $94$ non-zeroes (thus naively realizable
with $94-26=68$ additions),
$R$ a $26{\times}8$ matrix with $78$ non-zeroes (thus naively realizable
with $78-26=52$ additions),
$P$ a $16{\times}26$ matrix with $85$ non-zeroes (thus naively realizable
with $85-16=69$ additions).
The original algorithm  requires~\({(94{-}26){+}(78{-}26){+}(85{-}16){=}68{+}52{+}69{=}189}\) additions.
\plinopt~can reduce the computation of $L$ to $28$
additions\footnote{\(28\) additions are achieved by applying~\Cref{alg:kernel} on $\Transpose{L}$: we then obtain an algorithm with~\(46\) additions. Then transposing this algorithm via Tellegen's principle produces the~\({46-(26-8)=28}\) additions to realize~\(L\). Note also that while using only \Cref{alg:factoringout} we were not able to reduce below~\(29\) additions for~\(L\).},
that of~\(R\) to~\(22\) additions and that of~\(P\) to~\(47\)
additions. This is
a total of only~\({{28+22+47}={97}}\) additions.
\par
Note that these original~\({L,R}\) and~\(P\) matrices have respectively~\({6,0}\) and~\(32\) non-unit coefficients whose values are all~\(\pm\frac{1}{2}\).
\plinopt~reduces the number of divisions by~\(2\) to be respectively only~\({2,0}\) and~\(8\).
\subsection{Numerical accuracy}\label{ssec:optim}
With \plinopt, we have thus produced short implementations for our
different \textsc{hm} formulas.
We have for instance the implementation in~\Cref{alg:asopt}
of~\Cref{eq:asopt} using only~\(24\) additions and~\(12\)
multiplications/divisions.

\begin{table}[htb]%
\centering\fbox{\begin{minipage}{.95\linewidth}%
\[\setlength\arraycolsep{2pt}\renewcommand{\arraystretch}{1}
\begin{array}{llll}
t_1=\frac{\sqrt{3}}{3}a_{22}& t_2=a_{21}+t_1&s_1=\frac{\sqrt{3}}{3}b_{21}& s_2=s_1-b_{11}\\
t_3=a_{12}+t_2& l_1=\frac{\sqrt{3}}{2}a_{11}+\frac{1}{2}t_3&s_3=s_2+b_{22}& r_1=2s_1\\
l_2=a_{12}-t_1& l_3=t_2&r_2=s_2& r_3=s_1-b_{22}\\
l_4=2t_1& l_5=l_2-l_1&r_4=\frac{1}{2}s_3{-}\frac{\sqrt{3}}{2}b_{12}& r_5=r_3+r_4\\
l_6=l_5+l_4& l_7=l_5+l_3&r_6=r_1-r_5& r_7=r_5-r_2
\end{array}
\]
\[%
\begin{array}{llll}
p_1=l_1{\cdot}r_1&
p_2=l_2{\cdot}r_2&
p_3=l_3{\cdot}r_3&
p_4=l_4{\cdot}r_4\\
p_5=l_5{\cdot}r_5&
p_6=l_6{\cdot}r_6&
p_7=l_7{\cdot}r_7
\end{array}
\]
\[\setlength\arraycolsep{2pt}\renewcommand{\arraystretch}{1}
\begin{array}{llll}
w_2=p_5+p_1+p_6&
w_1=p_7+p_6&
w_3=w_2-p_2&
w_5=\frac{p_4+w_2}{2}\\
c_{12}=p_1-p_3-w_5&
c_{21}=w_3-w_5&
c_{22}=\sqrt{3}w_5\\
\multicolumn{4}{c}{c_{11}=\frac{\sqrt{3}}{3}(w_3-c_{12}-2w_1)}
\end{array}
\]
\end{minipage}}
\caption{\textsc{slp} of~\Cref{eq:asopt} with~\(24\) add.\ and~\(12\) mul./div.}\label{alg:asopt}
\end{table}

We show the accuracy of this variant by comparing it to the classical
variants in \Cref{fig:gammaunif}, first using a uniform~\([-1,1]\)
distribution
(next, in~\Cref{fig:gamma} we use a normal distribution).
We observe that our variant is one or two orders of magnitude more
accurate, and is quite close to that of the conventional algorithm.

\begin{figure}[!ht]\centering
\caption{Numerical accuracy vs size (uniform [-1,1]
  distribution)}\label{fig:gammaunif}
\begingroup
  \makeatletter
  \providecommand\color[2][]{%
    \GenericError{(gnuplot) \space\space\space\@spaces}{%
      Package color not loaded in conjunction with
      terminal option `colourtext'%
    }{See the gnuplot documentation for explanation.%
    }{Either use 'blacktext' in gnuplot or load the package
      color.sty in LaTeX.}%
    \renewcommand\color[2][]{}%
  }%
  \providecommand\includegraphics[2][]{%
    \GenericError{(gnuplot) \space\space\space\@spaces}{%
      Package graphicx or graphics not loaded%
    }{See the gnuplot documentation for explanation.%
    }{The gnuplot epslatex terminal needs graphicx.sty or graphics.sty.}%
    \renewcommand\includegraphics[2][]{}%
  }%
  \providecommand\rotatebox[2]{#2}%
  \@ifundefined{ifGPcolor}{%
    \newif\ifGPcolor
    \GPcolortrue
  }{}%
  \@ifundefined{ifGPblacktext}{%
    \newif\ifGPblacktext
    \GPblacktexttrue
  }{}%
  \let\gplgaddtomacro\g@addto@macro
  \gdef\gplbacktext{}%
  \gdef\gplfronttext{}%
  \makeatother
  \ifGPblacktext
    \def\colorrgb#1{}%
    \def\colorgray#1{}%
  \else
    \ifGPcolor
      \def\colorrgb#1{\color[rgb]{#1}}%
      \def\colorgray#1{\color[gray]{#1}}%
      \expandafter\def\csname LTw\endcsname{\color{white}}%
      \expandafter\def\csname LTb\endcsname{\color{black}}%
      \expandafter\def\csname LTa\endcsname{\color{black}}%
      \expandafter\def\csname LT0\endcsname{\color[rgb]{1,0,0}}%
      \expandafter\def\csname LT1\endcsname{\color[rgb]{0,1,0}}%
      \expandafter\def\csname LT2\endcsname{\color[rgb]{0,0,1}}%
      \expandafter\def\csname LT3\endcsname{\color[rgb]{1,0,1}}%
      \expandafter\def\csname LT4\endcsname{\color[rgb]{0,1,1}}%
      \expandafter\def\csname LT5\endcsname{\color[rgb]{1,1,0}}%
      \expandafter\def\csname LT6\endcsname{\color[rgb]{0,0,0}}%
      \expandafter\def\csname LT7\endcsname{\color[rgb]{1,0.3,0}}%
      \expandafter\def\csname LT8\endcsname{\color[rgb]{0.5,0.5,0.5}}%
    \else
      \def\colorrgb#1{\color{black}}%
      \def\colorgray#1{\color[gray]{#1}}%
      \expandafter\def\csname LTw\endcsname{\color{white}}%
      \expandafter\def\csname LTb\endcsname{\color{black}}%
      \expandafter\def\csname LTa\endcsname{\color{black}}%
      \expandafter\def\csname LT0\endcsname{\color{black}}%
      \expandafter\def\csname LT1\endcsname{\color{black}}%
      \expandafter\def\csname LT2\endcsname{\color{black}}%
      \expandafter\def\csname LT3\endcsname{\color{black}}%
      \expandafter\def\csname LT4\endcsname{\color{black}}%
      \expandafter\def\csname LT5\endcsname{\color{black}}%
      \expandafter\def\csname LT6\endcsname{\color{black}}%
      \expandafter\def\csname LT7\endcsname{\color{black}}%
      \expandafter\def\csname LT8\endcsname{\color{black}}%
    \fi
  \fi
    \setlength{\unitlength}{0.0500bp}%
    \ifx\gptboxheight\undefined%
      \newlength{\gptboxheight}%
      \newlength{\gptboxwidth}%
      \newsavebox{\gptboxtext}%
    \fi%
    \setlength{\fboxrule}{0.5pt}%
    \setlength{\fboxsep}{1pt}%
    \definecolor{tbcol}{rgb}{1,1,1}%
\begin{picture}(7560.00,4520.00)%
    \gplgaddtomacro\gplbacktext{%
      \csname LTb\endcsname
      \put(219,591){\makebox(0,0)[r]{\strut{}$10^{-14}$}}%
      \csname LTb\endcsname
      \put(219,1527){\makebox(0,0)[r]{\strut{}$10^{-13}$}}%
      \csname LTb\endcsname
      \put(219,2463){\makebox(0,0)[r]{\strut{}$10^{-12}$}}%
      \csname LTb\endcsname
      \put(219,3399){\makebox(0,0)[r]{\strut{}$10^{-11}$}}%
      \csname LTb\endcsname
      \put(401,174){\makebox(0,0){\strut{}$32$}}%
      \csname LTb\endcsname
      \put(2776,174){\makebox(0,0){\strut{}$64$}}%
      \csname LTb\endcsname
      \put(5151,174){\makebox(0,0){\strut{}$128$}}%
      \csname LTb\endcsname
      \put(7526,174){\makebox(0,0){\strut{}$256$}}%
    }%
    \gplgaddtomacro\gplfronttext{%
      \csname LTb\endcsname
      \put(728,4081){\makebox(0,0)[l]{\strut{}Winograd \cite{Winograd:1977:complexite}}}%
      \csname LTb\endcsname
      \put(728,3733){\makebox(0,0)[l]{\strut{}Strassen \cite{strassen:1969}}}%
      \csname LTb\endcsname
      \put(3433,4081){\makebox(0,0)[l]{\strut{}\Cref{alg:asopt,eq:asopt}}}%
      \csname LTb\endcsname
      \put(3433,3733){\makebox(0,0)[l]{\strut{}Conventional}}%
      \csname LTb\endcsname
      \put(0,4167){\makebox(0,0){\strut{}error}}%
      \csname LTb\endcsname
      \put(3916,0){\makebox(0,0){\strut{}square matrix dimension}}%
    }%
    \gplbacktext
    \put(0,0){\includegraphics[width={378.00bp},height={226.00bp}]{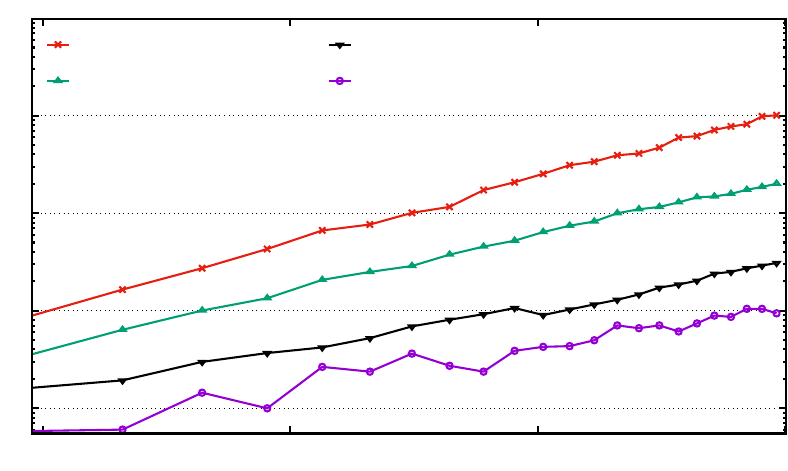}}%
    \gplfronttext
  \end{picture}%
\endgroup
\end{figure}

\subsection{Rational approximations}\label{sec:rationalapproximations}

Now, in~\cite{bini:1980} the authors consider all bilinear algorithms
using~\(7\) multiplications with constants of the form~\(\pm{2^i}\);
they showed that Strassen's original method~\cite{strassen:1969}
reaches in this class the minimum value~\(12\) of
their~\(\gamma_{0,1\infty}\) factor error bound (while for instance
that of Winograd~\cite{Winograd:1977:complexite} is~\(18\), see
also~\Cref{tab:accbounds}).
\par
We thus also propose here an algorithm in this class that has a
worse~\(\gamma_{0,1\infty}\) of~\(40\), but a~\(\growthfactor\)
of~\({2\sqrt{2}+\lfrac{75}{8}\approx{12.2034}}\), better
than those of Strassen~\(14.8284\) or Winograd~\(17.8530\)
(see~\Cref{tab:frobenius}). This algorithm is defined by the
following~\textsc{hm} representation:
\begin{equation}\label{eq:powers}
\begin{bmatrix}
0&-1&1&0\\
1&\frac{1}{2}&{-}\frac{1}{2}&{-}\frac{1}{4}\\
0&0&1&{-}\frac{1}{2}\\
0&1&0&{-}\frac{1}{2}\\
0&0&1&\frac{1}{2}\\
1&{-}\frac{1}{2}&\frac{1}{2}&{-}\frac{1}{4}\\
0&1&0&\frac{1}{2}\\
\end{bmatrix};\quad
\begin{bmatrix}
1&0&0&-1\\
1&\frac{1}{2}&0&0\\
0&\frac{1}{2}&0&-1\\
\frac{1}{2}&\frac{1}{4}&-1&{-}\frac{1}{2}\\
0&\frac{1}{2}&0&1\\
1&{-}\frac{1}{2}&0&0\\
\frac{1}{2}&{-}\frac{1}{4}&1&{-}\frac{1}{2}\\
\end{bmatrix};\quad
\Transpose{\begin{bmatrix}
0 & 1 & 1 & 0 \\
\frac{1}{2} & 1 & 0 & 0 \\
\frac{1}{4} & {-}\frac{1}{2} & {-}\frac{1}{2} & 1 \\
{-}\frac{1}{2} & 0 & 1 & 0 \\
\frac{1}{4} & \frac{1}{2} & \frac{1}{2} & 1 \\
\frac{1}{2} & -1 & 0 & 0 \\
\frac{1}{2} & 0 & 1 & 0 \\
\end{bmatrix}}.
\end{equation}

The induced algorithm produced by \plinopt~is given in~\Cref{alg:powers}.
\begin{table}[!ht]%
\fbox{%
\begin{minipage}{.95\linewidth}%
\[\setlength\arraycolsep{4pt}\renewcommand{\arraystretch}{1}
\begin{array}{llll}
r_{1} = \frac{1}{2}a_{22}& t_{2} = a_{21}-r_{1} &
u_{1} = \frac{1}{2}b_{12}& s_{1} = b_{11}+u_{1}\\
t_{3} = a_{12}+r_{1}&t_{0}=t_{2}-t_{3}&
s_{2} = u_{1}-b_{22}& u_{2} = s_{1}-b_{22}\\
t_{4} = a_{21}+r_{1}& r_{2}=t_{2}-a_{12}&
s_{4} = b_{22}+u_{1}&s_{0}=s_{1}-s_{4} \\
t_{5} = a_{11}+\frac{1}{2}r_{2} & t_{1}=t_{5}-t_{0}&
s_{3} = \frac{1}{2}u_{2}-b_{21} & s_{5}=s_{0}-s_{2}
\end{array}
\]
\[%
\begin{array}{llll}
p_{1}=t_{0}{\cdot}s_{0}&
p_{2}=t_{1}{\cdot}s_{1}&
p_{3}=t_{2}{\cdot}s_{2}&
p_{4}=t_{3}{\cdot}s_{3}\\
p_{5}=t_{4}{\cdot}s_{4}&
p_{6}=t_{5}{\cdot}s_{5}&
\multicolumn{2}{l}{p_{7}=(t_{4}{-}t_{0}){\cdot}(s_{0}{-}s_{3})}
\end{array}
\]
\[\setlength\arraycolsep{4pt}\renewcommand{\arraystretch}{1}
\begin{array}{llll}
c_{22}{=}p_{5}{+}p_{3}&
v_{1}{=}p_{1}{-}p_{6}{-}p_{3}&
v_{2}{=}p_{7}{+}p_{6}&
v_{3}{=}p_{4}{+}v_{1}\\
v_{4}{=}\frac{c_{22}}{2}&
c_{12}{=}p_{2}{+}v_{1}{+}v_{4}&
c_{21}{=}v_{2}{+}v_{3}{+}v_{4}&
c_{11}{=}\frac{(c_{12}+v_{2}-v_{3})}{2}
\end{array}
\]
\end{minipage}
}
\caption{\textsc{slp} of~\Cref{eq:powers} with~\(27\) add.,~\(6\) div.\ by~\(2\) and~\({\gamma_{F}\!\approx\!{12.2034}}\)}\label{alg:powers}
\end{table}

In fact, \Cref{eq:powers} was obtained by approximating the minimal
point of the~\(\growthfactor\) growth factor taken from
\Cref{prop:BestGrowthFactor} with the smallest powers of~\(2\).
Further rational higher-order approximations can be obtained in the
same vein, giving for
instance~\Cref{eq:powrot,alg:powrot,alg:0695,alg:0661}, as follows.

First, \Cref{eq:powrot} is an orthogonal optimization
of~\Cref{eq:powers}, with one canonical vector in each of components
of the \textsc{hm} representation.
\begin{equation}\label{eq:powrot}
\begin{smatrix}
\frac{4}{9}&{-}\frac{8}{9}&{-}\frac{8}{9}&{-}\frac{4}{9}\\
0&\frac{5}{9}&0&\frac{10}{9}\\
\frac{8}{9}&{-}\frac{2}{3}&0&0\\
\frac{4}{9}&\frac{2}{9}&\frac{8}{9}&\frac{4}{9}\\
0&{-}\frac{10}{9}&0&0\\
\frac{4}{9}&{-}\frac{1}{3}&{-}\frac{8}{9}&\frac{2}{3}\\
{-}\frac{4}{9}&{-}\frac{2}{9}&\frac{8}{9}&\frac{4}{9}\\
\end{smatrix};\
\begin{smatrix}
{-}\frac{3}{5}&\frac{4}{5}&{-}\frac{4}{5}&{-}\frac{3}{5}\\
0&\frac{1}{2}&0&-1\\
-1&\frac{1}{2}&0&0\\
0&\frac{5}{4}&0&0\\
\frac{3}{5}&{-}\frac{3}{10}&\frac{4}{5}&{-}\frac{2}{5}\\
\frac{2}{5}&\frac{3}{10}&{-}\frac{4}{5}&{-}\frac{3}{5}\\
{-}\frac{3}{5}&{-}\frac{9}{20}&{-}\frac{4}{5}&{-}\frac{3}{5}\\
\end{smatrix};\
\Transpose{\begin{smatrix}
\frac{9}{20}&\frac{9}{10}&\frac{9}{10}&{-}\frac{9}{20}\\
0&0&\frac{27}{40}&{-}\frac{9}{10}\\
{-}\frac{9}{8}&0&{-}\frac{9}{16}&0\\
\frac{9}{20}&\frac{9}{10}&\frac{9}{40}&\frac{9}{20}\\
{-}\frac{27}{40}&\frac{9}{10}&\frac{27}{80}&{-}\frac{9}{20}\\
0&0&{-}\frac{9}{8}&0\\
\frac{9}{20}&\frac{9}{10}&{-}\frac{9}{40}&{-}\frac{9}{20}
\end{smatrix}}.
\end{equation}

Unfortunately some small non-powers of~\(2\) are then unavoidable, but
this gives in~\Cref{alg:powrot} an algorithm realizing the formula
with fewer additions than that of~\Cref{alg:powers}.
\begin{table}[ht]%
\centering\fbox{\begin{minipage}{.95\linewidth}\vspace{-5pt}
\[\setlength\arraycolsep{4pt}\renewcommand{\arraystretch}{1}
\begin{array}{llll}
u_1{=}\frac{1}{2}a_{12}{+}a_{22}&
t_1{=}\frac{10}{9}u_1&
t_2{=}\frac{8}{9}a_{11}{-}\frac{2}{3}a_{12}&
t_4{=}\frac{10}{9}a_{12}\\[\smallskipamount]
t_3{=}\frac{8}{9}a_{21}{+}\frac{4}{9}\bigl(a_{11}{+}u_1\bigr)&
t_0{=}t_2{-}t_3&
t_5{=}t_1{+}t_0&
t_6{=}t_4{+}t_0\\[\smallskipamount]
v_1{=}\frac{1}{2}b_{12}&
s_1{=}v_1{-}b_{22}&
s_2{=}v_1{-}b_{11}&
s_3{=}\frac{5}{4}b_{12}\\[\smallskipamount]
s_4{=}\frac{2}{5}b_{22}{-}\frac{4}{5}b_{21}{+}\frac{3}{5}s_2&
s_0{=}s_1{+}s_4&
s_5{=}s_0{-}s_2&
s_6{=}s_3{-}s_0
\end{array}
\]
\vspace{-4pt}
\[%
\begin{array}{llll}
p_0{=}t_0{\cdot}s_0&
p_1{=}t_1{\cdot}s_1&
p_2{=}t_2{\cdot}s_2&
p_3{=}t_3{\cdot}s_3\\
p_4{=}t_4{\cdot}s_4&
p_5{=}t_5{\cdot}s_5&
p_6{=}t_6{\cdot}s_6
\end{array}
\]
\vspace{-4pt}
\[%
\begin{array}{llllll}
w_1{=}p_6{+}p_0{+}p_4&
w_2{=}p_5{+}p_6&
w_3{=}p_3{+}w_1&
w_4{=}p_2{+}p_4\\
w_5{=}p_1{+}w_1&
w_6{=}\frac{9}{20}w_3&&
c_{11}{=}w_6{-}\frac{9}{8}w_4\\[\smallskipamount]
c_{12}{=}\frac{9}{10}w_3&
\multicolumn{2}{l}{c_{21}{=}\frac{27}{40}w_5{-}\frac{9}{8}w_2{+}\frac{1}{2}c_{11}}&
c_{22}{=}w_6{-}\frac{9}{10}w_5
\end{array}
\]
\vspace{-5pt}
\end{minipage}}
\caption{\textsc{slp} of~\Cref{eq:powrot},~\({\growthfactor\approx{12.2034}}\), with~\(24\) add.\ and~\(19\) mul.}\label{alg:powrot}%
\end{table}
\par
Finally, we present in~\Cref{alg:0695,alg:0661}, successive higher-order rational approximations of the point~\({\bigl(\sqrt[4]{\lfrac{4}{3}},-\lfrac{1}{2}\bigr)}\) reducing the growth factor~\(\growthfactor\) to~\(12.0695\) (resp.~\(12.0661\)), approaching~\(12.06603\).
They then provide rational algorithms whose accuracy is pretty close
to our best one, as shown next in~\Cref{fig:gamma}.
\begin{equation}\label{alg:0695}
\begin{aligned}
\begin{smatrix}
{-}\frac{167042}{345665} &  \frac{295936}{345665} &  {-}\frac{295936}{345665} &  {-}\frac{167042}{345665}\\
{-}\frac{178623}{345665} &  {-}\frac{51622047}{176980480} &  \frac{295936}{345665} &  \frac{167042}{345665}\\
0 &  {-}\frac{51622047}{88490240} &  0 &  \frac{334084}{345665}\\
-1 &  \frac{289}{512} &  0 &  0\\
0 &  \frac{289}{256} &  0 &  0\\
{-}\frac{167042}{345665} &  {-}\frac{24137569}{88490240} &  {-}\frac{295936}{345665} &  {-}\frac{167042}{345665}\\
{-}\frac{167042}{345665} &  \frac{24137569}{88490240} &  {-}\frac{295936}{345665} &  \frac{167042}{345665}
\end{smatrix};\
\begin{smatrix}
{-}\frac{256}{289} &  {-}\frac{1}{2} &  \frac{1}{2} &  {-}\frac{256}{289}\\
{-}\frac{345665}{295936} &  0 &  0 &  0\\
{-}\frac{345665}{591872} &  0 &  \frac{345665}{334084} &  0\\
{-}\frac{178623}{295936} &  -1 &  0 &  0\\
\frac{178623}{591872} &  \frac{1}{2} &  \frac{178623}{334084} &  \frac{256}{289}\\
{-}\frac{289}{1024} &  \frac{1}{2} &  {-}\frac{1}{2} &  \frac{256}{289}\\
{-}\frac{289}{1024} &  \frac{1}{2} &  \frac{1}{2} &  {-}\frac{256}{289}
\end{smatrix};\\
\begin{smatrix}
\frac{295936}{345665}&\frac{295936}{345665}&0&0&\frac{295936}{345665}&\frac{295936}{345665}&0\\
\frac{178623}{345665}&{-}\frac{167042}{345665}&1&0&\frac{178623}{345665}&\frac{178623}{345665}&0\\
{-}\frac{178623}{345665}&{-}\frac{178623}{345665}&0&1&\frac{167042}{345665}&{-}\frac{178623}{345665}&0\\
\frac{295936}{345665}&\frac{51622047}{176980480}&\frac{289}{512}&{-}\frac{289}{512}&\frac{51622047}{176980480}&{-}\frac{31906176129}{102294717440}&{-}\frac{345665}{295936}\\
\end{smatrix}.
\end{aligned}
\end{equation}

\begin{equation}\label{alg:0661}
\begin{aligned}
\begin{smatrix}
\frac{33124}{38165}&\frac{19208}{38165}&{-}\frac{19208}{38165}&\frac{33124}{38165}\\
\frac{33124}{38165}&\frac{19208}{38165}&\frac{18957}{38165}&\frac{1857786}{6449885}\\
0&\frac{38416}{38165}&0&\frac{3715572}{6449885}\\
0&0&1&{-}\frac{98}{169}\\
0&0&0&\frac{196}{169}\\
\frac{33124}{38165}&\frac{19208}{38165}&{-}\frac{19208}{38165}&{-}\frac{1882384}{6449885}\\
\frac{33124}{38165}&{-}\frac{19208}{38165}&{-}\frac{19208}{38165}&\frac{1882384}{6449885}\\
\end{smatrix};\quad
\begin{smatrix}
{-}\frac{169}{196}&{-}\frac{1}{2}&\frac{1}{2}&{-}\frac{169}{196}\\
\frac{38165}{33124}&0&0&0\\
\frac{38165}{66248}&0&{-}\frac{38165}{38416}&0\\
\frac{18957}{33124}&1&0&0\\
\frac{18957}{66248}&\frac{1}{2}&\frac{18957}{38416}&\frac{169}{196}\\
{-}\frac{49}{169}&\frac{1}{2}&{-}\frac{1}{2}&\frac{169}{196}\\
{-}\frac{49}{169}&\frac{1}{2}&\frac{1}{2}&{-}\frac{169}{196}\\
\end{smatrix};\\
\begin{smatrix}
{-}\frac{18957}{38165} & \frac{19208}{38165} & -1 & 0 & {-}\frac{18957}{38165} & {-}\frac{18957}{38165} & 0 \\
{-}\frac{33124}{38165} & {-}\frac{1857786}{6449885} & {-}\frac{98}{169} & \frac{98}{169} & {-}\frac{1857786}{6449885} & \frac{359367849}{1264177460} & \frac{38165}{33124} \\
\frac{33124}{38165} & \frac{33124}{38165} & 0 & 0 & \frac{33124}{38165} & \frac{33124}{38165} & 0 \\
{-}\frac{18957}{38165} & {-}\frac{18957}{38165} & 0 & 1 & \frac{19208}{38165} & {-}\frac{18957}{38165} & 0 \\
\end{smatrix}.
\end{aligned}
\end{equation}
\subsection{Accuracy for a normal distribution}
The accuracy obtained with our different fast variants is given
in~\Cref{fig:gamma}.
For this, we use the Matlab framework of~\cite{Dai:2023aa},
which we forked in~\cite{jgd:2024:mFMM} where we have just added the
implementations of the variants presented here.
Thus, in~\Cref{fig:gammaunif,fig:gamma,fig:Schwartz,fig:badcond} we present the
error as the infinity norm of the difference between the result of our
implementations and the \emph{exact} matrix multiplication.

\begin{figure}[htb]\centering
\caption{Numerical accuracy vs size (normal distribution)}\label{fig:gamma}
%
%
\begingroup
  \makeatletter
  \providecommand\color[2][]{%
    \GenericError{(gnuplot) \space\space\space\@spaces}{%
      Package color not loaded in conjunction with
      terminal option `colourtext'%
    }{See the gnuplot documentation for explanation.%
    }{Either use 'blacktext' in gnuplot or load the package
      color.sty in LaTeX.}%
    \renewcommand\color[2][]{}%
  }%
  \providecommand\includegraphics[2][]{%
    \GenericError{(gnuplot) \space\space\space\@spaces}{%
      Package graphicx or graphics not loaded%
    }{See the gnuplot documentation for explanation.%
    }{The gnuplot epslatex terminal needs graphicx.sty or graphics.sty.}%
    \renewcommand\includegraphics[2][]{}%
  }%
  \providecommand\rotatebox[2]{#2}%
  \@ifundefined{ifGPcolor}{%
    \newif\ifGPcolor
    \GPcolortrue
  }{}%
  \@ifundefined{ifGPblacktext}{%
    \newif\ifGPblacktext
    \GPblacktexttrue
  }{}%
  \let\gplgaddtomacro\g@addto@macro
  \gdef\gplbacktext{}%
  \gdef\gplfronttext{}%
  \makeatother
  \ifGPblacktext
    \def\colorrgb#1{}%
    \def\colorgray#1{}%
  \else
    \ifGPcolor
      \def\colorrgb#1{\color[rgb]{#1}}%
      \def\colorgray#1{\color[gray]{#1}}%
      \expandafter\def\csname LTw\endcsname{\color{white}}%
      \expandafter\def\csname LTb\endcsname{\color{black}}%
      \expandafter\def\csname LTa\endcsname{\color{black}}%
      \expandafter\def\csname LT0\endcsname{\color[rgb]{1,0,0}}%
      \expandafter\def\csname LT1\endcsname{\color[rgb]{0,1,0}}%
      \expandafter\def\csname LT2\endcsname{\color[rgb]{0,0,1}}%
      \expandafter\def\csname LT3\endcsname{\color[rgb]{1,0,1}}%
      \expandafter\def\csname LT4\endcsname{\color[rgb]{0,1,1}}%
      \expandafter\def\csname LT5\endcsname{\color[rgb]{1,1,0}}%
      \expandafter\def\csname LT6\endcsname{\color[rgb]{0,0,0}}%
      \expandafter\def\csname LT7\endcsname{\color[rgb]{1,0.3,0}}%
      \expandafter\def\csname LT8\endcsname{\color[rgb]{0.5,0.5,0.5}}%
    \else
      \def\colorrgb#1{\color{black}}%
      \def\colorgray#1{\color[gray]{#1}}%
      \expandafter\def\csname LTw\endcsname{\color{white}}%
      \expandafter\def\csname LTb\endcsname{\color{black}}%
      \expandafter\def\csname LTa\endcsname{\color{black}}%
      \expandafter\def\csname LT0\endcsname{\color{black}}%
      \expandafter\def\csname LT1\endcsname{\color{black}}%
      \expandafter\def\csname LT2\endcsname{\color{black}}%
      \expandafter\def\csname LT3\endcsname{\color{black}}%
      \expandafter\def\csname LT4\endcsname{\color{black}}%
      \expandafter\def\csname LT5\endcsname{\color{black}}%
      \expandafter\def\csname LT6\endcsname{\color{black}}%
      \expandafter\def\csname LT7\endcsname{\color{black}}%
      \expandafter\def\csname LT8\endcsname{\color{black}}%
    \fi
  \fi
    \setlength{\unitlength}{0.0500bp}%
    \ifx\gptboxheight\undefined%
      \newlength{\gptboxheight}%
      \newlength{\gptboxwidth}%
      \newsavebox{\gptboxtext}%
    \fi%
    \setlength{\fboxrule}{0.5pt}%
    \setlength{\fboxsep}{1pt}%
    \definecolor{tbcol}{rgb}{1,1,1}%
\begin{picture}(7560.00,5040.00)%
    \gplgaddtomacro\gplbacktext{%
      \csname LTb\endcsname
      \put(219,585){\makebox(0,0)[r]{\strut{}$10^{-14}$}}%
      \csname LTb\endcsname
      \put(219,1653){\makebox(0,0)[r]{\strut{}$10^{-13}$}}%
      \csname LTb\endcsname
      \put(219,2721){\makebox(0,0)[r]{\strut{}$10^{-12}$}}%
      \csname LTb\endcsname
      \put(219,3789){\makebox(0,0)[r]{\strut{}$10^{-11}$}}%
      \csname LTb\endcsname
      \put(374,174){\makebox(0,0){\strut{}$32$}}%
      \csname LTb\endcsname
      \put(2164,174){\makebox(0,0){\strut{}$64$}}%
      \csname LTb\endcsname
      \put(3954,174){\makebox(0,0){\strut{}$128$}}%
      \csname LTb\endcsname
      \put(5744,174){\makebox(0,0){\strut{}$256$}}%
      \csname LTb\endcsname
      \put(7534,174){\makebox(0,0){\strut{}$512$}}%
    }%
    \gplgaddtomacro\gplfronttext{%
      \csname LTb\endcsname
      \put(728,4601){\makebox(0,0)[l]{\strut{}Winograd \cite{Winograd:1977:complexite}}}%
      \csname LTb\endcsname
      \put(728,4253){\makebox(0,0)[l]{\strut{}Strassen \cite{strassen:1969}}}%
      \csname LTb\endcsname
      \put(728,3904){\makebox(0,0)[l]{\strut{}\Cref{alg:powers,eq:powers}}}%
      \csname LTb\endcsname
      \put(728,3555){\makebox(0,0)[l]{\strut{}\Cref{alg:powrot,eq:powrot}}}%
      \csname LTb\endcsname
      \put(3433,4601){\makebox(0,0)[l]{\strut{}\Cref{alg:0695}}}%
      \csname LTb\endcsname
      \put(3433,4253){\makebox(0,0)[l]{\strut{}\Cref{alg:0661}}}%
      \csname LTb\endcsname
      \put(3433,3904){\makebox(0,0)[l]{\strut{}\Cref{alg:asopt,eq:asopt}}}%
      \csname LTb\endcsname
      \put(3433,3555){\makebox(0,0)[l]{\strut{}Conventional}}%
      \csname LTb\endcsname
      \put(0,4427){\makebox(0,0){\strut{}error}}%
      \csname LTb\endcsname
      \put(3916,0){\makebox(0,0){\strut{}square matrix dimension}}%
    }%
    \gplbacktext
    \put(0,0){\includegraphics[width={378.00bp},height={252.00bp}]{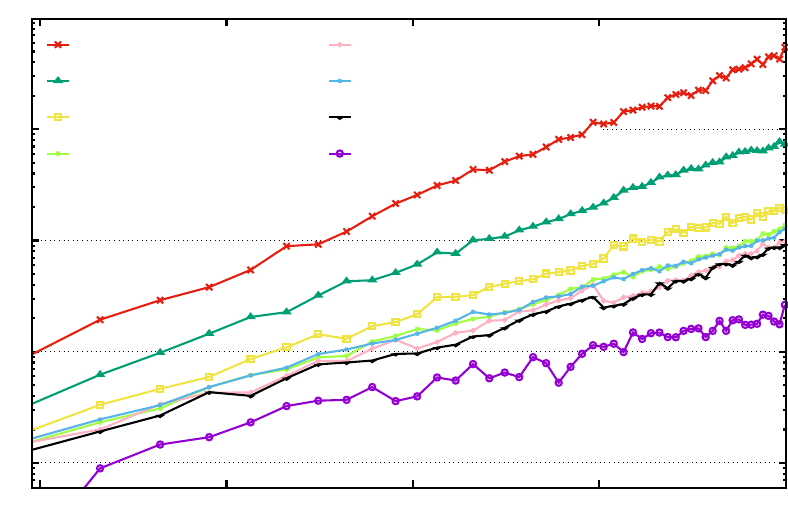}}%
    \gplfronttext
  \end{picture}%
\endgroup
\end{figure}
In~\Cref{fig:gamma}, all our variants, \Cref{alg:powers,alg:powrot,alg:0695,alg:0661,alg:asopt} with decreasing~\(\growthfactor\), are mostly more and more accurate.
Our best algorithm presents an order of magnitude advantage over Strassen's and two orders of magnitude advantage over Winograd's.
It is then quite close to the conventional algorithm's accuracy.
\Cref{fig:gamma} uses normal distribution while the same behavior was
also obtained with a uniform distribution in~\Cref{fig:gammaunif}.

\subsection{Alternative basis sparsification}\label{ssec:schwartz}
The technique of~\cite{Karstadt:2017aa,Beniamini:2019aa} reduces the number of operations by factoring each matrix in the \textsc{hm} decomposition into a sparser one via a~\(\matrixsize{4}{4}\) change of basis (\textsc{CoB}).
In a recursive version, the left and right-hand sides (resp.\ result) of considered \textsc{CoB} can be recursively precomputed (resp.\ post-computed),  for a total cost in~\({\bbigO{n^{2}\log{n}}}\).
In the meantime the sparser~\(\matrixsize{7}{4}\) matrices are applied, reducing the dominant term of the computation.
The optimal decomposition of Winograd's algorithm in~\cite[\S~3.3]{Karstadt:2017aa} reduces the number of intermediate additions from~\(15\) to~\(12\).
For a fully recursive version, this reduces the leading term in the cost bound from~\(6{n}^{\log_{2}{\!7}}\) to~\(5{n}^{\log_{2}{\!7}}\).

Applying this approach to the algorithm of~\Cref{eq:asopt}, while using  the \textsc{CoB} of~\Cref{eq:alternative} leads to the sparser \textsc{hm} representation in~\Cref{eq:schwartz}.
The leading term of the cost bound thus is reduced from~\(13n^{\log_{2}{\!7}}\) for~\Cref{alg:asopt} to~\(5n^{\log_{2}{\!7}}\) for~\Cref{alg:LCoB,alg:RCoB,alg:schwartzopt,alg:CoBP,alg:CoBschwartzopt}.
\begin{align}\label{eq:alternative}
\begin{smatrix}
0&0&0&\frac{2}{\sqrt{3}}\\
0&1&0&\frac{\sqrt{3}}{3}\\
0&0&1&{-}\frac{\sqrt{3}}{3}\\
{-}\frac{\sqrt{3}}{2}&{-}\frac{1}{2}&\frac{1}{2}&{-}\frac{\sqrt{3}}{2}\\
\end{smatrix}&;
&\begin{smatrix}
0&\frac{2}{\sqrt{3}}&0&0\\
1&{-}\frac{\sqrt{3}}{3}&0&0\\
0&\frac{\sqrt{3}}{3}&0&-1\\
{-}\frac{1}{2}&\frac{\sqrt{3}}{2}&{-}\frac{\sqrt{3}}{2}&{-}\frac{1}{2}\\
\end{smatrix}&;&
\Transpose{\begin{smatrix}
{-}\frac{2}{\sqrt{3}}&0&0&0\\
\frac{\sqrt{3}}{3}&-1&0&0\\
{-}\frac{\sqrt{3}}{3}&0&-1&0\\
\frac{\sqrt{3}}{2}&{-}\frac{1}{2}&\frac{1}{2}&\frac{\sqrt{3}}{2}\\
\end{smatrix}}\!\!,&
\\[5pt]\label{eq:schwartz}
\begin{bmatrix}
0&0&1&-1\\
0&0&1&0\\
0&1&0&0\\
-1&0&0&0\\
0&0&0&1\\
1&0&0&1\\
0&1&0&1\\
\end{bmatrix}&;
&\begin{bmatrix}
1&0&0&0\\
0&-1&0&0\\
0&0&1&0\\
0&0&1&-1\\
0&0&0&1\\
1&0&0&-1\\
0&1&0&1\\
\end{bmatrix}&;&
\Transpose{\begin{bmatrix}
0&-1&0&1\\
0&0&1&0\\
0&1&0&0\\
0&0&1&1\\
0&0&0&1\\
1&0&0&1\\
1&0&0&0\\
\end{bmatrix}}\!\!.&
\end{align}

To obtain this \textsc{CoB}, the generic technique
of~\cite{Beniamini:2019aa} can be used.
In our case, for~\(\matrixsize{4}{4}\) \textsc{CoB}, the following
heuristic was sufficient to obtain optimal (\(12\)-additions)
sparse~\({(0,\pm{1})}\) intermediate matrices:
\begin{enumerate}
\item Find independent columns of each \textsc{CoB} one at a time;
\item For this, alternatively factor-out common coefficients in these
  resulting columns and find a linear combination minimizing the
  density of the resulting column, using as coefficients of the
  combination only in~\({\lbrace{-1},0,1\rbrace}\) and some of the
  values of the coefficients of the input;
\item Until this alternation does not sparsify anymore.
\end{enumerate}

This heuristic is implemented in
\href{https://github.com/jgdumas/plinopt/blob/main/sparsifier.cpp}{\texttt{plinopt/sparsifier}}~\cite{jgd:2024:plinopt}
and the resulting  implementation is shown
in~\Cref{alg:LCoB,alg:RCoB,alg:CoBP}.
\newcommand{\LCoB}{\ensuremath{\text{LCoB}}}
\begin{algorithm}[!ht]
\caption{\LCoB$(\mat{A},\ell)$ left change-of-basis of~\Cref{eq:alternative}}\label{alg:LCoB}
\begin{algorithmic}[1]
\IfThenEnd{$\ell\leq{0}$}{\Return{$\mat{A}$.}}
\State{%
$m_1=\LCoB(a_{11},\ell{-}1)$;
$m_2=\LCoB(a_{21},\ell{-}1)$;}
\State{%
$m_3=\LCoB(a_{12},\ell{-}1)$;
$m_4=\LCoB(a_{22},\ell{-}1)$;}
\State{%
$t_1 = \frac{1}{\sqrt{3}}m_4$;
$t_2 = m_3-m_2$;
$t_3 = m_1+m_4$;}
\State{\Return{$\left[\frac{2}{\sqrt{3}}m_4,
m_2+t_1,
m_3-t_1,
\frac{1}{2}t_2-\frac{\sqrt{3}}{2}t_3\right]$.}}
\end{algorithmic}
\end{algorithm}

\newcommand{\RCoB}{\ensuremath{\text{RCoB}}}
\begin{algorithm}[!ht]
\caption{\RCoB$(\mat{A},\ell)$ right change-of-basis of~\Cref{eq:alternative}}\label{alg:RCoB}
\begin{algorithmic}[1]
\IfThenEnd{$\ell\leq{0}$}{\Return{$\mat{A}$.}}
\State{%
$m_1=\RCoB(a_{11},\ell{-}1)$;
$m_2=\RCoB(a_{21},\ell{-}1)$;}
\State{%
$m_3=\RCoB(a_{12},\ell{-}1)$;
$m_4=\RCoB(a_{22},\ell{-}1)$;}
\State{%
$t_1 =\frac{1}{\sqrt{3}}m_2$;
$t_2 = m_1+m_4$;
$t_3 = m_2-m_3$;}
\State{\Return{$\left[\frac{2}{\sqrt{3}}m_2,
m_1-t_1,
t_1-m_4,
\frac{\sqrt{3}}{2}t_3-\frac{1}{2}t_{2}\right]$.}}
\end{algorithmic}
\end{algorithm}

\newcommand{\CoBP}{\ensuremath{\text{CoBP}}}
\begin{algorithm}[!ht]
\caption{\CoBP\({(\mat{A},\ell)}\) product change-of-basis of~\Cref{eq:alternative}}\label{alg:CoBP}
\begin{algorithmic}[1]
\IfThenEnd{$\ell\leq{0}$}{\Return{\(\mat{A}\)}}
\State{%
$m_1=\CoBP(a_{11},\ell{-}1)$;
$m_2=\CoBP(a_{21},\ell{-}1)$;}
\State{%
$m_3=\CoBP(a_{12},\ell{-}1)$;
$m_4=\CoBP(a_{22},\ell{-}1)$;}
\State{%
$t_1 = \frac{1}{2}m_4$;
$t_2 = m_2-m_3$;
$t_3 = \frac{\sqrt{3}}{2}m_4$;}
\State{\Return{\(\left[t_3+\frac{1}{\sqrt{3}}t_2-m_1\frac{2}{\sqrt{3}},-m_2-t_1,t_1-m_3,t_3\right]\).}}
\end{algorithmic}
\end{algorithm}
Finally, we obtain~\Cref{alg:CoBschwartzopt}, where the inner \textsc{hm} representation was also implemented with~\plinopt, as shown in~\Cref{alg:schwartzopt}.
\begin{table}[!ht]%
\fbox{%
\begin{minipage}{.95\linewidth}\vspace{-5pt}
\[%
\begin{array}{lll}
s_1=a_{11}+a_{12}&
s_2=a_{11}+a_{22}&
s_3=a_{11}-a_{21}\\
t_1=b_{12}+b_{22}&
t_2=b_{11}+b_{12}&
t_3=b_{12}+b_{21}
\end{array}
\]
\vspace{-5pt}
\[%
\begin{array}{llll}
p_1 = a_{11}{\cdot}b_{12} &
p_2 = s_1{\cdot}b_{21} &
p_3 = a_{21}{\cdot}t_1 &
p_4 = a_{12}{\cdot}t_2 \\
p_5 = s_2{\cdot}b_{22} &
p_6 = a_{22}{\cdot}t_3 &
p_7 = s_3{\cdot}b_{11}
\end{array}\]
\vspace{-5pt}
\[%
\begin{array}{llll}
c_{11}{=}p_7{-}p_6 &
c_{12}{=}p_2{+}p_3 &
c_{21}{=}p_4{-}p_5 &
c_{22}{=}p_1{+}p_2{+}p_5{+}p_6
\end{array}\]
\vspace{-10pt}
\end{minipage}}
\caption{\textsc{slp} of~\Cref{eq:schwartz} with~\(12\) additions}\label{alg:schwartzopt}
\end{table}

\begin{algorithm}[!ht]
\caption{Sparsification applied to~\Cref{eq:asopt} (via~\Cref{eq:alternative,eq:schwartz})}\label{alg:CoBschwartzopt}
\begin{algorithmic}[1]
\Require{$\mat{A},\mat{B}\in\Field^{{{n_{0}}2^\ell}\times{{n_{0}}2^\ell}}$.}
\Ensure{$\mat{C}=\MatrixProduct{A}{B}$.}
\State{$\bar{\mat{A}}\leftarrow\LCoB(\mat{A},\ell)$; $\bar{\mat{B}}\leftarrow\RCoB(\mat{B},\ell)$ \hfill\Comment{Via~\Cref{alg:LCoB,alg:RCoB}}}
\State{$\bar{\mat{C}}\leftarrow\bar{\mat{A}}\cdot\bar{\mat{B}}$; \hfill\Comment{Via~\Cref{alg:schwartzopt} with~$\ell$ recursive calls}}
\State{\Return{$\mat{C}\leftarrow\CoBP(\bar{\mat{C}},\ell)$.\hfill\Comment{Via~\Cref{alg:CoBP}}}}
\end{algorithmic}
\end{algorithm}

The sparsification process can improve the~\(\growthfactor\) growth factor: when applied to Winograd's algorithm~\cite{Winograd:1977:complexite}, it goes down from~\({{7}+4\sqrt{2}+3\sqrt{3} \approx{17.853}}\) to~\({4+6\sqrt{2}\approx{12.486}}\) in~\cite{Karstadt:2017aa}.
When applied to~\Cref{eq:asopt}, it decreases from~\({{{2\sqrt{2}}+{\lfrac{16}{\sqrt{3}}}}\approx{\!12.066}}\) to~\({{7+3\sqrt{2}}\approx{11.243}}\).

The resulting sparser bilinear operator itself, however, no longer
correspond to a matrix multiplication algorithm.
Therefore, the error of this operator also include the contribution of
the change of basis and thus only follows the
weaker bound of~\Cref{thm:altbase}.

\Cref{fig:Schwartz} shows, nevertheless, that in practice,
\Cref{alg:CoBschwartzopt} accuracy is in fact very close to that of
its counterpart without an alternative basis.

We thus have obtained an algorithm that enjoys simultaneously
the best known leading term in the cost bound, and a close to the best
known numerical accuracy for sub-cubic $2\times2$ algorithms.
\begin{figure}[htbp]\centering
\caption{Numerical effect of sparsification (\footnotesize{normal
    distribution})}\label{fig:Schwartz}
\begingroup
  \makeatletter
  \providecommand\color[2][]{%
    \GenericError{(gnuplot) \space\space\space\@spaces}{%
      Package color not loaded in conjunction with
      terminal option `colourtext'%
    }{See the gnuplot documentation for explanation.%
    }{Either use 'blacktext' in gnuplot or load the package
      color.sty in LaTeX.}%
    \renewcommand\color[2][]{}%
  }%
  \providecommand\includegraphics[2][]{%
    \GenericError{(gnuplot) \space\space\space\@spaces}{%
      Package graphicx or graphics not loaded%
    }{See the gnuplot documentation for explanation.%
    }{The gnuplot epslatex terminal needs graphicx.sty or graphics.sty.}%
    \renewcommand\includegraphics[2][]{}%
  }%
  \providecommand\rotatebox[2]{#2}%
  \@ifundefined{ifGPcolor}{%
    \newif\ifGPcolor
    \GPcolortrue
  }{}%
  \@ifundefined{ifGPblacktext}{%
    \newif\ifGPblacktext
    \GPblacktexttrue
  }{}%
  \let\gplgaddtomacro\g@addto@macro
  \gdef\gplbacktext{}%
  \gdef\gplfronttext{}%
  \makeatother
  \ifGPblacktext
    \def\colorrgb#1{}%
    \def\colorgray#1{}%
  \else
    \ifGPcolor
      \def\colorrgb#1{\color[rgb]{#1}}%
      \def\colorgray#1{\color[gray]{#1}}%
      \expandafter\def\csname LTw\endcsname{\color{white}}%
      \expandafter\def\csname LTb\endcsname{\color{black}}%
      \expandafter\def\csname LTa\endcsname{\color{black}}%
      \expandafter\def\csname LT0\endcsname{\color[rgb]{1,0,0}}%
      \expandafter\def\csname LT1\endcsname{\color[rgb]{0,1,0}}%
      \expandafter\def\csname LT2\endcsname{\color[rgb]{0,0,1}}%
      \expandafter\def\csname LT3\endcsname{\color[rgb]{1,0,1}}%
      \expandafter\def\csname LT4\endcsname{\color[rgb]{0,1,1}}%
      \expandafter\def\csname LT5\endcsname{\color[rgb]{1,1,0}}%
      \expandafter\def\csname LT6\endcsname{\color[rgb]{0,0,0}}%
      \expandafter\def\csname LT7\endcsname{\color[rgb]{1,0.3,0}}%
      \expandafter\def\csname LT8\endcsname{\color[rgb]{0.5,0.5,0.5}}%
    \else
      \def\colorrgb#1{\color{black}}%
      \def\colorgray#1{\color[gray]{#1}}%
      \expandafter\def\csname LTw\endcsname{\color{white}}%
      \expandafter\def\csname LTb\endcsname{\color{black}}%
      \expandafter\def\csname LTa\endcsname{\color{black}}%
      \expandafter\def\csname LT0\endcsname{\color{black}}%
      \expandafter\def\csname LT1\endcsname{\color{black}}%
      \expandafter\def\csname LT2\endcsname{\color{black}}%
      \expandafter\def\csname LT3\endcsname{\color{black}}%
      \expandafter\def\csname LT4\endcsname{\color{black}}%
      \expandafter\def\csname LT5\endcsname{\color{black}}%
      \expandafter\def\csname LT6\endcsname{\color{black}}%
      \expandafter\def\csname LT7\endcsname{\color{black}}%
      \expandafter\def\csname LT8\endcsname{\color{black}}%
    \fi
  \fi
    \setlength{\unitlength}{0.0500bp}%
    \ifx\gptboxheight\undefined%
      \newlength{\gptboxheight}%
      \newlength{\gptboxwidth}%
      \newsavebox{\gptboxtext}%
    \fi%
    \setlength{\fboxrule}{0.5pt}%
    \setlength{\fboxsep}{1pt}%
    \definecolor{tbcol}{rgb}{1,1,1}%
\begin{picture}(7560.00,4680.00)%
    \gplgaddtomacro\gplbacktext{%
      \csname LTb\endcsname
      \put(219,566){\makebox(0,0)[r]{\strut{}$10^{-14}$}}%
      \csname LTb\endcsname
      \put(219,1549){\makebox(0,0)[r]{\strut{}$10^{-13}$}}%
      \csname LTb\endcsname
      \put(219,2531){\makebox(0,0)[r]{\strut{}$10^{-12}$}}%
      \csname LTb\endcsname
      \put(219,3513){\makebox(0,0)[r]{\strut{}$10^{-11}$}}%
      \csname LTb\endcsname
      \put(374,174){\makebox(0,0){\strut{}$32$}}%
      \csname LTb\endcsname
      \put(2164,174){\makebox(0,0){\strut{}$64$}}%
      \csname LTb\endcsname
      \put(3954,174){\makebox(0,0){\strut{}$128$}}%
      \csname LTb\endcsname
      \put(5744,174){\makebox(0,0){\strut{}$256$}}%
      \csname LTb\endcsname
      \put(7534,174){\makebox(0,0){\strut{}$512$}}%
    }%
    \gplgaddtomacro\gplfronttext{%
      \csname LTb\endcsname
      \put(728,4241){\makebox(0,0)[l]{\strut{}Sparse Winograd, by \cite{Karstadt:2017aa}}}%
      \csname LTb\endcsname
      \put(728,3893){\makebox(0,0)[l]{\strut{}Winograd \cite{Winograd:1977:complexite}}}%
      \csname LTb\endcsname
      \put(728,3544){\makebox(0,0)[l]{\strut{}Sparse Strassen}}%
      \csname LTb\endcsname
      \put(728,3195){\makebox(0,0)[l]{\strut{}Strassen \cite{strassen:1969}}}%
      \csname LTb\endcsname
      \put(3799,4241){\makebox(0,0)[l]{\strut{}Sparse \Cref{alg:asopt} (\Cref{alg:CoBschwartzopt})}}%
      \csname LTb\endcsname
      \put(3799,3893){\makebox(0,0)[l]{\strut{}\Cref{alg:asopt,eq:asopt}}}%
      \csname LTb\endcsname
      \put(3799,3544){\makebox(0,0)[l]{\strut{}Conventional}}%
      \csname LTb\endcsname
      \put(0,4247){\makebox(0,0){\strut{}error}}%
      \csname LTb\endcsname
      \put(3916,0){\makebox(0,0){\strut{}square matrix dimension}}%
    }%
    \gplbacktext
    \put(0,0){\includegraphics[width={378.00bp},height={234.00bp}]{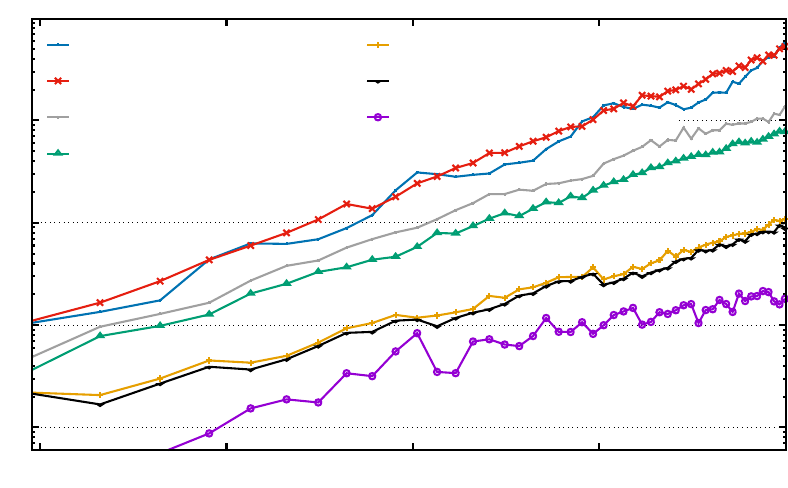}}%
    \gplfronttext
  \end{picture}%
\endgroup
\end{figure}

Following~\cite[\S~3.2]{Ozaki:2012:erfMM}, we can also confirm our
algorithms' accuracy on badly conditioned matrices.
We thus now study the effect of sparsification
on random matrix with preassigned singular values and large condition
number~\({\mathord{\approx}{10^{12}}}\) given by the Matlab function
gallery 'randsvd'.
We see in~\Cref{fig:badcond} that the fast variants behavior is
unchanged while only the conventional algorithm performs better.

\begin{figure}[htbp]\centering
\caption{Numerical Effect of Sparsification (\footnotesize{large conditioning})}\label{fig:badcond}
%
%
\begingroup
  \makeatletter
  \providecommand\color[2][]{%
    \GenericError{(gnuplot) \space\space\space\@spaces}{%
      Package color not loaded in conjunction with
      terminal option `colourtext'%
    }{See the gnuplot documentation for explanation.%
    }{Either use 'blacktext' in gnuplot or load the package
      color.sty in LaTeX.}%
    \renewcommand\color[2][]{}%
  }%
  \providecommand\includegraphics[2][]{%
    \GenericError{(gnuplot) \space\space\space\@spaces}{%
      Package graphicx or graphics not loaded%
    }{See the gnuplot documentation for explanation.%
    }{The gnuplot epslatex terminal needs graphicx.sty or graphics.sty.}%
    \renewcommand\includegraphics[2][]{}%
  }%
  \providecommand\rotatebox[2]{#2}%
  \@ifundefined{ifGPcolor}{%
    \newif\ifGPcolor
    \GPcolortrue
  }{}%
  \@ifundefined{ifGPblacktext}{%
    \newif\ifGPblacktext
    \GPblacktexttrue
  }{}%
  \let\gplgaddtomacro\g@addto@macro
  \gdef\gplbacktext{}%
  \gdef\gplfronttext{}%
  \makeatother
  \ifGPblacktext
    \def\colorrgb#1{}%
    \def\colorgray#1{}%
  \else
    \ifGPcolor
      \def\colorrgb#1{\color[rgb]{#1}}%
      \def\colorgray#1{\color[gray]{#1}}%
      \expandafter\def\csname LTw\endcsname{\color{white}}%
      \expandafter\def\csname LTb\endcsname{\color{black}}%
      \expandafter\def\csname LTa\endcsname{\color{black}}%
      \expandafter\def\csname LT0\endcsname{\color[rgb]{1,0,0}}%
      \expandafter\def\csname LT1\endcsname{\color[rgb]{0,1,0}}%
      \expandafter\def\csname LT2\endcsname{\color[rgb]{0,0,1}}%
      \expandafter\def\csname LT3\endcsname{\color[rgb]{1,0,1}}%
      \expandafter\def\csname LT4\endcsname{\color[rgb]{0,1,1}}%
      \expandafter\def\csname LT5\endcsname{\color[rgb]{1,1,0}}%
      \expandafter\def\csname LT6\endcsname{\color[rgb]{0,0,0}}%
      \expandafter\def\csname LT7\endcsname{\color[rgb]{1,0.3,0}}%
      \expandafter\def\csname LT8\endcsname{\color[rgb]{0.5,0.5,0.5}}%
    \else
      \def\colorrgb#1{\color{black}}%
      \def\colorgray#1{\color[gray]{#1}}%
      \expandafter\def\csname LTw\endcsname{\color{white}}%
      \expandafter\def\csname LTb\endcsname{\color{black}}%
      \expandafter\def\csname LTa\endcsname{\color{black}}%
      \expandafter\def\csname LT0\endcsname{\color{black}}%
      \expandafter\def\csname LT1\endcsname{\color{black}}%
      \expandafter\def\csname LT2\endcsname{\color{black}}%
      \expandafter\def\csname LT3\endcsname{\color{black}}%
      \expandafter\def\csname LT4\endcsname{\color{black}}%
      \expandafter\def\csname LT5\endcsname{\color{black}}%
      \expandafter\def\csname LT6\endcsname{\color{black}}%
      \expandafter\def\csname LT7\endcsname{\color{black}}%
      \expandafter\def\csname LT8\endcsname{\color{black}}%
    \fi
  \fi
    \setlength{\unitlength}{0.0500bp}%
    \ifx\gptboxheight\undefined%
      \newlength{\gptboxheight}%
      \newlength{\gptboxwidth}%
      \newsavebox{\gptboxtext}%
    \fi%
    \setlength{\fboxrule}{0.5pt}%
    \setlength{\fboxsep}{1pt}%
    \definecolor{tbcol}{rgb}{1,1,1}%
\begin{picture}(7560.00,4680.00)%
    \gplgaddtomacro\gplbacktext{%
      \csname LTb\endcsname
      \put(219,1083){\makebox(0,0)[r]{\strut{}$10^{-15}$}}%
      \csname LTb\endcsname
      \put(219,1943){\makebox(0,0)[r]{\strut{}$10^{-14}$}}%
      \csname LTb\endcsname
      \put(219,2804){\makebox(0,0)[r]{\strut{}$10^{-13}$}}%
      \csname LTb\endcsname
      \put(219,3664){\makebox(0,0)[r]{\strut{}$10^{-12}$}}%
      \csname LTb\endcsname
      \put(374,174){\makebox(0,0){\strut{}$32$}}%
      \csname LTb\endcsname
      \put(2164,174){\makebox(0,0){\strut{}$64$}}%
      \csname LTb\endcsname
      \put(3954,174){\makebox(0,0){\strut{}$128$}}%
      \csname LTb\endcsname
      \put(5744,174){\makebox(0,0){\strut{}$256$}}%
      \csname LTb\endcsname
      \put(7534,174){\makebox(0,0){\strut{}$512$}}%
    }%
    \gplgaddtomacro\gplfronttext{%
      \csname LTb\endcsname
      \put(728,4241){\makebox(0,0)[l]{\strut{}Sparse Winograd, by \cite{Karstadt:2017aa}}}%
      \csname LTb\endcsname
      \put(728,3893){\makebox(0,0)[l]{\strut{}Winograd \cite{Winograd:1977:complexite}}}%
      \csname LTb\endcsname
      \put(728,3544){\makebox(0,0)[l]{\strut{}Sparse Strassen}}%
      \csname LTb\endcsname
      \put(728,3195){\makebox(0,0)[l]{\strut{}Strassen \cite{strassen:1969}}}%
      \csname LTb\endcsname
      \put(3799,4241){\makebox(0,0)[l]{\strut{}Sparse \Cref{alg:asopt} (\Cref{alg:CoBschwartzopt})}}%
      \csname LTb\endcsname
      \put(3799,3893){\makebox(0,0)[l]{\strut{}\Cref{alg:asopt,eq:asopt}}}%
      \csname LTb\endcsname
      \put(3799,3544){\makebox(0,0)[l]{\strut{}Conventional}}%
      \csname LTb\endcsname
      \put(0,4334){\makebox(0,0){\strut{}error}}%
      \csname LTb\endcsname
      \put(3916,0){\makebox(0,0){\strut{}square matrix dimension}}%
    }%
    \gplbacktext
    \put(0,0){\includegraphics[width={378.00bp},height={234.00bp}]{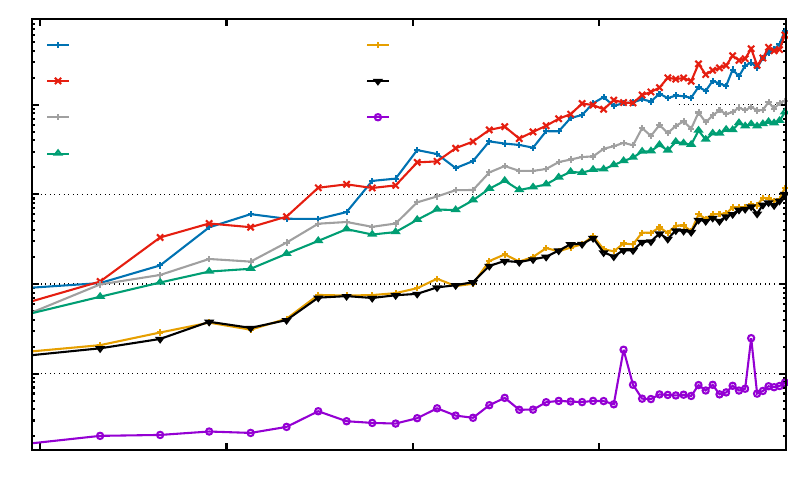}}%
    \gplfronttext
  \end{picture}%
\endgroup
\end{figure}

\subsection{Experiments on accurate versions of Smirnov's algorithm}\label{ssec:54}
We report in this section on the experimental behavior of our accurate
variants of Smirnov's $\FMMA{6}{3}{3}{40}$, $\FMMA{3}{6}{3}{40}$ and
$\FMMA{3}{3}{6}{40}$ algorithms on uniform distributions of
rectangular matrices.
\begin{figure}[!htbp]\centering
\caption{Numerical accuracy of $\FMMA{3}{3}{6}{40}$
  algorithms}\label{fig:336}
\begingroup
  \fontfamily{1}%
  \selectfont
  \makeatletter
  \providecommand\color[2][]{%
    \GenericError{(gnuplot) \space\space\space\@spaces}{%
      Package color not loaded in conjunction with
      terminal option `colourtext'%
    }{See the gnuplot documentation for explanation.%
    }{Either use 'blacktext' in gnuplot or load the package
      color.sty in LaTeX.}%
    \renewcommand\color[2][]{}%
  }%
  \providecommand\includegraphics[2][]{%
    \GenericError{(gnuplot) \space\space\space\@spaces}{%
      Package graphicx or graphics not loaded%
    }{See the gnuplot documentation for explanation.%
    }{The gnuplot epslatex terminal needs graphicx.sty or graphics.sty.}%
    \renewcommand\includegraphics[2][]{}%
  }%
  \providecommand\rotatebox[2]{#2}%
  \@ifundefined{ifGPcolor}{%
    \newif\ifGPcolor
    \GPcolortrue
  }{}%
  \@ifundefined{ifGPblacktext}{%
    \newif\ifGPblacktext
    \GPblacktexttrue
  }{}%
  \let\gplgaddtomacro\g@addto@macro
  \gdef\gplbacktext{}%
  \gdef\gplfronttext{}%
  \makeatother
  \ifGPblacktext
    \def\colorrgb#1{}%
    \def\colorgray#1{}%
  \else
    \ifGPcolor
      \def\colorrgb#1{\color[rgb]{#1}}%
      \def\colorgray#1{\color[gray]{#1}}%
      \expandafter\def\csname LTw\endcsname{\color{white}}%
      \expandafter\def\csname LTb\endcsname{\color{black}}%
      \expandafter\def\csname LTa\endcsname{\color{black}}%
      \expandafter\def\csname LT0\endcsname{\color[rgb]{1,0,0}}%
      \expandafter\def\csname LT1\endcsname{\color[rgb]{0,1,0}}%
      \expandafter\def\csname LT2\endcsname{\color[rgb]{0,0,1}}%
      \expandafter\def\csname LT3\endcsname{\color[rgb]{1,0,1}}%
      \expandafter\def\csname LT4\endcsname{\color[rgb]{0,1,1}}%
      \expandafter\def\csname LT5\endcsname{\color[rgb]{1,1,0}}%
      \expandafter\def\csname LT6\endcsname{\color[rgb]{0,0,0}}%
      \expandafter\def\csname LT7\endcsname{\color[rgb]{1,0.3,0}}%
      \expandafter\def\csname LT8\endcsname{\color[rgb]{0.5,0.5,0.5}}%
    \else
      \def\colorrgb#1{\color{black}}%
      \def\colorgray#1{\color[gray]{#1}}%
      \expandafter\def\csname LTw\endcsname{\color{white}}%
      \expandafter\def\csname LTb\endcsname{\color{black}}%
      \expandafter\def\csname LTa\endcsname{\color{black}}%
      \expandafter\def\csname LT0\endcsname{\color{black}}%
      \expandafter\def\csname LT1\endcsname{\color{black}}%
      \expandafter\def\csname LT2\endcsname{\color{black}}%
      \expandafter\def\csname LT3\endcsname{\color{black}}%
      \expandafter\def\csname LT4\endcsname{\color{black}}%
      \expandafter\def\csname LT5\endcsname{\color{black}}%
      \expandafter\def\csname LT6\endcsname{\color{black}}%
      \expandafter\def\csname LT7\endcsname{\color{black}}%
      \expandafter\def\csname LT8\endcsname{\color{black}}%
    \fi
  \fi
    \setlength{\unitlength}{0.0500bp}%
    \ifx\gptboxheight\undefined%
      \newlength{\gptboxheight}%
      \newlength{\gptboxwidth}%
      \newsavebox{\gptboxtext}%
    \fi%
    \setlength{\fboxrule}{0.5pt}%
    \setlength{\fboxsep}{1pt}%
    \definecolor{tbcol}{rgb}{1,1,1}%
\begin{picture}(7560.00,4520.00)%
    \gplgaddtomacro\gplbacktext{%
      \csname LTb\endcsname
      \put(894,930){\makebox(0,0)[r]{\strut{}$10^{-16}$}}%
      \csname LTb\endcsname
      \put(894,1475){\makebox(0,0)[r]{\strut{}$10^{-15}$}}%
      \csname LTb\endcsname
      \put(894,2020){\makebox(0,0)[r]{\strut{}$10^{-14}$}}%
      \csname LTb\endcsname
      \put(894,2564){\makebox(0,0)[r]{\strut{}$10^{-13}$}}%
      \csname LTb\endcsname
      \put(894,3109){\makebox(0,0)[r]{\strut{}$10^{-12}$}}%
      \csname LTb\endcsname
      \put(894,3653){\makebox(0,0)[r]{\strut{}$10^{-11}$}}%
      \csname LTb\endcsname
      \put(894,4198){\makebox(0,0)[r]{\strut{}$10^{-10}$}}%
      \csname LTb\endcsname
      \put(2243,527){\makebox(0,0){\strut{}3x3x6}}%
      \csname LTb\endcsname
      \put(3492,527){\makebox(0,0){\strut{}9x9x36}}%
      \csname LTb\endcsname
      \put(4740,527){\makebox(0,0){\strut{}27x27x216}}%
      \csname LTb\endcsname
      \put(5989,527){\makebox(0,0){\strut{}81x81x1296}}%
    }%
    \gplgaddtomacro\gplfronttext{%
      \csname LTb\endcsname
      \put(1795,4044){\makebox(0,0)[l]{\strut{}Fast 3x3x6, Alt. B.}}%
      \csname LTb\endcsname
      \put(1795,3804){\makebox(0,0)[l]{\strut{}Fast 3x3x6}}%
      \csname LTb\endcsname
      \put(1795,3565){\makebox(0,0)[l]{\strut{}Fast and accurate 3x3x6, Alt. B.}}%
      \csname LTb\endcsname
      \put(1795,3325){\makebox(0,0)[l]{\strut{}Fast and accurate 3x3x6}}%
      \csname LTb\endcsname
      \put(1795,3085){\makebox(0,0)[l]{\strut{}Conventional 3x3x6}}%
      \csname LTb\endcsname
      \put(195,2513){\rotatebox{90.00}{\makebox(0,0){\strut{}error}}}%
      \csname LTb\endcsname
      \put(4116,167){\makebox(0,0){\strut{}matrix dimensions}}%
    }%
    \gplbacktext
    \put(0,0){\includegraphics[width={378.00bp},height={226.00bp}]{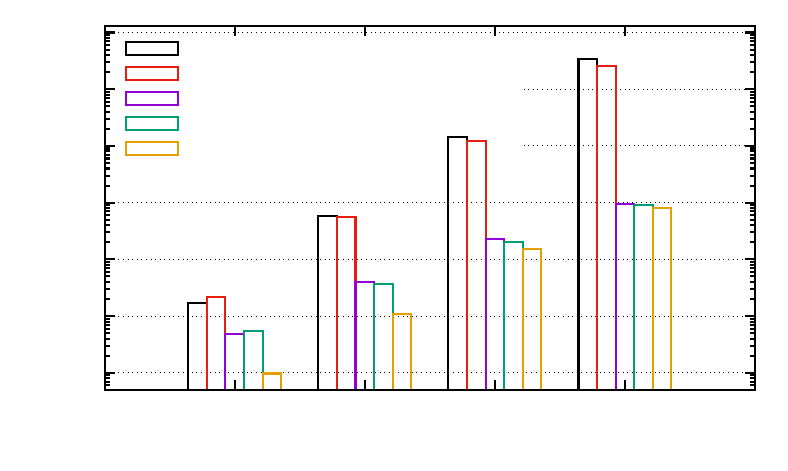}}%
    \gplfronttext
  \end{picture}%
\endgroup
\end{figure}

\begin{figure}[!htbp]\centering
\caption{Numerical accuracy of $\FMMA{6}{3}{3}{40}$
  algorithms}\label{fig:633}
\begingroup
  \fontfamily{1}%
  \selectfont
  \makeatletter
  \providecommand\color[2][]{%
    \GenericError{(gnuplot) \space\space\space\@spaces}{%
      Package color not loaded in conjunction with
      terminal option `colourtext'%
    }{See the gnuplot documentation for explanation.%
    }{Either use 'blacktext' in gnuplot or load the package
      color.sty in LaTeX.}%
    \renewcommand\color[2][]{}%
  }%
  \providecommand\includegraphics[2][]{%
    \GenericError{(gnuplot) \space\space\space\@spaces}{%
      Package graphicx or graphics not loaded%
    }{See the gnuplot documentation for explanation.%
    }{The gnuplot epslatex terminal needs graphicx.sty or graphics.sty.}%
    \renewcommand\includegraphics[2][]{}%
  }%
  \providecommand\rotatebox[2]{#2}%
  \@ifundefined{ifGPcolor}{%
    \newif\ifGPcolor
    \GPcolortrue
  }{}%
  \@ifundefined{ifGPblacktext}{%
    \newif\ifGPblacktext
    \GPblacktexttrue
  }{}%
  \let\gplgaddtomacro\g@addto@macro
  \gdef\gplbacktext{}%
  \gdef\gplfronttext{}%
  \makeatother
  \ifGPblacktext
    \def\colorrgb#1{}%
    \def\colorgray#1{}%
  \else
    \ifGPcolor
      \def\colorrgb#1{\color[rgb]{#1}}%
      \def\colorgray#1{\color[gray]{#1}}%
      \expandafter\def\csname LTw\endcsname{\color{white}}%
      \expandafter\def\csname LTb\endcsname{\color{black}}%
      \expandafter\def\csname LTa\endcsname{\color{black}}%
      \expandafter\def\csname LT0\endcsname{\color[rgb]{1,0,0}}%
      \expandafter\def\csname LT1\endcsname{\color[rgb]{0,1,0}}%
      \expandafter\def\csname LT2\endcsname{\color[rgb]{0,0,1}}%
      \expandafter\def\csname LT3\endcsname{\color[rgb]{1,0,1}}%
      \expandafter\def\csname LT4\endcsname{\color[rgb]{0,1,1}}%
      \expandafter\def\csname LT5\endcsname{\color[rgb]{1,1,0}}%
      \expandafter\def\csname LT6\endcsname{\color[rgb]{0,0,0}}%
      \expandafter\def\csname LT7\endcsname{\color[rgb]{1,0.3,0}}%
      \expandafter\def\csname LT8\endcsname{\color[rgb]{0.5,0.5,0.5}}%
    \else
      \def\colorrgb#1{\color{black}}%
      \def\colorgray#1{\color[gray]{#1}}%
      \expandafter\def\csname LTw\endcsname{\color{white}}%
      \expandafter\def\csname LTb\endcsname{\color{black}}%
      \expandafter\def\csname LTa\endcsname{\color{black}}%
      \expandafter\def\csname LT0\endcsname{\color{black}}%
      \expandafter\def\csname LT1\endcsname{\color{black}}%
      \expandafter\def\csname LT2\endcsname{\color{black}}%
      \expandafter\def\csname LT3\endcsname{\color{black}}%
      \expandafter\def\csname LT4\endcsname{\color{black}}%
      \expandafter\def\csname LT5\endcsname{\color{black}}%
      \expandafter\def\csname LT6\endcsname{\color{black}}%
      \expandafter\def\csname LT7\endcsname{\color{black}}%
      \expandafter\def\csname LT8\endcsname{\color{black}}%
    \fi
  \fi
    \setlength{\unitlength}{0.0500bp}%
    \ifx\gptboxheight\undefined%
      \newlength{\gptboxheight}%
      \newlength{\gptboxwidth}%
      \newsavebox{\gptboxtext}%
    \fi%
    \setlength{\fboxrule}{0.5pt}%
    \setlength{\fboxsep}{1pt}%
    \definecolor{tbcol}{rgb}{1,1,1}%
\begin{picture}(7560.00,4520.00)%
    \gplgaddtomacro\gplbacktext{%
      \csname LTb\endcsname
      \put(894,930){\makebox(0,0)[r]{\strut{}$10^{-16}$}}%
      \csname LTb\endcsname
      \put(894,1475){\makebox(0,0)[r]{\strut{}$10^{-15}$}}%
      \csname LTb\endcsname
      \put(894,2020){\makebox(0,0)[r]{\strut{}$10^{-14}$}}%
      \csname LTb\endcsname
      \put(894,2564){\makebox(0,0)[r]{\strut{}$10^{-13}$}}%
      \csname LTb\endcsname
      \put(894,3109){\makebox(0,0)[r]{\strut{}$10^{-12}$}}%
      \csname LTb\endcsname
      \put(894,3653){\makebox(0,0)[r]{\strut{}$10^{-11}$}}%
      \csname LTb\endcsname
      \put(894,4198){\makebox(0,0)[r]{\strut{}$10^{-10}$}}%
      \csname LTb\endcsname
      \put(2243,527){\makebox(0,0){\strut{}6x3x3}}%
      \csname LTb\endcsname
      \put(3492,527){\makebox(0,0){\strut{}36x9x9}}%
      \csname LTb\endcsname
      \put(4740,527){\makebox(0,0){\strut{}216x27x27}}%
      \csname LTb\endcsname
      \put(5989,527){\makebox(0,0){\strut{}1296x81x81}}%
    }%
    \gplgaddtomacro\gplfronttext{%
      \csname LTb\endcsname
      \put(1795,4044){\makebox(0,0)[l]{\strut{}Fast 6x3x3, Alt. B.}}%
      \csname LTb\endcsname
      \put(1795,3804){\makebox(0,0)[l]{\strut{}Fast 6x3x3}}%
      \csname LTb\endcsname
      \put(1795,3565){\makebox(0,0)[l]{\strut{}Fast and accurate 6x3x3, Alt. B.}}%
      \csname LTb\endcsname
      \put(1795,3325){\makebox(0,0)[l]{\strut{}Fast and accurate 6x3x3}}%
      \csname LTb\endcsname
      \put(1795,3085){\makebox(0,0)[l]{\strut{}Conventional 6x3x3}}%
      \csname LTb\endcsname
      \put(195,2513){\rotatebox{90.00}{\makebox(0,0){\strut{}error}}}%
      \csname LTb\endcsname
      \put(4116,167){\makebox(0,0){\strut{}matrix dimensions}}%
    }%
    \gplbacktext
    \put(0,0){\includegraphics[width={378.00bp},height={226.00bp}]{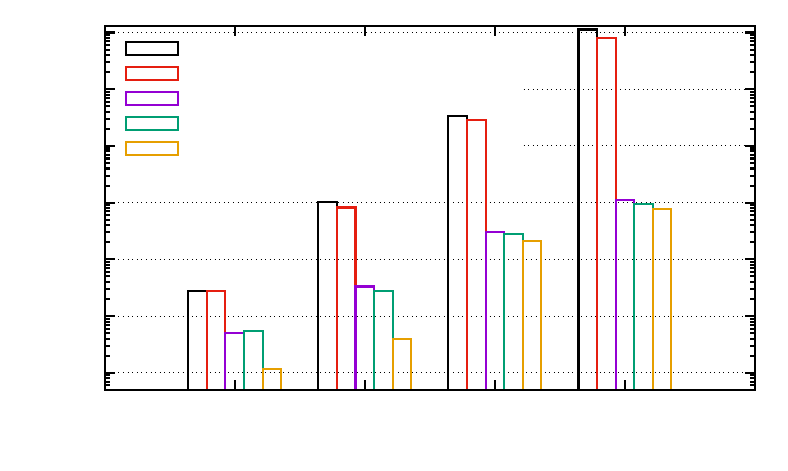}}%
    \gplfronttext
  \end{picture}%
\endgroup
\end{figure}

\begin{figure}[!htbp]\centering
\caption{Numerical accuracy of $\FMMA{3}{6}{3}{40}$
  algorithms}\label{fig:363}
\begingroup
  \fontfamily{1}%
  \selectfont
  \makeatletter
  \providecommand\color[2][]{%
    \GenericError{(gnuplot) \space\space\space\@spaces}{%
      Package color not loaded in conjunction with
      terminal option `colourtext'%
    }{See the gnuplot documentation for explanation.%
    }{Either use 'blacktext' in gnuplot or load the package
      color.sty in LaTeX.}%
    \renewcommand\color[2][]{}%
  }%
  \providecommand\includegraphics[2][]{%
    \GenericError{(gnuplot) \space\space\space\@spaces}{%
      Package graphicx or graphics not loaded%
    }{See the gnuplot documentation for explanation.%
    }{The gnuplot epslatex terminal needs graphicx.sty or graphics.sty.}%
    \renewcommand\includegraphics[2][]{}%
  }%
  \providecommand\rotatebox[2]{#2}%
  \@ifundefined{ifGPcolor}{%
    \newif\ifGPcolor
    \GPcolortrue
  }{}%
  \@ifundefined{ifGPblacktext}{%
    \newif\ifGPblacktext
    \GPblacktexttrue
  }{}%
  \let\gplgaddtomacro\g@addto@macro
  \gdef\gplbacktext{}%
  \gdef\gplfronttext{}%
  \makeatother
  \ifGPblacktext
    \def\colorrgb#1{}%
    \def\colorgray#1{}%
  \else
    \ifGPcolor
      \def\colorrgb#1{\color[rgb]{#1}}%
      \def\colorgray#1{\color[gray]{#1}}%
      \expandafter\def\csname LTw\endcsname{\color{white}}%
      \expandafter\def\csname LTb\endcsname{\color{black}}%
      \expandafter\def\csname LTa\endcsname{\color{black}}%
      \expandafter\def\csname LT0\endcsname{\color[rgb]{1,0,0}}%
      \expandafter\def\csname LT1\endcsname{\color[rgb]{0,1,0}}%
      \expandafter\def\csname LT2\endcsname{\color[rgb]{0,0,1}}%
      \expandafter\def\csname LT3\endcsname{\color[rgb]{1,0,1}}%
      \expandafter\def\csname LT4\endcsname{\color[rgb]{0,1,1}}%
      \expandafter\def\csname LT5\endcsname{\color[rgb]{1,1,0}}%
      \expandafter\def\csname LT6\endcsname{\color[rgb]{0,0,0}}%
      \expandafter\def\csname LT7\endcsname{\color[rgb]{1,0.3,0}}%
      \expandafter\def\csname LT8\endcsname{\color[rgb]{0.5,0.5,0.5}}%
    \else
      \def\colorrgb#1{\color{black}}%
      \def\colorgray#1{\color[gray]{#1}}%
      \expandafter\def\csname LTw\endcsname{\color{white}}%
      \expandafter\def\csname LTb\endcsname{\color{black}}%
      \expandafter\def\csname LTa\endcsname{\color{black}}%
      \expandafter\def\csname LT0\endcsname{\color{black}}%
      \expandafter\def\csname LT1\endcsname{\color{black}}%
      \expandafter\def\csname LT2\endcsname{\color{black}}%
      \expandafter\def\csname LT3\endcsname{\color{black}}%
      \expandafter\def\csname LT4\endcsname{\color{black}}%
      \expandafter\def\csname LT5\endcsname{\color{black}}%
      \expandafter\def\csname LT6\endcsname{\color{black}}%
      \expandafter\def\csname LT7\endcsname{\color{black}}%
      \expandafter\def\csname LT8\endcsname{\color{black}}%
    \fi
  \fi
    \setlength{\unitlength}{0.0500bp}%
    \ifx\gptboxheight\undefined%
      \newlength{\gptboxheight}%
      \newlength{\gptboxwidth}%
      \newsavebox{\gptboxtext}%
    \fi%
    \setlength{\fboxrule}{0.5pt}%
    \setlength{\fboxsep}{1pt}%
    \definecolor{tbcol}{rgb}{1,1,1}%
\begin{picture}(7560.00,4520.00)%
    \gplgaddtomacro\gplbacktext{%
      \csname LTb\endcsname
      \put(894,930){\makebox(0,0)[r]{\strut{}$10^{-16}$}}%
      \csname LTb\endcsname
      \put(894,1475){\makebox(0,0)[r]{\strut{}$10^{-15}$}}%
      \csname LTb\endcsname
      \put(894,2020){\makebox(0,0)[r]{\strut{}$10^{-14}$}}%
      \csname LTb\endcsname
      \put(894,2564){\makebox(0,0)[r]{\strut{}$10^{-13}$}}%
      \csname LTb\endcsname
      \put(894,3109){\makebox(0,0)[r]{\strut{}$10^{-12}$}}%
      \csname LTb\endcsname
      \put(894,3653){\makebox(0,0)[r]{\strut{}$10^{-11}$}}%
      \csname LTb\endcsname
      \put(894,4198){\makebox(0,0)[r]{\strut{}$10^{-10}$}}%
      \csname LTb\endcsname
      \put(2243,527){\makebox(0,0){\strut{}3x6x3}}%
      \csname LTb\endcsname
      \put(3492,527){\makebox(0,0){\strut{}9x36x9}}%
      \csname LTb\endcsname
      \put(4740,527){\makebox(0,0){\strut{}27x216x27}}%
      \csname LTb\endcsname
      \put(5989,527){\makebox(0,0){\strut{}81x1296x81}}%
    }%
    \gplgaddtomacro\gplfronttext{%
      \csname LTb\endcsname
      \put(1795,4044){\makebox(0,0)[l]{\strut{}Fast 3x6x3, Alt. B.}}%
      \csname LTb\endcsname
      \put(1795,3804){\makebox(0,0)[l]{\strut{}Fast 3x6x3}}%
      \csname LTb\endcsname
      \put(1795,3565){\makebox(0,0)[l]{\strut{}Fast and accurate 3x6x3, Alt. B.}}%
      \csname LTb\endcsname
      \put(1795,3325){\makebox(0,0)[l]{\strut{}Fast and accurate 3x6x3}}%
      \csname LTb\endcsname
      \put(1795,3085){\makebox(0,0)[l]{\strut{}Conventional 3x6x3}}%
      \csname LTb\endcsname
      \put(195,2513){\rotatebox{90.00}{\makebox(0,0){\strut{}error}}}%
      \csname LTb\endcsname
      \put(4116,167){\makebox(0,0){\strut{}matrix dimensions}}%
    }%
    \gplbacktext
    \put(0,0){\includegraphics[width={378.00bp},height={226.00bp}]{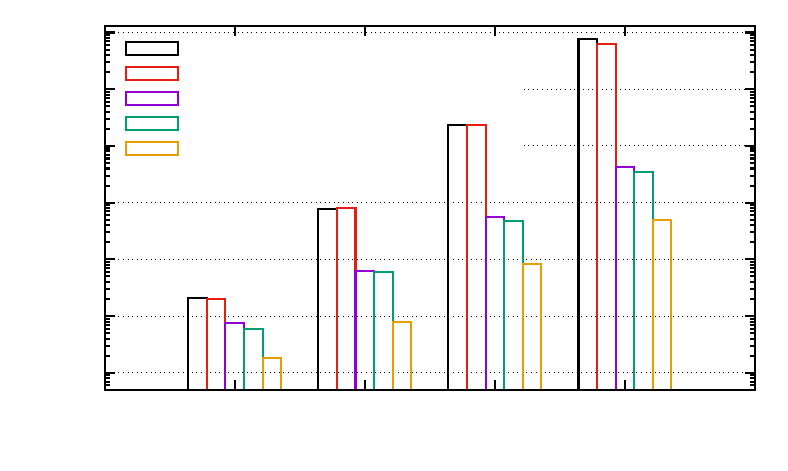}}%
    \gplfronttext
  \end{picture}%
\endgroup
\end{figure}

\Cref{fig:336,fig:363,fig:633} show that all the accurate versions
are very close to the conventional algorithm in terms of accuracy, and
much better than the original versions.
The figures also show that their implementation using an alternative
basis does not really modify the accuracy.

\begin{figure}[!htbp]\centering
\caption{Numerical accuracy when mixing $\FMMA{6}{3}{3}{40}$,
  $\FMMA{3}{6}{3}{40}$ and $\FMMA{3}{3}{6}{40}$
  algorithms}\label{fig:54}
\begingroup
  \fontfamily{1}%
  \selectfont
  \makeatletter
  \providecommand\color[2][]{%
    \GenericError{(gnuplot) \space\space\space\@spaces}{%
      Package color not loaded in conjunction with
      terminal option `colourtext'%
    }{See the gnuplot documentation for explanation.%
    }{Either use 'blacktext' in gnuplot or load the package
      color.sty in LaTeX.}%
    \renewcommand\color[2][]{}%
  }%
  \providecommand\includegraphics[2][]{%
    \GenericError{(gnuplot) \space\space\space\@spaces}{%
      Package graphicx or graphics not loaded%
    }{See the gnuplot documentation for explanation.%
    }{The gnuplot epslatex terminal needs graphicx.sty or graphics.sty.}%
    \renewcommand\includegraphics[2][]{}%
  }%
  \providecommand\rotatebox[2]{#2}%
  \@ifundefined{ifGPcolor}{%
    \newif\ifGPcolor
    \GPcolortrue
  }{}%
  \@ifundefined{ifGPblacktext}{%
    \newif\ifGPblacktext
    \GPblacktexttrue
  }{}%
  \let\gplgaddtomacro\g@addto@macro
  \gdef\gplbacktext{}%
  \gdef\gplfronttext{}%
  \makeatother
  \ifGPblacktext
    \def\colorrgb#1{}%
    \def\colorgray#1{}%
  \else
    \ifGPcolor
      \def\colorrgb#1{\color[rgb]{#1}}%
      \def\colorgray#1{\color[gray]{#1}}%
      \expandafter\def\csname LTw\endcsname{\color{white}}%
      \expandafter\def\csname LTb\endcsname{\color{black}}%
      \expandafter\def\csname LTa\endcsname{\color{black}}%
      \expandafter\def\csname LT0\endcsname{\color[rgb]{1,0,0}}%
      \expandafter\def\csname LT1\endcsname{\color[rgb]{0,1,0}}%
      \expandafter\def\csname LT2\endcsname{\color[rgb]{0,0,1}}%
      \expandafter\def\csname LT3\endcsname{\color[rgb]{1,0,1}}%
      \expandafter\def\csname LT4\endcsname{\color[rgb]{0,1,1}}%
      \expandafter\def\csname LT5\endcsname{\color[rgb]{1,1,0}}%
      \expandafter\def\csname LT6\endcsname{\color[rgb]{0,0,0}}%
      \expandafter\def\csname LT7\endcsname{\color[rgb]{1,0.3,0}}%
      \expandafter\def\csname LT8\endcsname{\color[rgb]{0.5,0.5,0.5}}%
    \else
      \def\colorrgb#1{\color{black}}%
      \def\colorgray#1{\color[gray]{#1}}%
      \expandafter\def\csname LTw\endcsname{\color{white}}%
      \expandafter\def\csname LTb\endcsname{\color{black}}%
      \expandafter\def\csname LTa\endcsname{\color{black}}%
      \expandafter\def\csname LT0\endcsname{\color{black}}%
      \expandafter\def\csname LT1\endcsname{\color{black}}%
      \expandafter\def\csname LT2\endcsname{\color{black}}%
      \expandafter\def\csname LT3\endcsname{\color{black}}%
      \expandafter\def\csname LT4\endcsname{\color{black}}%
      \expandafter\def\csname LT5\endcsname{\color{black}}%
      \expandafter\def\csname LT6\endcsname{\color{black}}%
      \expandafter\def\csname LT7\endcsname{\color{black}}%
      \expandafter\def\csname LT8\endcsname{\color{black}}%
    \fi
  \fi
    \setlength{\unitlength}{0.0500bp}%
    \ifx\gptboxheight\undefined%
      \newlength{\gptboxheight}%
      \newlength{\gptboxwidth}%
      \newsavebox{\gptboxtext}%
    \fi%
    \setlength{\fboxrule}{0.5pt}%
    \setlength{\fboxsep}{1pt}%
    \definecolor{tbcol}{rgb}{1,1,1}%
\begin{picture}(7560.00,4520.00)%
    \gplgaddtomacro\gplbacktext{%
      \csname LTb\endcsname
      \put(894,930){\makebox(0,0)[r]{\strut{}$10^{-16}$}}%
      \csname LTb\endcsname
      \put(894,1475){\makebox(0,0)[r]{\strut{}$10^{-15}$}}%
      \csname LTb\endcsname
      \put(894,2020){\makebox(0,0)[r]{\strut{}$10^{-14}$}}%
      \csname LTb\endcsname
      \put(894,2564){\makebox(0,0)[r]{\strut{}$10^{-13}$}}%
      \csname LTb\endcsname
      \put(894,3109){\makebox(0,0)[r]{\strut{}$10^{-12}$}}%
      \csname LTb\endcsname
      \put(894,3653){\makebox(0,0)[r]{\strut{}$10^{-11}$}}%
      \csname LTb\endcsname
      \put(894,4198){\makebox(0,0)[r]{\strut{}$10^{-10}$}}%
      \csname LTb\endcsname
      \put(2243,527){\makebox(0,0){\strut{}$6{	imes}3{	imes}3$}}%
      \csname LTb\endcsname
      \put(3492,527){\makebox(0,0){\strut{}18x9x18}}%
      \csname LTb\endcsname
      \put(4740,527){\makebox(0,0){\strut{}54x54x54}}%
      \csname LTb\endcsname
      \put(5989,527){\makebox(0,0){\strut{}324x162x162}}%
    }%
    \gplgaddtomacro\gplfronttext{%
      \csname LTb\endcsname
      \put(1795,4044){\makebox(0,0)[l]{\strut{}Fast}}%
      \csname LTb\endcsname
      \put(1795,3804){\makebox(0,0)[l]{\strut{}Fast and accurate}}%
      \csname LTb\endcsname
      \put(1795,3565){\makebox(0,0)[l]{\strut{}Conventional}}%
      \csname LTb\endcsname
      \put(195,2513){\rotatebox{90.00}{\makebox(0,0){\strut{}error}}}%
      \csname LTb\endcsname
      \put(4116,167){\makebox(0,0){\strut{}matrix dimensions}}%
    }%
    \gplbacktext
    \put(0,0){\includegraphics[width={378.00bp},height={226.00bp}]{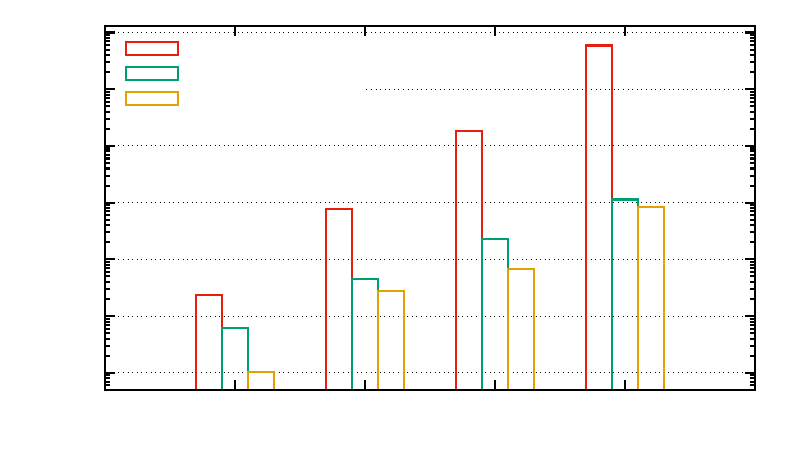}}%
    \gplfronttext
  \end{picture}%
\endgroup
\end{figure}
Finally, \cref{fig:54} shows the accuracy obtained when mixing these algorithms: we present there a variant that recursively calls one between the following~\({\FMMA{6}{3}{3}{40},\FMMA{3}{6}{3}{40}}\) and~\(\FMMA{3}{3}{6}{40}\) algorithms, depending on the rectangular pattern of the input.
For instance on~\({{54}\times{54}}\) matrices, it recursively calls first~\(\FMMA{6}{3}{3}{40}\) on~\({{9}\times{18}}\) by~\({{18}\times{18}}\) blocks, then~\(\FMMA{3}{3}{6}{40}\) on~\({{3}\times{6}}\) by~\({{6}\times{3}}\) blocks, at end with~\(\FMMA{3}{6}{3}{40}\) on leaves.
Again the accurate variants are largely more precise while their error remains on the order of magnitude of the conventional one.
\section{Conclusion and future work}
We have presented a technique and a software to find more accurate
recursive fast matrix multiplication algorithms.
Our analysis shows that our most accurate~\(\FMMA{2}{2}{2}{7}\)
formulas are probably optimal with respect to the tensor nuclear
(Frobenius) norm.
We still anyway have a potential gap of at most 2.6\% to further
explore in this case.
More generally, we could improve the orbit exploration. For Strassen's
like algorithm there is a single orbit but this is not true in
general.
For instance with the Smirnov's families of \(\FMMA{3}{3}{6}{40}\) we
where only able to explore some small portions of the orbit space.
Therefore, some more directed explorations, along flip-graphs for
instance~\cite{Kauers:2023:flipgraphs}, could be a possibility.
\par
Then, to further improve their practical behavior, as done, e.g.,
in~\cite[\S~4.3]{demmel:2007b}, \cite[\S~6.1]{ballard:2012a}
or~\cite[\S~6]{BBDLS16}, some diagonal scaling adapted to specific
input matrices can be added to any algorithms and thus to any of the
variants presented here.
The idea is to precondition the input matrices to be multiplied.
For this, consider \Cref{lem:sandwiching}, and find well suited ad-hoc diagonal matrices~\({\mat{U},\mat{V},\mat{W}}\) that reduce the norm of the elements in either of the input.
Obviously, these pre-processing strategies could also be applied here
to further improve the practical accuracy.
\par
We also have shown ways to optimize the time complexity of recursive
matrix product variants to simultaneously obtain a better accuracy in
practice and a time complexity bound with the best currently known
leading term (obtained via alternative basis sparsification).
There remains to investigate how these variants behave in practice compared to state of the art
implementations.
\par
Also, isotropies play a central role in this matter as shown by the
fact that the minimal growth factor reached for
the~\(\matrixsize{2}{2}\) formulas in this work is exactly the same as
that of the algorithm obtained by~\cite[Eq.~(22)]{Grochow:2016aa},
while reconstructing Strassen's algorithm, using only the knowledge of
its stabilizer and its representation with minimal Frobenius norms.

\bibliographystyle{elsarticle-num-names}
\bibliography{strassenaccurate}

\newpage
\begin{sidewaystable*}[htbp]\centering
\scalebox{.8}{\(
\begin{smatrix}
1&0&-1&0&1&-1&-1&1&0\\
1&0&1&0&1&-1&-1&-1&0\\
1&0&1&0&1&1&1&-1&0\\
1&0&-1&0&1&1&1&1&0\\
-1&0&-1&0&1&-1&1&1&0\\
-1&0&-1&0&1&1&-1&1&0\\
-1&0&1&0&1&-1&1&-1&0\\
-1&0&1&0&1&1&-1&-1&0\\
-1&0&1&0&1&-1&-1&1&0\\
-1&0&-1&0&1&1&1&-1&0\\
1&0&1&0&-1&1&1&1&0\\
-1&0&1&0&1&1&1&1&0\\
1&0&1&0&1&1&-1&1&0\\
1&0&1&0&1&-1&1&1&0\\
-1&0&1&0&-1&-1&1&1&0\\
-1&0&1&0&-1&1&-1&1&0\\
1&0&1&0&0&0&-1&0&-1\\
-1&0&-1&0&0&0&-1&0&-1\\
1&0&-1&0&0&0&1&0&-1\\
1&0&-1&0&0&0&-1&0&1\\
1&1&0&0&0&0&-1&-1&0\\
1&-1&0&0&0&0&1&-1&0\\
1&1&0&0&0&0&1&1&0\\
-1&1&0&0&0&0&1&-1&0\\
0&0&0&-1&-1&0&-1&-1&0\\
0&0&0&-1&1&0&-1&1&0\\
0&0&0&-1&1&0&1&-1&0\\
0&0&0&-1&-1&0&1&1&0\\
0&0&0&0&-1&1&0&-1&1\\
0&0&0&0&-1&-1&0&-1&-1\\
0&0&0&0&-1&-1&0&1&1\\
0&0&0&0&1&-1&0&-1&1\\
0&1&1&0&1&1&0&0&0\\
0&-1&1&0&1&-1&0&0&0\\
0&1&-1&0&1&-1&0&0&0\\
0&-1&-1&0&1&1&0&0&0\\
-1&0&1&1&0&-1&0&0&0\\
-1&0&-1&-1&0&-1&0&0&0\\
1&0&-1&1&0&-1&0&0&0\\
-1&0&-1&1&0&1&0&0&0\\
\end{smatrix}
\quad
\begin{smatrix}
0&0&-\frac{1}{4}&-\frac{1}{4}&-\frac{1}{4}&0&0&\frac{1}{4}&0&0&\frac{1}{4}&-\frac{1}{4}&-\frac{1}{4}&-\frac{1}{4}&0&\frac{1}{4}&0&0\\
0&0&-\frac{1}{4}&\frac{1}{4}&\frac{1}{4}&0&0&-\frac{1}{4}&0&0&\frac{1}{4}&-\frac{1}{4}&-\frac{1}{4}&\frac{1}{4}&0&\frac{1}{4}&0&0\\
0&0&-\frac{1}{4}&-\frac{1}{4}&\frac{1}{4}&0&0&\frac{1}{4}&0&0&-\frac{1}{4}&-\frac{1}{4}&-\frac{1}{4}&\frac{1}{4}&0&-\frac{1}{4}&0&0\\
0&0&\frac{1}{4}&-\frac{1}{4}&\frac{1}{4}&0&0&\frac{1}{4}&0&0&\frac{1}{4}&\frac{1}{4}&\frac{1}{4}&\frac{1}{4}&0&\frac{1}{4}&0&0\\
-\frac{1}{4}&0&\frac{1}{4}&0&0&-\frac{1}{4}&-\frac{1}{4}&\frac{1}{4}&0&\frac{1}{4}&0&0&0&-\frac{1}{4}&0&0&\frac{1}{4}&-\frac{1}{4}\\
-\frac{1}{4}&0&\frac{1}{4}&0&0&\frac{1}{4}&\frac{1}{4}&-\frac{1}{4}&0&\frac{1}{4}&0&0&0&-\frac{1}{4}&0&0&\frac{1}{4}&\frac{1}{4}\\
-\frac{1}{4}&0&-\frac{1}{4}&0&0&-\frac{1}{4}&\frac{1}{4}&\frac{1}{4}&0&-\frac{1}{4}&0&0&0&-\frac{1}{4}&0&0&-\frac{1}{4}&\frac{1}{4}\\
\frac{1}{4}&0&\frac{1}{4}&0&0&-\frac{1}{4}&\frac{1}{4}&\frac{1}{4}&0&\frac{1}{4}&0&0&0&\frac{1}{4}&0&0&\frac{1}{4}&\frac{1}{4}\\
-\frac{1}{4}&-\frac{1}{4}&0&\frac{1}{4}&0&0&\frac{1}{4}&0&\frac{1}{4}&0&0&\frac{1}{4}&0&0&-\frac{1}{4}&-\frac{1}{4}&-\frac{1}{4}&0\\
-\frac{1}{4}&\frac{1}{4}&0&-\frac{1}{4}&0&0&\frac{1}{4}&0&-\frac{1}{4}&0&0&-\frac{1}{4}&0&0&-\frac{1}{4}&-\frac{1}{4}&\frac{1}{4}&0\\
-\frac{1}{4}&\frac{1}{4}&0&\frac{1}{4}&0&0&-\frac{1}{4}&0&\frac{1}{4}&0&0&-\frac{1}{4}&0&0&-\frac{1}{4}&\frac{1}{4}&\frac{1}{4}&0\\
\frac{1}{4}&\frac{1}{4}&0&\frac{1}{4}&0&0&\frac{1}{4}&0&\frac{1}{4}&0&0&-\frac{1}{4}&0&0&\frac{1}{4}&-\frac{1}{4}&\frac{1}{4}&0\\
0&-\frac{1}{4}&0&0&\frac{1}{4}&\frac{1}{4}&0&0&\frac{1}{4}&\frac{1}{4}&-\frac{1}{4}&0&-\frac{1}{4}&0&\frac{1}{4}&0&0&\frac{1}{4}\\
0&\frac{1}{4}&0&0&-\frac{1}{4}&\frac{1}{4}&0&0&\frac{1}{4}&-\frac{1}{4}&-\frac{1}{4}&0&\frac{1}{4}&0&-\frac{1}{4}&0&0&\frac{1}{4}\\
0&-\frac{1}{4}&0&0&-\frac{1}{4}&-\frac{1}{4}&0&0&-\frac{1}{4}&\frac{1}{4}&-\frac{1}{4}&0&-\frac{1}{4}&0&-\frac{1}{4}&0&0&\frac{1}{4}\\
0&-\frac{1}{4}&0&0&-\frac{1}{4}&\frac{1}{4}&0&0&\frac{1}{4}&\frac{1}{4}&\frac{1}{4}&0&-\frac{1}{4}&0&-\frac{1}{4}&0&0&-\frac{1}{4}\\
0&0&0&\frac{1}{4}&0&\frac{1}{4}&0&-\frac{1}{4}&\frac{1}{4}&0&0&0&-\frac{1}{4}&\frac{1}{4}&\frac{1}{4}&\frac{1}{4}&-\frac{1}{4}&\frac{1}{4}\\
0&0&0&\frac{1}{4}&0&\frac{1}{4}&0&-\frac{1}{4}&\frac{1}{4}&0&0&0&\frac{1}{4}&-\frac{1}{4}&-\frac{1}{4}&\frac{1}{4}&\frac{1}{4}&\frac{1}{4}\\
0&0&0&\frac{1}{4}&0&\frac{1}{4}&0&-\frac{1}{4}&\frac{1}{4}&0&0&0&-\frac{1}{4}&-\frac{1}{4}&-\frac{1}{4}&-\frac{1}{4}&-\frac{1}{4}&-\frac{1}{4}\\
0&0&0&\frac{1}{4}&0&\frac{1}{4}&0&-\frac{1}{4}&\frac{1}{4}&0&0&0&\frac{1}{4}&\frac{1}{4}&\frac{1}{4}&-\frac{1}{4}&\frac{1}{4}&-\frac{1}{4}\\
-\frac{1}{4}&0&0&0&-\frac{1}{4}&0&-\frac{1}{4}&\frac{1}{4}&-\frac{1}{4}&\frac{1}{4}&-\frac{1}{4}&\frac{1}{4}&0&-\frac{1}{4}&-\frac{1}{4}&0&0&0\\
-\frac{1}{4}&0&0&0&-\frac{1}{4}&0&\frac{1}{4}&-\frac{1}{4}&\frac{1}{4}&\frac{1}{4}&\frac{1}{4}&\frac{1}{4}&0&-\frac{1}{4}&-\frac{1}{4}&0&0&0\\
-\frac{1}{4}&0&0&0&-\frac{1}{4}&0&-\frac{1}{4}&-\frac{1}{4}&\frac{1}{4}&-\frac{1}{4}&-\frac{1}{4}&-\frac{1}{4}&0&-\frac{1}{4}&-\frac{1}{4}&0&0&0\\
\frac{1}{4}&0&0&0&\frac{1}{4}&0&-\frac{1}{4}&-\frac{1}{4}&\frac{1}{4}&\frac{1}{4}&-\frac{1}{4}&\frac{1}{4}&0&\frac{1}{4}&\frac{1}{4}&0&0&0\\
-\frac{1}{4}&-\frac{1}{4}&\frac{1}{4}&-\frac{1}{4}&\frac{1}{4}&-\frac{1}{4}&-\frac{1}{4}&0&0&0&\frac{1}{4}&0&0&0&0&\frac{1}{4}&0&-\frac{1}{4}\\
-\frac{1}{4}&-\frac{1}{4}&\frac{1}{4}&\frac{1}{4}&\frac{1}{4}&\frac{1}{4}&\frac{1}{4}&0&0&0&-\frac{1}{4}&0&0&0&0&-\frac{1}{4}&0&\frac{1}{4}\\
-\frac{1}{4}&\frac{1}{4}&-\frac{1}{4}&-\frac{1}{4}&\frac{1}{4}&-\frac{1}{4}&\frac{1}{4}&0&0&0&-\frac{1}{4}&0&0&0&0&-\frac{1}{4}&0&\frac{1}{4}\\
\frac{1}{4}&-\frac{1}{4}&\frac{1}{4}&-\frac{1}{4}&-\frac{1}{4}&-\frac{1}{4}&\frac{1}{4}&0&0&0&-\frac{1}{4}&0&0&0&0&-\frac{1}{4}&0&\frac{1}{4}\\
0&0&0&\frac{1}{4}&0&\frac{1}{4}&0&-\frac{1}{4}&\frac{1}{4}&0&0&0&\frac{1}{4}&\frac{1}{4}&-\frac{1}{4}&-\frac{1}{4}&-\frac{1}{4}&\frac{1}{4}\\
0&0&0&\frac{1}{4}&0&\frac{1}{4}&0&-\frac{1}{4}&\frac{1}{4}&0&0&0&-\frac{1}{4}&-\frac{1}{4}&\frac{1}{4}&-\frac{1}{4}&\frac{1}{4}&\frac{1}{4}\\
0&0&0&\frac{1}{4}&0&\frac{1}{4}&0&-\frac{1}{4}&\frac{1}{4}&0&0&0&\frac{1}{4}&-\frac{1}{4}&\frac{1}{4}&\frac{1}{4}&-\frac{1}{4}&-\frac{1}{4}\\
0&0&0&\frac{1}{4}&0&\frac{1}{4}&0&-\frac{1}{4}&\frac{1}{4}&0&0&0&-\frac{1}{4}&\frac{1}{4}&-\frac{1}{4}&\frac{1}{4}&\frac{1}{4}&-\frac{1}{4}\\
\frac{1}{4}&0&0&0&\frac{1}{4}&0&\frac{1}{4}&\frac{1}{4}&\frac{1}{4}&\frac{1}{4}&-\frac{1}{4}&-\frac{1}{4}&0&\frac{1}{4}&\frac{1}{4}&0&0&0\\
-\frac{1}{4}&0&0&0&-\frac{1}{4}&0&\frac{1}{4}&\frac{1}{4}&\frac{1}{4}&-\frac{1}{4}&-\frac{1}{4}&\frac{1}{4}&0&-\frac{1}{4}&-\frac{1}{4}&0&0&0\\
-\frac{1}{4}&0&0&0&-\frac{1}{4}&0&-\frac{1}{4}&\frac{1}{4}&\frac{1}{4}&\frac{1}{4}&\frac{1}{4}&-\frac{1}{4}&0&-\frac{1}{4}&-\frac{1}{4}&0&0&0\\
\frac{1}{4}&0&0&0&\frac{1}{4}&0&-\frac{1}{4}&\frac{1}{4}&\frac{1}{4}&-\frac{1}{4}&\frac{1}{4}&\frac{1}{4}&0&\frac{1}{4}&\frac{1}{4}&0&0&0\\
\frac{1}{4}&\frac{1}{4}&\frac{1}{4}&-\frac{1}{4}&\frac{1}{4}&\frac{1}{4}&-\frac{1}{4}&0&0&0&\frac{1}{4}&0&0&0&0&\frac{1}{4}&0&-\frac{1}{4}\\
\frac{1}{4}&-\frac{1}{4}&-\frac{1}{4}&-\frac{1}{4}&\frac{1}{4}&\frac{1}{4}&\frac{1}{4}&0&0&0&-\frac{1}{4}&0&0&0&0&-\frac{1}{4}&0&\frac{1}{4}\\
\frac{1}{4}&\frac{1}{4}&\frac{1}{4}&\frac{1}{4}&\frac{1}{4}&-\frac{1}{4}&\frac{1}{4}&0&0&0&-\frac{1}{4}&0&0&0&0&-\frac{1}{4}&0&\frac{1}{4}\\
-\frac{1}{4}&\frac{1}{4}&\frac{1}{4}&-\frac{1}{4}&-\frac{1}{4}&\frac{1}{4}&\frac{1}{4}&0&0&0&-\frac{1}{4}&0&0&0&0&-\frac{1}{4}&0&\frac{1}{4}\\
\end{smatrix}
\quad
\Transpose{\begin{smatrix}
\frac{1}{2}&0&-\frac{1}{2}&-\frac{1}{2}&0&0&\frac{1}{2}&\frac{1}{2}&0&0&0&-\frac{1}{2}&0&0&\frac{1}{2}&0&\frac{1}{2}&-\frac{1}{2}\\
-\frac{1}{2}&0&-\frac{1}{2}&\frac{1}{2}&0&0&\frac{1}{2}&-\frac{1}{2}&0&0&0&-\frac{1}{2}&0&0&\frac{1}{2}&0&-\frac{1}{2}&\frac{1}{2}\\
-\frac{1}{2}&0&-\frac{1}{2}&-\frac{1}{2}&0&0&-\frac{1}{2}&\frac{1}{2}&0&0&0&-\frac{1}{2}&0&0&-\frac{1}{2}&0&\frac{1}{2}&\frac{1}{2}\\
-\frac{1}{2}&0&\frac{1}{2}&-\frac{1}{2}&0&0&\frac{1}{2}&\frac{1}{2}&0&0&0&\frac{1}{2}&0&0&\frac{1}{2}&0&\frac{1}{2}&\frac{1}{2}\\
0&0&-\frac{1}{2}&0&-\frac{1}{2}&\frac{1}{2}&0&\frac{1}{2}&0&\frac{1}{2}&-\frac{1}{2}&0&-\frac{1}{2}&0&\frac{1}{2}&\frac{1}{2}&0&0\\
0&0&-\frac{1}{2}&0&-\frac{1}{2}&-\frac{1}{2}&0&-\frac{1}{2}&0&\frac{1}{2}&\frac{1}{2}&0&\frac{1}{2}&0&-\frac{1}{2}&\frac{1}{2}&0&0\\
0&0&\frac{1}{2}&0&-\frac{1}{2}&\frac{1}{2}&0&\frac{1}{2}&0&-\frac{1}{2}&\frac{1}{2}&0&-\frac{1}{2}&0&-\frac{1}{2}&\frac{1}{2}&0&0\\
0&0&-\frac{1}{2}&0&\frac{1}{2}&\frac{1}{2}&0&\frac{1}{2}&0&\frac{1}{2}&\frac{1}{2}&0&-\frac{1}{2}&0&-\frac{1}{2}&-\frac{1}{2}&0&0\\
0&\frac{1}{2}&0&-\frac{1}{2}&-\frac{1}{2}&0&0&0&\frac{1}{2}&0&\frac{1}{2}&\frac{1}{2}&\frac{1}{2}&\frac{1}{2}&0&0&0&\frac{1}{2}\\
0&-\frac{1}{2}&0&\frac{1}{2}&-\frac{1}{2}&0&0&0&-\frac{1}{2}&0&\frac{1}{2}&-\frac{1}{2}&-\frac{1}{2}&\frac{1}{2}&0&0&0&\frac{1}{2}\\
0&\frac{1}{2}&0&\frac{1}{2}&\frac{1}{2}&0&0&0&-\frac{1}{2}&0&\frac{1}{2}&\frac{1}{2}&-\frac{1}{2}&\frac{1}{2}&0&0&0&-\frac{1}{2}\\
0&-\frac{1}{2}&0&-\frac{1}{2}&\frac{1}{2}&0&0&0&\frac{1}{2}&0&\frac{1}{2}&-\frac{1}{2}&\frac{1}{2}&\frac{1}{2}&0&0&0&-\frac{1}{2}\\
-\frac{1}{2}&-\frac{1}{2}&0&0&0&\frac{1}{2}&-\frac{1}{2}&0&\frac{1}{2}&\frac{1}{2}&0&0&0&\frac{1}{2}&0&\frac{1}{2}&-\frac{1}{2}&0\\
\frac{1}{2}&\frac{1}{2}&0&0&0&\frac{1}{2}&-\frac{1}{2}&0&\frac{1}{2}&-\frac{1}{2}&0&0&0&\frac{1}{2}&0&-\frac{1}{2}&-\frac{1}{2}&0\\
-\frac{1}{2}&\frac{1}{2}&0&0&0&\frac{1}{2}&\frac{1}{2}&0&\frac{1}{2}&-\frac{1}{2}&0&0&0&-\frac{1}{2}&0&\frac{1}{2}&-\frac{1}{2}&0\\
-\frac{1}{2}&\frac{1}{2}&0&0&0&-\frac{1}{2}&-\frac{1}{2}&0&-\frac{1}{2}&-\frac{1}{2}&0&0&0&\frac{1}{2}&0&\frac{1}{2}&\frac{1}{2}&0\\
0&\frac{1}{2}&\frac{1}{2}&0&0&0&0&0&0&-\frac{1}{2}&0&\frac{1}{2}&\frac{1}{2}&-\frac{1}{2}&-\frac{1}{2}&-\frac{1}{2}&\frac{1}{2}&-\frac{1}{2}\\
0&\frac{1}{2}&\frac{1}{2}&0&0&0&0&0&0&-\frac{1}{2}&0&\frac{1}{2}&-\frac{1}{2}&\frac{1}{2}&\frac{1}{2}&-\frac{1}{2}&-\frac{1}{2}&-\frac{1}{2}\\
0&\frac{1}{2}&\frac{1}{2}&0&0&0&0&0&0&-\frac{1}{2}&0&\frac{1}{2}&\frac{1}{2}&\frac{1}{2}&\frac{1}{2}&\frac{1}{2}&\frac{1}{2}&\frac{1}{2}\\
0&-\frac{1}{2}&-\frac{1}{2}&0&0&0&0&0&0&\frac{1}{2}&0&-\frac{1}{2}&\frac{1}{2}&\frac{1}{2}&\frac{1}{2}&-\frac{1}{2}&\frac{1}{2}&-\frac{1}{2}\\
-\frac{1}{2}&\frac{1}{2}&-\frac{1}{2}&\frac{1}{2}&-\frac{1}{2}&\frac{1}{2}&\frac{1}{2}&0&0&0&-\frac{1}{2}&0&0&-\frac{1}{2}&\frac{1}{2}&0&0&0\\
-\frac{1}{2}&\frac{1}{2}&-\frac{1}{2}&-\frac{1}{2}&-\frac{1}{2}&-\frac{1}{2}&-\frac{1}{2}&0&0&0&\frac{1}{2}&0&0&\frac{1}{2}&-\frac{1}{2}&0&0&0\\
-\frac{1}{2}&-\frac{1}{2}&\frac{1}{2}&-\frac{1}{2}&-\frac{1}{2}&-\frac{1}{2}&\frac{1}{2}&0&0&0&-\frac{1}{2}&0&0&-\frac{1}{2}&\frac{1}{2}&0&0&0\\
-\frac{1}{2}&-\frac{1}{2}&\frac{1}{2}&\frac{1}{2}&-\frac{1}{2}&\frac{1}{2}&-\frac{1}{2}&0&0&0&\frac{1}{2}&0&0&\frac{1}{2}&-\frac{1}{2}&0&0&0\\
-\frac{1}{2}&0&0&0&-\frac{1}{2}&0&\frac{1}{2}&\frac{1}{2}&-\frac{1}{2}&\frac{1}{2}&-\frac{1}{2}&\frac{1}{2}&0&0&0&\frac{1}{2}&0&\frac{1}{2}\\
\frac{1}{2}&0&0&0&\frac{1}{2}&0&\frac{1}{2}&\frac{1}{2}&-\frac{1}{2}&-\frac{1}{2}&-\frac{1}{2}&-\frac{1}{2}&0&0&0&-\frac{1}{2}&0&-\frac{1}{2}\\
\frac{1}{2}&0&0&0&\frac{1}{2}&0&\frac{1}{2}&-\frac{1}{2}&\frac{1}{2}&\frac{1}{2}&-\frac{1}{2}&\frac{1}{2}&0&0&0&-\frac{1}{2}&0&-\frac{1}{2}\\
\frac{1}{2}&0&0&0&\frac{1}{2}&0&-\frac{1}{2}&\frac{1}{2}&-\frac{1}{2}&\frac{1}{2}&\frac{1}{2}&\frac{1}{2}&0&0&0&-\frac{1}{2}&0&-\frac{1}{2}\\
0&\frac{1}{2}&\frac{1}{2}&0&0&0&0&0&0&-\frac{1}{2}&0&\frac{1}{2}&\frac{1}{2}&\frac{1}{2}&-\frac{1}{2}&-\frac{1}{2}&-\frac{1}{2}&\frac{1}{2}\\
0&-\frac{1}{2}&-\frac{1}{2}&0&0&0&0&0&0&\frac{1}{2}&0&-\frac{1}{2}&\frac{1}{2}&\frac{1}{2}&-\frac{1}{2}&\frac{1}{2}&-\frac{1}{2}&-\frac{1}{2}\\
0&\frac{1}{2}&\frac{1}{2}&0&0&0&0&0&0&-\frac{1}{2}&0&\frac{1}{2}&\frac{1}{2}&-\frac{1}{2}&\frac{1}{2}&\frac{1}{2}&-\frac{1}{2}&-\frac{1}{2}\\
0&\frac{1}{2}&\frac{1}{2}&0&0&0&0&0&0&-\frac{1}{2}&0&\frac{1}{2}&-\frac{1}{2}&\frac{1}{2}&-\frac{1}{2}&\frac{1}{2}&\frac{1}{2}&-\frac{1}{2}\\
\frac{1}{2}&\frac{1}{2}&\frac{1}{2}&\frac{1}{2}&-\frac{1}{2}&-\frac{1}{2}&\frac{1}{2}&0&0&0&-\frac{1}{2}&0&0&-\frac{1}{2}&\frac{1}{2}&0&0&0\\
-\frac{1}{2}&-\frac{1}{2}&-\frac{1}{2}&\frac{1}{2}&\frac{1}{2}&-\frac{1}{2}&\frac{1}{2}&0&0&0&-\frac{1}{2}&0&0&-\frac{1}{2}&\frac{1}{2}&0&0&0\\
-\frac{1}{2}&\frac{1}{2}&\frac{1}{2}&\frac{1}{2}&\frac{1}{2}&-\frac{1}{2}&-\frac{1}{2}&0&0&0&\frac{1}{2}&0&0&\frac{1}{2}&-\frac{1}{2}&0&0&0\\
\frac{1}{2}&-\frac{1}{2}&-\frac{1}{2}&\frac{1}{2}&-\frac{1}{2}&-\frac{1}{2}&-\frac{1}{2}&0&0&0&\frac{1}{2}&0&0&\frac{1}{2}&-\frac{1}{2}&0&0&0\\
-\frac{1}{2}&0&0&0&-\frac{1}{2}&0&\frac{1}{2}&\frac{1}{2}&\frac{1}{2}&-\frac{1}{2}&\frac{1}{2}&\frac{1}{2}&0&0&0&\frac{1}{2}&0&\frac{1}{2}\\
-\frac{1}{2}&0&0&0&-\frac{1}{2}&0&-\frac{1}{2}&\frac{1}{2}&\frac{1}{2}&\frac{1}{2}&-\frac{1}{2}&-\frac{1}{2}&0&0&0&\frac{1}{2}&0&\frac{1}{2}\\
\frac{1}{2}&0&0&0&\frac{1}{2}&0&\frac{1}{2}&\frac{1}{2}&\frac{1}{2}&\frac{1}{2}&\frac{1}{2}&-\frac{1}{2}&0&0&0&-\frac{1}{2}&0&-\frac{1}{2}\\
\frac{1}{2}&0&0&0&\frac{1}{2}&0&-\frac{1}{2}&\frac{1}{2}&\frac{1}{2}&-\frac{1}{2}&-\frac{1}{2}&\frac{1}{2}&0&0&0&-\frac{1}{2}&0&-\frac{1}{2}\\
\end{smatrix}}\)}
  \caption{Accurate version of Smirnov's $\FMMA{3}{3}{6}{40}$ matrix
    multiplication.}\label{fig:hm336acc}
\end{sidewaystable*}

\end{document}